\documentclass[11pt]{amsart}
\pdfoutput=1

\usepackage{amssymb}
\usepackage{amsthm}
\usepackage{amsfonts}
\usepackage{amsmath}
\usepackage{pmboxdraw}
\usepackage{verbatim} 
\usepackage{graphicx}
\usepackage{color}
\usepackage[colorlinks=true, citecolor=blue, filecolor=black, linkcolor=black, urlcolor=black]{hyperref}
\usepackage{cite}
\usepackage[normalem]{ulem}
\usepackage{subcaption}
\usepackage{todonotes}
\usepackage{mathtools}

\usepackage{tikz}
\usetikzlibrary{arrows}
\usetikzlibrary{decorations.pathmorphing}
\usetikzlibrary{decorations.markings}
\usetikzlibrary{patterns}
\usetikzlibrary{automata}
\usetikzlibrary{positioning}
\usepackage{tikz-cd}
\tikzset{->-/.style={decoration={
			markings,
			mark=at position #1 with {\arrow{latex}}},postaction={decorate}}}

\tikzset{-<-/.style={decoration={
			markings,
			mark=at position #1 with {\arrowreversed{latex}}},postaction={decorate}}}

\usetikzlibrary{shapes.misc}\tikzset{cross/.style={cross out, draw, 
		minimum size=2*(#1-\pgflinewidth), 
		inner sep=0pt, outer sep=0pt}}

\usepackage{pgfplots}

\definecolor{dullmagenta}{rgb}{0.4,0,0.4}   
\definecolor{darkblue}{rgb}{0,0,0.4}
\hypersetup{linkcolor=black,citecolor=blue,filecolor=dullmagenta,urlcolor=darkblue}

\newcommand{\RN}[1]{%
	\textup{\uppercase\expandafter{\romannumeral#1}}%
}

\def\bp{{\bar\partial}}

\def\pa{\partial}

\def\wh{\widehat}
\def\wt{\widetilde}
\def\ve{\varepsilon}

\def\Re{ \mathrm{Re}}

\def\LL{\mathcal{L}}

\def\SS{\mathcal{S}}

\def\C{\mathbb{C}}
\def\D{\mathbb{D}}

\def\P{\mathbf{P}}
\def\R{\mathbb{R}}

\def\AAA{\mathfrak{A}}
\def\BBB{\mathfrak{B}}

\newcommand{\re}{\operatorname{Re}}

\newcommand{\erfc}{\operatorname{erfc}}

\theoremstyle{plain}
\newtheorem*{thm*}{Theorem}
\newtheorem{thm}{Theorem}[section]
\newtheorem{lem}[thm]{Lemma}
\newtheorem{cor}[thm]{Corollary}
\newtheorem{prop}[thm]{Proposition}
\newtheorem{lem*}[thm]{Lemma}
\newtheorem{rmk}[thm]{Remark}
\newtheorem{rmks*}[thm]{Remarks}
\newtheorem{eg*}[thm]{Example}
\newtheorem{egs*}[thm]{Examples}

\theoremstyle{definition}
\newtheorem*{def*}{Definition}
\newtheorem*{question*}{Question}

\theoremstyle{remark}

\makeatletter
\@namedef{subjclassname@2020}{\textup{2020} Mathematics Subject Classification}
\makeatother

\numberwithin{equation}{section}

\newcommand{\bfR}{\mathbf{R}}
\newcommand{\bfK}{\mathbf{K}}

\begin{document}
\title[Lemniscate ensembles with spectral singularity]{Lemniscate ensembles with spectral singularity}


\author{Sung-Soo Byun}
\address{Center for Mathematical Challenges, Korea Institute for Advanced Study, 85 Hoegiro, Dongdaemun-gu, Seoul 02455, Republic of Korea}
\email{sungsoobyun@kias.re.kr}

\author{Seung-Yeop Lee}
\address{Department of Mathematics and Statistics, University of South Florida, Tampa, FL, USA}
\email{lees3@usf.edu}

\author{Meng Yang}
\address{Department of Mathematical Sciences, University of Copenhagen, Copenhagen, Universitetsparken 5, 2100 København Ø, Denmark}
\email{my@math.ku.dk}

\date{\today}
\thanks{ Sung-Soo Byun was partially supported by Samsung Science and Technology Foundation (SSTF-BA1401-51), by a KIAS Individual Grant (SP083201) via the Center for Mathematical Challenges at Korea Institute for Advanced Study, by the National Research Foundation of Korea (NRF-2019R1A5A1028324), and by the POSCO TJ Park Foundation (POSCO Science Fellowship).
Meng Yang was supported by VILLUM FOUNDEN research grant no. 29369. }

\begin{abstract}
We consider a family of random normal matrix models whose eigenvalues tend to occupy lemniscate type droplets as the size of the matrix increases. 
Under the insertion of a point charge, we derive the scaling limit at the singular boundary point, which is expressed in terms of the solution to the model Painlev\'{e} IV Riemann-Hilbert problem.
For this, we introduce a version of the Christoffel-Darboux identity and combine it with the strong asymptotics of the associated orthogonal polynomials due to Bertola, Elias Rebelo and Grava.  
\end{abstract}

\maketitle

\section{Introduction and main results} \label{Section_Intro}

In the theory of random normal matrices \cite{MR2116267, MR2172690}, one usually starts with a suitable real-valued function $\mathcal{W}$ called the \emph{external potential} and consider a normal matrix of size $N$ picked randomly with respect to the measure proportional to
$e^{ -N \, \textup{Tr} \, \mathcal{W}(M) }\,dM$. 
Here $dM$ is the induced surface measure on the space of normal matrices $\{ M \in \C^{ N^2 } : MM^*=M^*M  \}$. 
Then its eigenvalues $\{\lambda_j\}_{1}^N $ behave like equally charged Coulomb particles \cite{forrester2010log,Serfaty,Lewin22} in the external field $N\mathcal{W}$ at specific inverse temperature $\beta=2$, namely, the joint probability distribution of the system is proportional to
\begin{equation}\label{Coulomb distri}
\prod_{j<k}  |\lambda_j-\lambda_k|^2 \prod_{j=1}^{N}  e^{-N  \mathcal{W}(\lambda_j) } \, dA(\lambda_j), \qquad (dA(\lambda):=d^2\lambda/\pi).
\end{equation}
We refer to \cite[Section 5]{byun2022progress} for a recent review. 

As the size of the matrix increases, the eigenvalue ensemble tends to minimise the weighted logarithmic energy functional \cite{ST97}, which can be recognised as the continuum limit of its discrete Hamiltonian, see e.g. \cite{MR3820329,MR4244340}. 
In particular the support of the limiting empirical distribution is given by a certain compact set called the \emph{droplet}. 
Due to Sakai's regularity theory \cite{MR1097025}, it is well known that for a real analytic potential $\mathcal{W}$, all but finitely many boundary points of the droplet are ``regular'' in a proper sense. 
Furthermore in the case that there exists a local Schwarz function near the prescribed boundary point, the possible types of singularities are classified. 
On the other hand, the construction of a droplet containing singular boundary points requires a separate analysis, see \cite{MR4229527,balogh2015equilibrium,MR3280250,MR2921180,MR3454377,byun2023planar,criado2022vector} and references therein for recent works in this direction. 

The detailed statistical information about the joint intensity functions of the eigenvalue system can be effectively analysed by the \emph{correlation (reproducing) kernel} of the orthogonal polynomials with respect to the weighted Lebesgue measure $e^{-N\mathcal{W}}\,dA$.   
Recently, for a quite general class of the potentials, the asymptotic behaviours of the associated planar orthogonal polynomials were obtained by Hedenmalm and Wennman \cite{hedenmalm2017planar}. 
As a consequence, they derived the boundary scaling limit of the correlation kernel, which leads to the local universality at regular boundary points of the droplet. (We also refer to \cite{ameur2011fluctuations} for an earlier work on the local universality at regular bulk points.)

On the other hand, it is intuitively clear that different kinds of scaling limits should appear at singular boundary points. However the description of such scaling limits remains open in general and we aim to contribute to this problem. 
In particular we shall consider two types of singularities; one is the \emph{lemniscate} type singularity arising from the local geometric structure of the droplet (see Figure~\ref{Fig_LemE}) and the other is the \emph{spectral} singularity arising from an insertion of a point charge. 

\begin{figure}[h!]
    \centering
    	\includegraphics[width=0.8\textwidth]{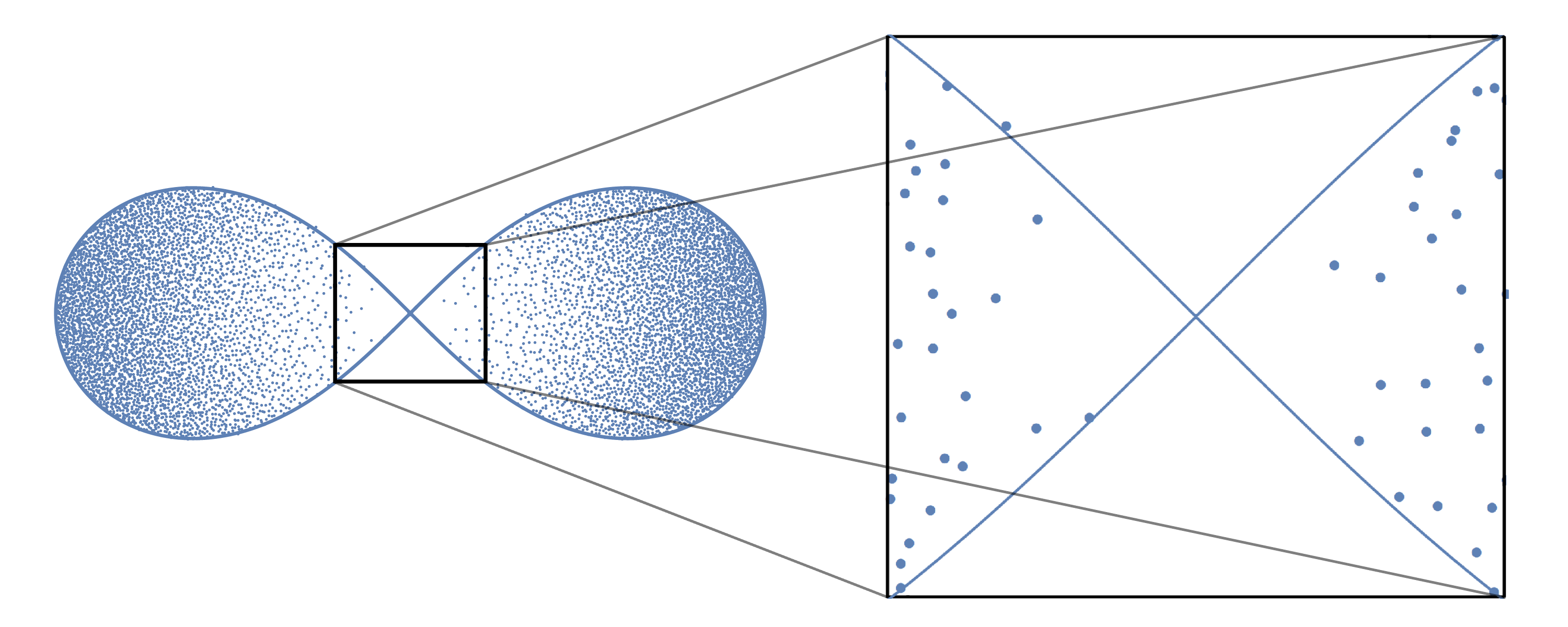}
    \caption{An illustration of a lemniscate ensemble}  
    \label{Fig_LemE}
\end{figure}

\subsection{Setup}

Let us be more precise now in introducing our model that we call the \emph{lemniscate ensemble} following \cite{MR4030288}. 
First we consider the (shifted) Gaussian potential $Q$ of the form 
\begin{equation} \label{Q toy}
	Q(\zeta):=|\zeta-a|^2, \qquad a \ge 0.
\end{equation} 
This is a building block to define
\begin{equation} \label{V_d}
V(\zeta):=\frac1d Q(\zeta^d)=\frac{1}{d}  \, |\zeta^d-a|^2,
\end{equation}
where $d>1$ is a fixed integer. 
We remark that even though $Q$ can be realised as a special case of $V$ with $d=1$, we intentionally distinguish this case for our purpose described below. 
For a given point charge $c >-1$, let
\begin{equation} \label{Qc Vc}
	Q_c(\zeta):=Q(\zeta)-\frac{2c}{N}\log|\zeta|, \qquad V_c(\zeta):=V(\zeta)-\frac{2c}{N}\log|\zeta|.
\end{equation}
Such an extra logarithmic factor is often referred to as a spectral singularity, see e.g. \cite[Chapter 6]{akemann2011oxford}.  
Here the condition $c >-1$ is required to guarantee that the partition functions $Z_N,\wh{Z}_N$ below are finite. 

We shall study random normal matrix ensembles $\{ \zeta_j \}_{1}^N$, $\{ \wh{\zeta}_j \}_{1}^N$ associated with the potentials $V_c$, $Q_c$ respectively. By definition, their joint probability distributions $\P_N,\wh{\P}_N$ are given by
\begin{align}
	d\P_N(\zeta_1,\ldots,\zeta_N)&=\frac{1}{Z_N} \prod_{j<k}  |\zeta_j-\zeta_k|^2 \prod_{j=1}^{N} |\zeta_j|^{2c} e^{-N  V(\zeta_j) } \, dA(\zeta_j),
	\\
	d\wh{\P}_N(\wh{\zeta}_1,\ldots,\wh{\zeta}_N)&=\frac{1}{\wh{Z}_N} \prod_{j<k}  |\wh{\zeta}_j-\wh{\zeta}_k|^2 \prod_{j=1}^{N} |\wh{\zeta}_j|^{2c} e^{-N  Q(\wh{\zeta}_j) } \, dA(\wh{\zeta}_j),
\end{align}
where $Z_N,\wh{Z}_N$ are normalisation constants which turn $\P_N,\wh{\P}_N$ into probability measures, see \cite{deano2019characteristic,MR3946715} for asymptotics of the partition function $\widehat{Z}_N$. 
We also mention that for an integer-valued $c$, the system $\{ \wh{\zeta}_j \}_{1}^N$ has an alternative realisation as eigenvalues of the \emph{induced Ginibre ensemble} \cite{MR2881072}, an extension of the Ginibre ensemble to include zero eigenvalues.


The well-known circular law \cite{ginibre1965statistical} asserts that as $N$ increases, the eigenvalues $\{ \wh{\zeta}_j \}_{1}^N$ tend to be uniformly distributed on the disc $\wh{S}:=\{ \zeta \in \C: |\zeta-a|^2 \le 1  \}$. 
As a consequence, it is easy to observe that $\{ \zeta_j \}_{1}^N$ tend to occupy the droplet 
\begin{equation}
S:=\{ \zeta \in \C: |\zeta^d-a|^2 \le 1 \}
\end{equation}
and that the limiting density on $S$ with respect to the area measure $dA$ is given by 
\begin{equation} \label{Delta V}
	\Delta V(\zeta)= d \, |\zeta|^{2d-2},  \qquad (\Delta:=\pa \bp),
\end{equation}
see \cite[Lemma 1]{balogh2015equilibrium}. Note that the topology of $S$ reveals a phase transition at the value $a=1$, where the droplet $S$ is of lemniscate type having $d$-fold symmetry, see Figure~\ref{Fig_SVd}. 

	\begin{figure}[h!]
	\begin{center}
		\begin{subfigure}[h]{0.32\textwidth}
		\centering
			\includegraphics[width=0.8\textwidth]{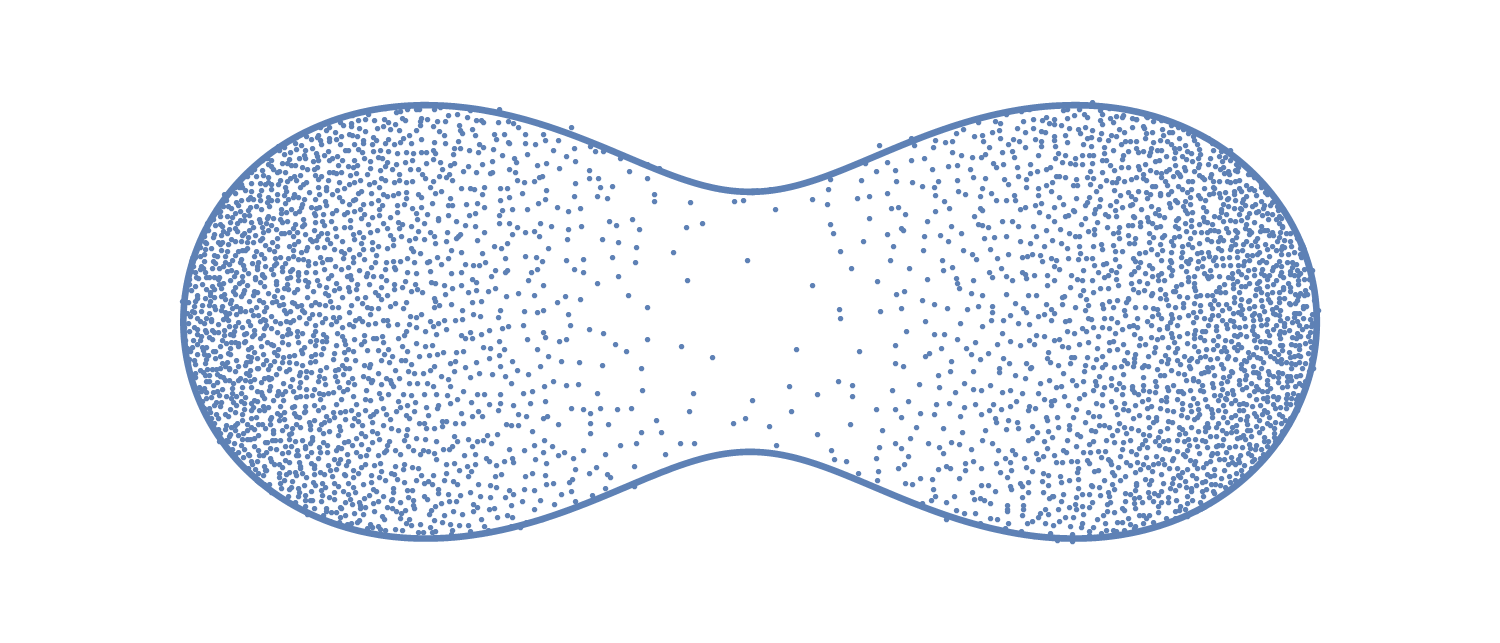}
			\caption{$d=2, a=0.9$}
		\end{subfigure} 
		\begin{subfigure}[h]{0.32\textwidth}
		\centering
			\includegraphics[width=0.8\textwidth]{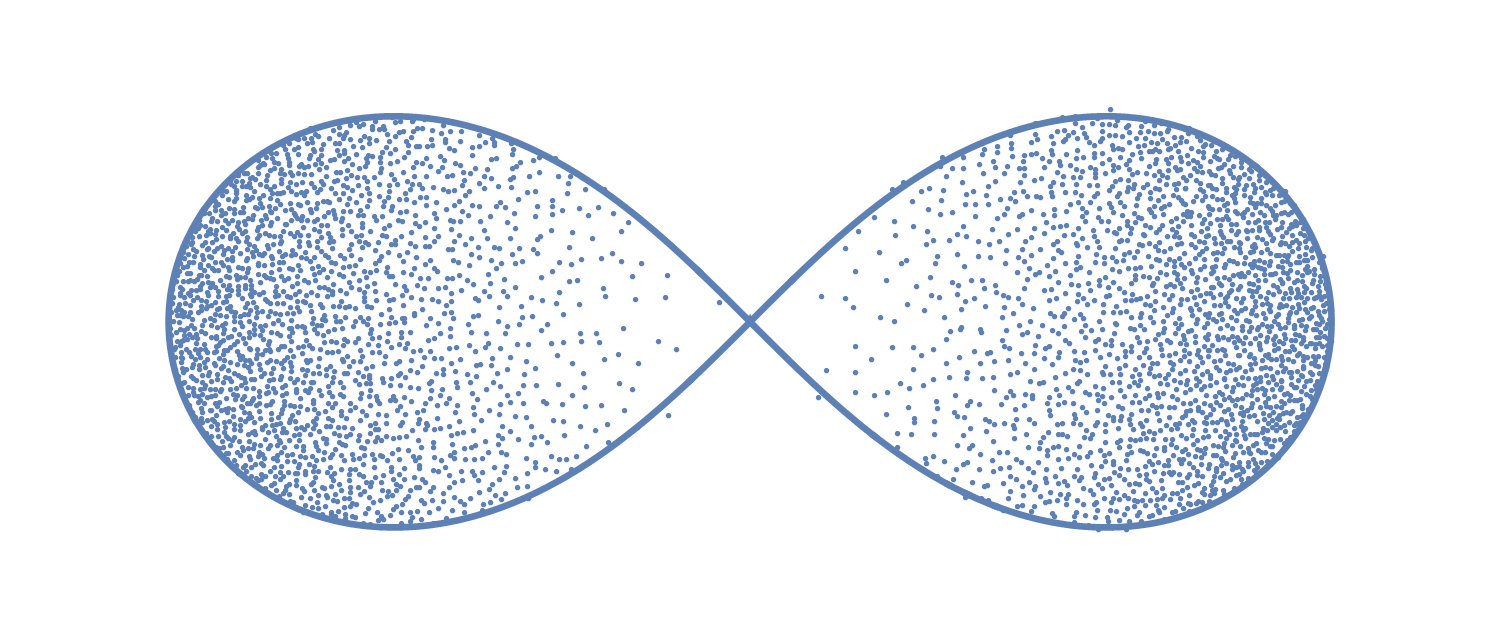}
			\caption{$d=2, a=1$}
		\end{subfigure} 
		\begin{subfigure}[h]{0.32\textwidth}
		\centering
			\includegraphics[width=0.8\textwidth]{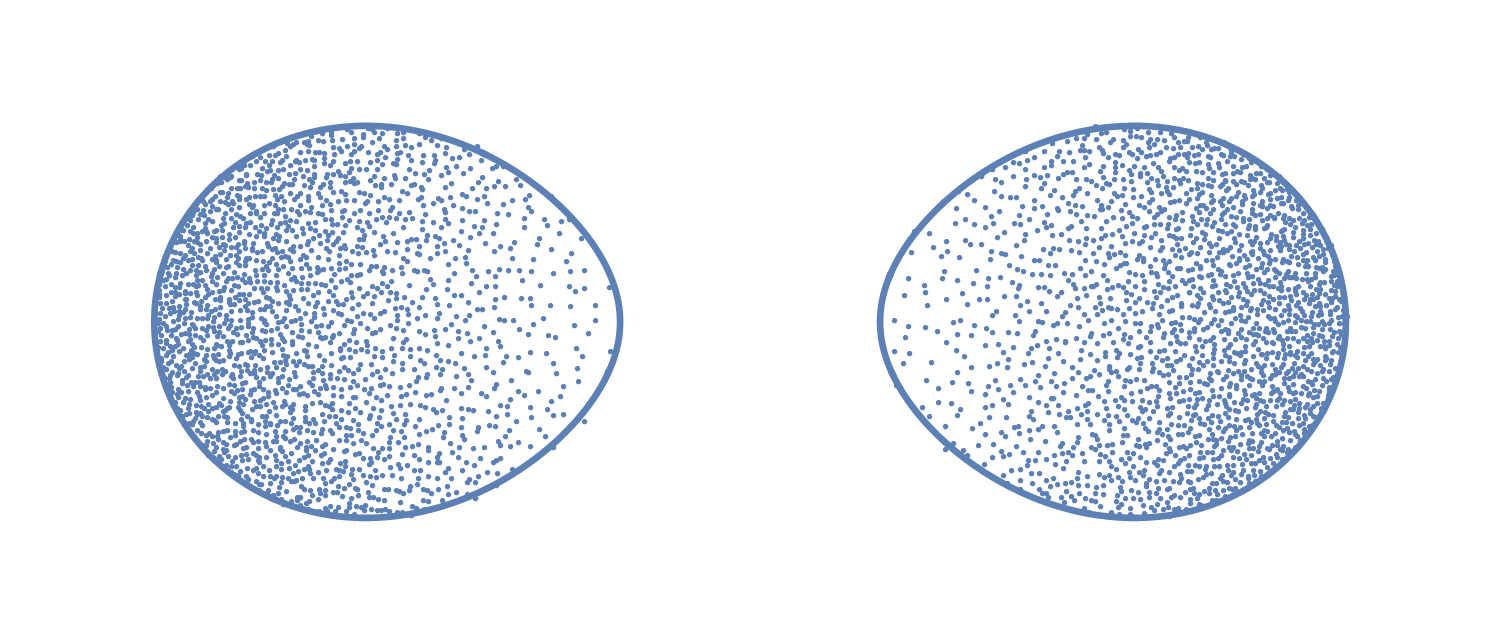}
			\caption{$d=2, a=1.1$}
		\end{subfigure} 
		
			\begin{subfigure}[h]{0.32\textwidth}
			\centering
			\includegraphics[width=0.8\textwidth]{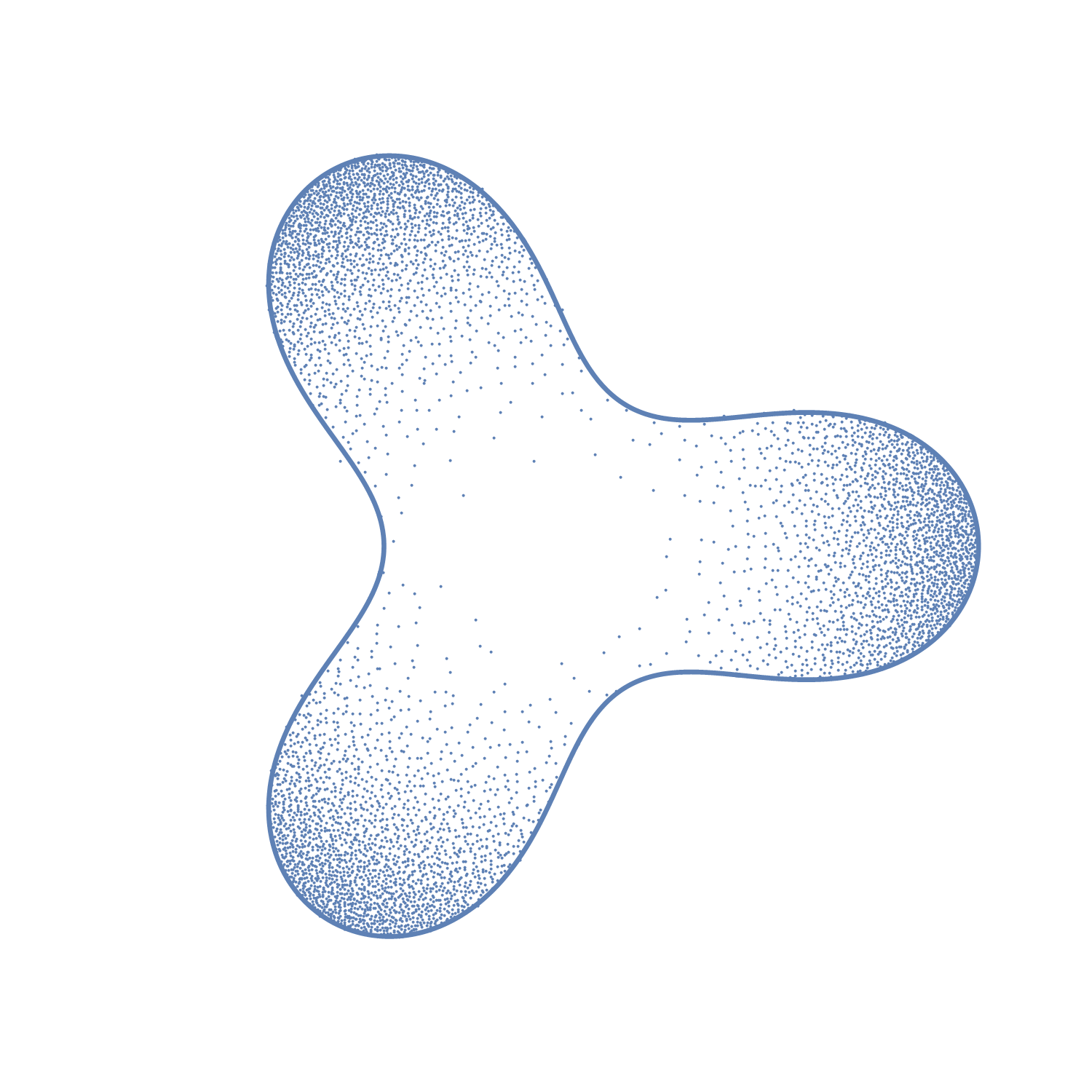}
			\caption{$d=3, a=0.9$}
		\end{subfigure} 
		\begin{subfigure}[h]{0.32\textwidth}
		\centering
			\includegraphics[width=0.8\textwidth]{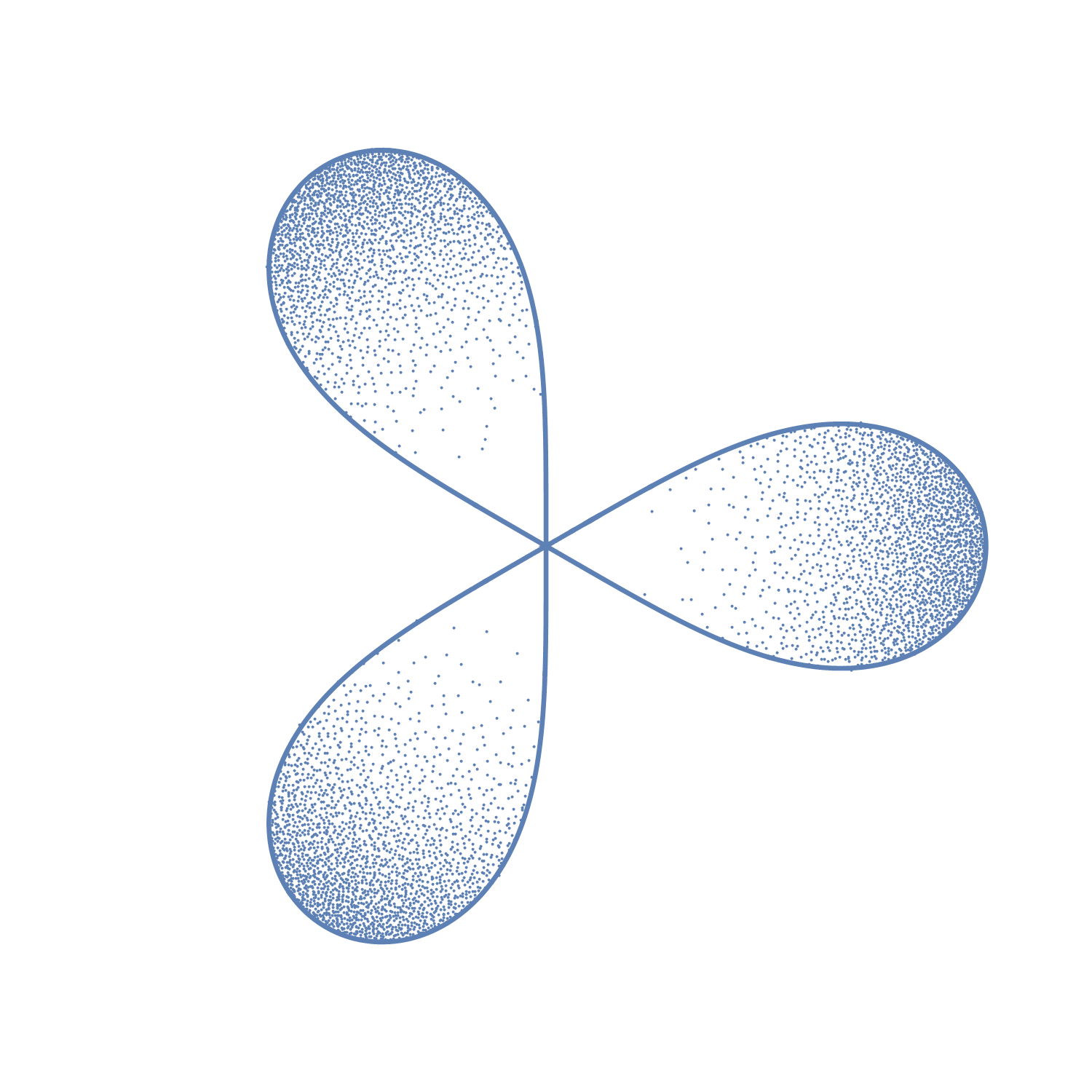}
			\caption{$d=3, a=1$}
		\end{subfigure} 
		\begin{subfigure}[h]{0.32\textwidth}
		\centering
			\includegraphics[width=0.8\textwidth]{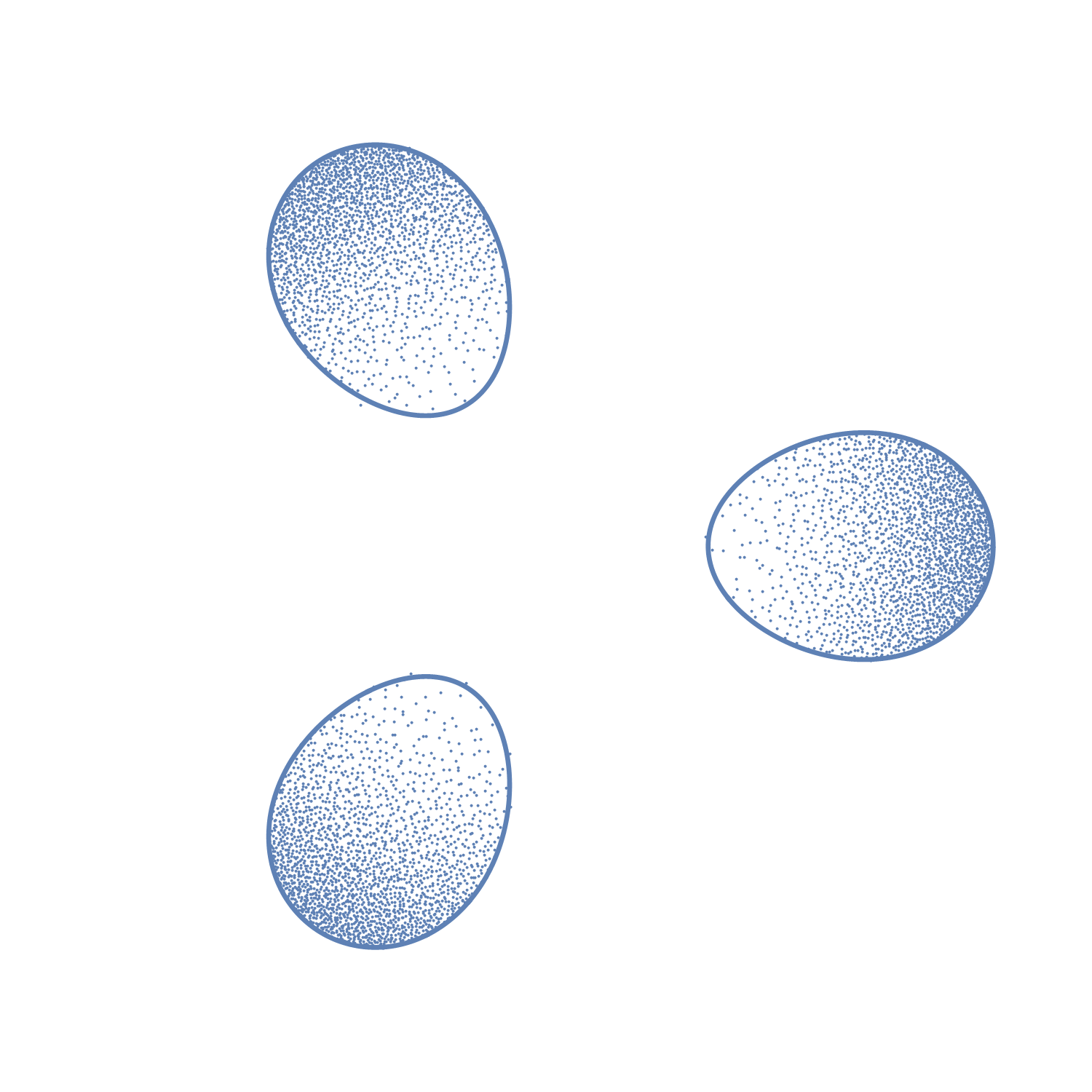}
			\caption{$d=3, a=1.1$}
		\end{subfigure} 
	\end{center}
	\caption{Plots of $\pa S$ for a few values of $d$ and $a$} \label{Fig_SVd}
\end{figure}

We denote by $p_{j,N}^c$, $q_{j,N}^c$ the orthonormal polynomials of degree $j$ with respect to the weighted measure $e^{-NQ_c}\,dA$, $e^{-NV_c}\,dA$, respectively, i.e. 
\begin{align}
\int_{ \C }  p_{j,N}^{c}(\zeta) \overline{p_{k,N}^c (\zeta) } |\zeta|^{2c} e^{-N Q(\zeta)}\,dA(\zeta)=\int_{ \C }  q_{j,N}^c(\zeta) \overline{ q_{k,N}^c (\zeta) } |\zeta|^{2c} e^{-N V(\zeta)}\,dA(\zeta)=\delta_{jk}.
\end{align}
Here $\delta_{jk}$ is the Kronecker delta. 
The strong asymptotics of $p_{j,N}^c$ were extensively studied in \cite{MR3670735,MR3280250,MR3849128,MR3668632}, see also recent works \cite{lee2020strong,MR3962350,berezin2022planar} on the case with multiple point charges. We also refer the reader to \cite{MR2921180,MR3939592,MR3303173,MR2078083,MR3306308,MR3383811} for the strong asymptotics of planar orthogonal polynomials associated with some other classes of potentials. 

Let us write $\bfK_N^c$, $\widehat{\bfK}_N^c$ for the correlation kernels of the point processes $\{ \zeta_j \}_{1}^N$, $\{ \wh{\zeta}_j \}_{1}^N$, respectively. Due to Dyson's determinantal formula, we have the canonical expressions
\begin{align}
	\bfK_{N}^c(\zeta,\eta)&=(\zeta \bar{\eta})^ce^{ -\frac{N}{2} (V(\zeta)+V(\eta))  } \sum_{j=0}^{N-1} q_{j,N}^c(\zeta) \overline{q_{j,N}^c(\eta)}, \label{bfKN}
	\\
	\widehat{\bfK}_{N}^c(\zeta,\eta)&=(\zeta \bar{\eta})^ce^{ -\frac{N}{2} (Q(\zeta)+ Q(\eta)) } \sum_{j=0}^{N-1} p_{j,N}^c(\zeta) \overline{p_{j,N}^c(\eta)}. \label{bfwhKN}
\end{align}
The joint intensity (correlation) functions are then given in terms of the determinant of such correlation kernels, see e.g. \cite{forrester2010log,byun2022progress}.

To describe the local statistics of $\{ \zeta_j \}_{1}^N$, $\{ \wh{\zeta}_j \}_{1}^N$ at the origin, it is convenient to define the rescaled point processes $\{ z_j \}_{1}^N$, $ \{ \wh{z}_j \}_{1}^N$ as 
\begin{equation}
z_j:=(N/d)^{\frac{1}{2d}} \cdot \zeta_j, \qquad \wh{z}_j:=N^{\frac12} \cdot \wh{\zeta}_j,
\end{equation}
see Figure~\ref{Fig_LemE}. 
Here the rescaling order $N^{\frac{1}{2d}}$ is chosen according to the mean eigenvalue density \eqref{Delta V} at the origin. 
By definition, the correlation kernels $K_N^c, \widehat{K}_N^c$ associated with the point processes $\{ z_j \}_{1}^N,\{ \wh{z}_j \}_{1}^N$ are given by
\begin{align} \label{KN KN hat def}
K_N^c(z,w):=\frac{1}{(N/d)^{\frac1d}} \bfK_N^c\Big( \frac{z}{(N/d)^{\frac{1}{2d}}}, \frac{w}{(N/d)^{\frac{1}{2d}}} \Big), \qquad 	\widehat{K}_N^c(z,w):=\frac{1}{N} \widehat{\bfK}_N^c\Big( \frac{z}{\sqrt{N}}, \frac{w}{\sqrt{N}} \Big).
\end{align}

We aim to derive the large-$N$ limits 
\begin{equation}
K^c:=\lim\limits_{N\to\infty} K_N^c, \qquad \widehat{K}^c:=\lim\limits_{N\to\infty} \widehat{K}_N^c
\end{equation}
of the correlation kernels, where the convergence is uniform on compact subsets of $\C$. The existence of the large-$N$ limits can be found in \cite[Theorem 1.1]{ameur2018random} and \cite[Lemma 3]{MR4030288}. Let us also stress here that by \cite[Lemma 1]{MR4030288}, the limiting point processes are indeed determined by their $1$-point densities 
\begin{equation}
R^c(z):=K^c(z,z), \qquad \wh{R}^c(z):=\wh{K}^c(z,z).
\end{equation}

\subsection{Finite-$N$ analysis}

The main ingredient to analyse the correlation kernel is a version of the \emph{Christoffel-Darboux identity}. 
This can be applied to various situations for instance to the case studied in \cite{MR3280250}; cf. see \cite{byun2022determinantal} for a recent implementation.

To describe the Christoffel-Darboux formula, let $P_j$ be the \emph{monic} orthogonal polynomial of degree $j$ satisfying 
\begin{align} \label{Pj monic}
	\int_{ \C }  P_j(\zeta) \overline{ P_k (\zeta) } |\zeta-a|^{2c} e^{-N |\zeta|^2}\,dA(\zeta)= h_j \, \delta_{jk},
\end{align}
where $h_j$ is the orthogonal norm. 
We denote
\begin{equation}
\psi_j(\zeta):= (\zeta-a)^c P_j(\zeta), \qquad \phi_j (\zeta):=(\zeta-a)^c \frac{P_j(\zeta)}{h_j}.  
\end{equation}
and define 
\begin{equation} \label{def of pre kernel}
	\wt{\bfK}_{N}^c(\zeta,\eta):=e^{-N \zeta \bar{\eta} } \sum_{j=0}^{n-1}  \overline{ \phi_j(\eta) } \psi_j(\zeta)= |\zeta-a|^{2c} e^{-N \zeta \bar{\eta} } \sum_{j=0}^{N-1} P_j(\zeta) \overline{ P_j (\eta) }. 
\end{equation}
Note that it is related to $\wh{\bfK}_{N}^c$ in \eqref{bfwhKN} as
\begin{equation} \label{bfwhKN bfwtKN}
\wh{\bfK}_{N}^c(\zeta,\zeta)=\wt{\bfK}_{N}^c(a-\zeta,a-\zeta).
\end{equation}
We obtain the following theorem.

\begin{thm}[\textbf{Christoffel-Darboux formula}] \label{Thm_CDI} 
Suppose that $a \not=0$. 
Then we have 
\begin{align} \label{CDI_v3}
	\begin{split}
		\bp_\eta \wt{\bfK}_{n}^c(\zeta,\eta)
		&=e^{ -N \zeta \bar{\eta} }  \frac{1}{ \tfrac{n+c}{N}h_{n-1}-h_{n}  }	\bp_\eta \overline{ \psi_{n}(\eta) }    \Big( \psi_{n}(\zeta)-\zeta \psi_{n-1}(\zeta) \Big)
		\\
		&\quad -e^{ -N \zeta \bar{\eta} } \frac{P_{n+1}(a)}{P_n(a)} \frac{N\,h_N/h_{N-1} }{ \tfrac{n+c+1}{N} h_n-h_{n+1}   } \overline{ \psi_{n-1}(\eta) }  \Big(  \psi_{n+1}(\zeta)-\zeta \psi_n(\zeta)  \Big).
	\end{split}
\end{align}
\end{thm}

Contrary to the classical Christoffel-Darboux formula for the orthogonal polynomial kernel on the real axis, the identity \eqref{CDI_v3} shows that the summation in \eqref{def of pre kernel} can be expressed in terms of the three last orthogonal polynomials.

\begin{rmk} \label{Rmk_a not0}
For the radially symmetric case when $a=0$, we have
\begin{equation}
	P_j(\zeta)=\zeta^j, \qquad 	h_j= \frac{\Gamma(j+c+1)}{ N^{j+c+1} }.
\end{equation}
Thus Theorem~\ref{Thm_CDI} cannot be directly applied to this case since 
\begin{equation}
\tfrac{N+c}{N}h_{N-1}-h_{N}=\psi_{N}(\zeta)-\zeta \psi_{N-1}(\zeta)=0.
\end{equation}
On the one hand, for $a \not= 0$, it was shown in \cite[Appendix D]{MR3280250} that $\tfrac{N+c}{N}h_{N-1}-h_{N}$ does not vanish. 
\end{rmk}

\begin{rmk}[Three-term recurrence relation]
In Subsection~\ref{Subsec_CDI IG}, we also show that the orthogonal polynomial $P_k$ satisfies the (non-standard) three-term recurrence relation of the form
\begin{equation}  \label{3 term recurrence}
z \, P_k(z)= P_{k+1}(z) +b_k \,P_k(z)+c_k \, z \, P_{k-1}(z),
\end{equation}
where 
\begin{equation}
b_k:=- \frac{P_{k+1}(0)}{ P_k(0) }, \qquad  c_k:= \frac{ P_k(0)-P_{k+1}'(0)-b_k P_k'(0) }{ P_{k-1}(0)  }.
\end{equation}
This relation \eqref{3 term recurrence} plays an important role in the proof of Theorem~\ref{Thm_CDI}. 
We mention that it is shown in \cite[Corollary 5.3]{akemann2021skew} that $P_k$ does not satisfy the standard three-term recurrence relation, (i.e. the relation of the form \eqref{3 term recurrence} with $c_k z$ replaced by $c_k$). 
\end{rmk}

\begin{eg*} \emph{(Exactly solvable case: $c=1$)}
For an integer-valued point charge $c$, one can explicitly express the associated orthogonal polynomials using the well-known special functions. For instance, when $c=1$, we have
\begin{align} \label{Pk c1}
	\begin{split}
P_k(\zeta)&=\sum_{j=0}^{k} a^{k-j} \frac{k!}{j!} \frac{\Gamma(j+1,Na^2)}{\Gamma(k+1,Na^2)} \zeta^j
= \frac{1}{\zeta-a} \Big(  \zeta^{k+1}-e^{ aN(\zeta-a) } \frac{Q(k+1,Na\,\zeta)}{ Q(k+1,Na^2) } a^{k+1}  \Big)
	\end{split}
\end{align}
and 
\begin{equation} \label{hk c1}
	h_k	= \frac{(k+1)!}{N^{k+2}} \frac{Q(k+2,N a^2)}{ Q(k+1,Na^2) },
\end{equation}
see  \cite[Section 3]{MR1982915}.
Here $Q$ is the regularised incomplete Gamma function. 
Then by using some basic properties of the incomplete Gamma function, one can directly check the Christoffel-Darboux formula \eqref{CDI_v3} as well as the three-term recurrence relation \eqref{3 term recurrence}. 
\end{eg*}

Due to the relation \eqref{bfwhKN bfwtKN}, one can notice that the use of Theorem~\ref{Thm_CDI} can be made to derive the asymptotic behaviours of $\widehat{K}_N^c$ in \eqref{KN KN hat def}.
Furthermore, the behaviours of $K_N^c$ follows from the following proposition. 

\begin{prop}[\textbf{Multi-fold transform}]\label{Prop_multitrans}
	For each $c>-1$ and $d \in \mathbb{N}$, we have 
 \begin{align}
  \bfK_{dN}^c(\zeta,\eta)& =d(\zeta \bar{\eta})^{d-1} \sum_{l=0}^{d-1} \widehat{\bfK}_N^{ \frac{c+l+1}{d}-1 }(\zeta^d,\eta^d),    \label{bf K whK N rel}
  \\
  \label{K whK N rel}
		K_{dN}^c(z,w) &= d(z\bar{w})^{d-1} \sum_{l=0}^{d-1} \widehat{K}_N^{ \frac{c+l+1}{d}-1  } (z^d,w^d).
 \end{align}
\end{prop}

We refer to \cite[Proposition 2.1]{claeys2008universality} and \cite[Appendix B]{akemann2001qcd3} for related statements on Hermitian matrix models.
This proposition follows from a simple relation \eqref{ONP transform} between the orthogonal polynomials. 

\subsection{Scaling limits}

We now focus on the critical regime when $a \to 1$ in a way that the scaled parameter 
\begin{equation} \label{a critical}
\mathcal{S}:=2\sqrt{N}(a-1)
\end{equation}
remains bounded. 
An analogue of such regime in the Hermitian random matrix theory is  called \emph{multi-criticality} \cite{MR1949138,MR2434886}, see also \cite{MR4229527} for the chiral counterpart.

We first introduce the Riemann-Hilbert problem for $\widetilde\Psi^c$ that describes the special solution of the Painlev\'{e} \RN{4} appeared in \cite{MR3849128}. 
See \cite[Subsection 2.2]{MR3849128} for more details. 
For a given real parameter $s$, the matrix $\widetilde\Psi^c(\zeta;s)$ of size $2$ is analytic in $\C \backslash (\Gamma_1 \cup \Gamma_\infty \cup \R_-)$ and admits non-tangential boundary values. Here $\Gamma_\infty=i \R$ and $\Gamma_1$ is a contour in the left-half plane crossing the origin as shown in  Figure~\ref{Fig_jump0}.  One may simply assume that  $\Gamma_1$ comes from the infinity straight to the origin in an angle between $\pi$ and $3\pi/2$ and going straight back to the infinity in an angle between $\pi/2$ and $\pi$. 

\begin{figure}[h!]  
\centering
\begin{tikzpicture}[scale=0.85]
\draw[thick,postaction={decorate, decoration={markings, mark = at position 0.5 with {\arrow{>}}}} ] (-6,0) -- (0,0);
\draw[thick,postaction={decorate, decoration={markings, mark = at position 0.5 with {\arrow{>}}}} ] (0,0) -- (0,3);
\draw[thick,postaction={decorate, decoration={markings, mark = at position 0.5 with {\arrow{>}}}} ] (0,-3) -- (0,0);

	\draw[thick, postaction={decorate, decoration={markings, mark = at position 0.5 with {\arrow{>}}}} ]
				(0,0)[out=150, in=5] to
				(-6,1.5);

	\draw[thick, postaction={decorate, decoration={markings, mark = at position 0.5 with {\arrow{>}}}} ]
				(-6,-1.5)[out=-5, in=-150] to
				(0,0);

\foreach \Point/\PointLabel in {(0,0)/0}
\draw[fill=black] \Point circle (0.05) node[below right] {$\PointLabel$};

\foreach \Point/\PointLabel in {(-1,3)/\Omega_{2},(-5,0.8)/\Omega_{0}, (0.6,0.6)/\Omega_{\infty},(0.4,-2)/\Gamma_{\infty},(-5.5,-1.7)/\Gamma_1}
\draw[fill=black]  
 \Point node[below right] {$\PointLabel$};
 \end{tikzpicture}
 \caption{The jump contours of $\widetilde\Psi^c(\zeta;s)$. } \label{Fig_jump0} 
 \end{figure}
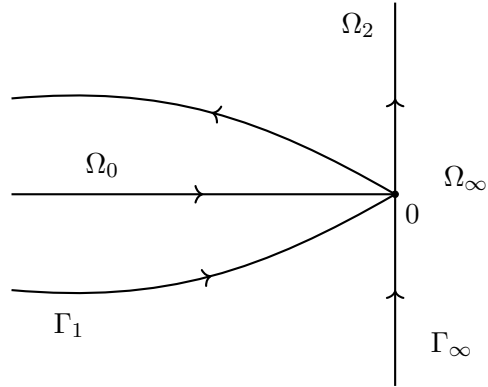
The jump conditions and the asymptotic behaviours of $\widetilde\Psi^c$ are given as follows:
\begin{itemize}
	\item The jump condition is given by
	\begin{equation} \label{RHmodel jump}
	\widetilde\Psi^c_{+}(\zeta;s)=\widetilde\Psi^c_{-}(\zeta;s) 
	\begin{cases}
		\begin{pmatrix}
			1 & -1
			\\
			0 & 1
		\end{pmatrix}, & \zeta \in \Gamma_1,
		\smallskip 
		\\
		\begin{pmatrix}
			1 & 0
			\\
			1 & 1
		\end{pmatrix}, & \zeta \in \Gamma_\infty,
		\smallskip 
		\\
		e^{-c \pi i \sigma_3} & \zeta \in \R_- ,
	\end{cases}   
	\end{equation}
	where 
	$\sigma_3$ is the third Pauli matrix. Here $\widetilde\Psi_{\pm}^c(\zeta,s)$ are continuous boundary values on the left and right of the jumping contours;
	\item As $\zeta \to \infty$, 
	$$\widetilde\Psi^c(\zeta;s)= \bigg( I+\frac{\Psi_1(s)}{\zeta}+\frac{\Psi_2(s)}{\zeta^2}+O( \tfrac{1}{\zeta^3}) \bigg)\zeta^{ -\frac{c}{2} \sigma_3 }\, e^{-(\frac{\zeta^2}{4}+\frac{s}{2}\zeta) \sigma_3}. $$
 Here
$$
\begin{aligned}
\Psi_1(s)&=\begin{pmatrix}
			H(s) & \frac{Z(s)}{U(s)}
			\\
			U(s) & -H(s)
		\end{pmatrix},\\
\Psi_2(s)&=\begin{pmatrix}
		\frac{1}{2}(H(s)^2+Z(z)-sH(s)) & \frac{Z(s)(Z(s)+c-Y(s)s-H(s)Y(s))}{U(s)Y(s)}
			\\
			U(s)(H(s)+Y(s)-s) &\frac{1}{2}(H(s)^2+Z(z)+sH(s))
		\end{pmatrix}.		
\end{aligned}
$$
	\item As $\zeta \to 0$ in the region $\Omega_\infty$, 
	\begin{equation}
	 \widetilde\Psi^c(\zeta;s)= \zeta^{  -\frac{c}{2} \sigma_3 } \cdot O(1). 
	\end{equation}
\end{itemize}

By \cite{MR3849128} and the references therein, the unique solution to the above Riemann-Hilbert problem is related to the Painlev\'e \RN{4} equation by the following lax pair,
\begin{equation}\label{differ 1}
    \frac{d}{d \zeta}\widetilde\Psi^c=A\widetilde\Psi^c, \qquad  \frac{d}{d s}\widetilde\Psi^c=B\widetilde\Psi^c, 
\end{equation}
where
\begin{equation}
A={A}(\zeta;s)=-\frac{\zeta+s}{2}\sigma_3+\begin{bmatrix}0&\frac{Z}{U}\\
-U&0\end{bmatrix} +\frac{1}{\zeta} \begin{bmatrix}-\frac{c}{2}-Z&\frac{Z^2+cZ}{YU}\\
-UY&Z+\frac{c}{2}\end{bmatrix} 
\end{equation} and 
\begin{equation}
{B}={B}(\zeta;s)=-\frac{\zeta}{2}\sigma_3+\begin{bmatrix}0&\frac{Z}{U}\\
-U&0\end{bmatrix}.    \end{equation}
The compatibility condition of the linear system \eqref{differ 1} gives $A_s-B_{\zeta}+[A,B]=0.$ 
It follows that     
   \begin{equation} \label{U' Z' Y'}
   U'=U(Y-s), \qquad Z'=ZY-\frac{Z^2+c Z}{Y},\qquad Y'=s Y-Y^2-2Z-c.    
   \end{equation}
Using the above differential equations, one can observe that $Y$ satisfies the Painlev\'{e} \RN{4} equation 
\begin{equation} \label{def of Y}
Y''=\frac12 \frac{ (Y')^2 }{Y}+\frac32 Y^3-2s Y^2+\Big( 1+\frac{s^2}{2}+c \Big) Y-\frac{c^2}{2Y}.
\end{equation}
The functions $Y(s),Z(s),$ and $H(s)$ are interrelated through
\begin{equation} \label{def of U Z H}
Z=\frac12(sY-c-Y'-Y^2), \qquad H=\Big(s-\frac{c}{Y}-Y\Big) Z-\frac{Z^2}{Y}.
\end{equation} 
We also write 
\begin{equation}\label{def of W}
W:=Z/U.
\end{equation}
In terms of the function $W$, we define  
\begin{align} \label{def of AAA}
\AAA \equiv \AAA(s) &: = \Big[ \frac{W''(s)}{W(s)} -\Big( \frac{W'(s)}{W(s)} \Big)^2 \Big]^{-1} 
\\
\label{def of BBB}
\BBB \equiv \BBB(s)
&=2 \frac{W'(s)}{W(s)\AAA(s)} +\frac{W'(s)}{W(s)}\Big(\frac{W''(s) }{W(s)}-\frac{W'''(s)}{W'(s)}\Big),
\end{align}
cf. \eqref{W' W''} and \eqref{W'''}. 

We write $\wh{R}^c_{\textup{edge}}$, $R^c_{\textup{edge}}$ for the associated limiting $1$-point functions when $a$ is given by \eqref{a critical}. Let us also denote by $\wh{K}^c_{\textup{edge}}$, $K^c_{\textup{edge}}$ the corresponding correlation kernels.
Let us define an analytic continuation, $\Psi^c(\zeta;s)$, of the matrix function $\widetilde\Psi^c(\zeta;s)$ by

\begin{equation}\label{def psic}
	\Psi^c(\zeta;s):=
	\begin{cases}
		\widetilde\Psi^c(\zeta;s), & \zeta \in \Omega_\infty,
		\smallskip 
		\\
		\widetilde\Psi^c(\zeta;s) \begin{pmatrix}
			1 & 0
			\\
			1 & 1
		\end{pmatrix}, & \zeta \in \Omega_2,
		\smallskip 
		\\
		\widetilde\Psi^c(\zeta;s) \begin{pmatrix}
			1 & -1
			\\
			1 & 0
		\end{pmatrix}, & \zeta \in \Omega_0,
	\end{cases}   
	\end{equation}
where the regions $\Omega_0$, $\Omega_2$ and $\Omega_\infty$ are specified in Figure \ref{Fig_jump0}.

\begin{thm}\label{Thm_bInsertion} 
Let $\Psi_{jk}^c$ be the $(j,k)$ entry of $\Psi^c$ defined in \eqref{def psic}.
\begin{itemize}
    \item \textup{(\textbf{Induced Ginibre ensemble with a point charge at a boundary point})} For each $c \in (-1,0)$, we have 
\begin{align}\label{eq rhat}
\begin{split}
 \wh{R}_{\textup{edge}}^c(z) & =  \frac{ \AAA(\SS) \,(-1)^{c}  }{ \sqrt{2\pi} } \,e^{ -\frac{ (z-\SS)^2 }{4} } \,z^{ \frac{c}{2} }  \,\Psi^c_{21}(-z;\SS) \,W(\SS)
\\
&\quad \times \int_{-\infty}^{\bar{z}}   e^{ -z w-\frac{ (w-\SS)^2 }{4} } w^{ \frac{c}{2} }   \Big[   \Psi^c_{11}(-w;\SS) \mathcal{X} + \Psi^c_{21}(-w;\SS) \mathcal{Y}    +Z(\SS) \Big] \,dw, 
\end{split}
\end{align}
where 
\begin{equation}
 \mathcal{X} \equiv \mathcal{X}(\SS) = \frac{Z}{w}-z+\AAA \BBB-  \frac{Z+c}{Y} + Y, \qquad   \mathcal{Y} \equiv \mathcal{Y}(\SS) = -\frac{W}{w}\frac{Z+c}{Y} .
\end{equation}
\item \textup{(\textbf{Recursive formula})} For each $c>-1,$ we have
	\begin{equation} \label{recursion_1 pt}
		\wh{R}^{c+1}_{\textup{edge}}(z)=\wh{R}^{c}_{\textup{edge}}(z)-		 \frac{|\wh{K}^c_{\textup{edge}}(0,z)|^2}{\wh{R}^c_{\textup{edge}}(0)}.
	\end{equation}
\end{itemize}
\end{thm}

\begin{rmk}
It is known \cite[Theorem 2.5]{MR3849128} that the solution $\widetilde\Psi^c(\zeta;\SS)$ exists for $\SS<-0.7701449782$. Also it is expected that there are discrete values of $\SS$ where the solution does not exist.  Hence the above theorem makes sense for all $\SS$ except those discrete values. The asymptotic behaviours of the $1$-point function in \eqref{eq rhat} will be discussed in Appendix~\ref{appendix_asymptotic 1pt}.
\end{rmk}

As an immediate consequence of Proposition~\ref{Prop_multitrans} and Theorem~\ref{Thm_bInsertion}, we obtain the scaling limit of the lemniscate ensembles. 

\begin{thm}[\textbf{\textup{Lemniscate ensemble with $d$-fold symmetry}}]
Under the same assumptions of Theorem~\ref{Thm_bInsertion}, we have that for each $d>1$ and $c>-1$, 
\begin{align}
\begin{split} \label{lemniscate 1pt function}
R^c_{\textup{edge}}(z) 
&= d\, |z|^{2d-2} \sum_{l=0}^{d-1}  \frac{ \AAA(\SS) \,(-1)^{ \frac{c-l}{d} }  }{ \sqrt{2\pi} } \,e^{ -\frac{ (z^d-\SS)^2 }{4} } \,z^{ \frac{c-l}{2} }  \,\Psi_{21}^{ \frac{c-l}{d}  }(-z^d;\SS) \,W(\SS)
\\
&\quad \times \int_{-\infty}^{\bar{z}^d}   e^{ -z^d w-\frac{ (w-\SS)^2 }{4} } w^{  \frac{c-l}{2d} }   \Big[   \Psi_{11}^{ \frac{c-l}{d}  }(-w;\SS) \mathcal{X} + \Psi_{21}^{ \frac{c-l}{d}  }(-w;\SS) \mathcal{Y}    +Z(\SS) \Big] \,dw. 
\end{split}
\end{align}
\end{thm}

\begin{rmk} \label{Rmk_bulk}
For a fixed $a \in [0,1)$, thus when the origin is inside of the droplet $S$ (see Figure~\ref{Fig_SVd} (A) and (D)), it was shown in \cite{ameur2018random} that the limiting $1$-point functions $\wh{R}^c_{\textup{bulk}},R^c_{\textup{bulk}}$ are given by
\begin{equation} \label{K ML da}
	\wh{R}^c_{\textup{bulk}}(z)=  |z|^{2c} e^{-|z|^2}E_{1,1+c} (|z|^2),\qquad R^c_{\textup{bulk}}(z)=d |z|^{2c}e^{ -|z|^{2d} } E_{\frac{1}{d},\frac{1+c}{d}}(|z|^2),
\end{equation}
where $E_{a,b}$ is the two-parametric Mittag-Leffler function
\begin{equation}\label{E_{a,b}}
E_{a,b}(z):= \sum_{k=0}^\infty \frac{z^k}{\Gamma(ak+b)}.
\end{equation}
We mention that the approach in \cite{ameur2018random} using Ward's equation relies on the fact that the limiting $1$-point functions \eqref{K ML da} are rotationally symmetric.
Our approach can also be applied to the bulk case when $a \in (0,1)$, which provides an alternative derivation of the limiting one-point functions \eqref{K ML da}, see Theorem~\ref{Thm_largeN pc}. 
\end{rmk}

An additional advantage of our approach using the Christoffel-Darboux identity lies in the fact that both in Theorems~\ref{Thm_bInsertion} and ~\ref{Thm_largeN pc}, it indeed allows to compute not only the leading order asymptotic but also its fine asymptotic as long as the detailed strong asymptotics of the associated orthogonal polynomial are provided. 
We refer to \cite{lee2016fine,byun2021universal} for previous works in this direction on exactly solvable models. 

\medskip 

The rest of this paper is organised as follows.
\begin{itemize}
    \item  In Section~\ref{Section_RecMFt}, we derive \eqref{recursion_1 pt} and \eqref{lemniscate 1pt function} by showing the recursive formula and the multi-fold transformation of correlation kernels in a general context.
    \item  In Section~\ref{Section_CDI}, we present the  Christoffel-Darboux identities for some class of planar orthogonal polynomials and show Theorem~\ref{Thm_CDI}.
    \item In Section~\ref{Section_largeN}, we derive the large-$N$ limit of the correlation kernel and complete the proof of Theorem~\ref{Thm_bInsertion}. 
    \item This article contains several appendices. 
In appendices~\ref{Section_OP fine asymp} and ~\ref{Section_OP norms fine asymp}, we compile detailed computations used in the proofs of Theorem \ref{Thm_bInsertion}. 
In Appendix~\ref{appendix_asymptotic 1pt}, we discuss the asymptotic behaviours of the $1$-point function in Theorem~\ref{Thm_bInsertion}. 
In Appendix~\ref{appendix_bulk}, we re-derive the bulk scaling limits in Remark~\ref{Rmk_bulk} based on our strategy of using the Christoffel-Darboux formula. 
\end{itemize}

\section{Recursive formula and multi-fold transformation} \label{Section_RecMFt}

In this section, we present the \emph{recursive formula} and \emph{multi-fold transformation} of correlation kernels. 
First let us recall some well-known facts.

Note that $E_{1,c}(z)=z^{-c} e^{z} P(c,z)$, where $P(c,z):=\frac{1}{\Gamma(c)}\gamma(c,z)$ is the regularised incomplete Gamma function. 
For $a \in [0,1)$ fixed, we write $\wh{K}^c_{\textup{bulk}}$, $K^c_{\textup{bulk}}$ for the corresponding correlation kernels.
It then follows from \eqref{K ML da} that
\begin{equation} \label{K R wh bulk}
	\wh{R}^c_{\textup{bulk}}(z)=P(c,|z|^2),\qquad \wh{K}^c_{\textup{bulk}}(z,w)=G(z,w)P(c,z\bar{w}),
	\end{equation}
where 
\begin{equation}
G(z,w):=e^{z\bar{w}-|z|^2/2-|w|^2/2}
\end{equation}
is the bulk Ginibre kernel. 
On the other hand when $c=0$, we have the boundary Ginibre kernel 
 \begin{equation} \label{K bGinibre}
	 	\wh{R}^0_{\textup{edge}}(z)=\tfrac12\erfc( -\tfrac{z+\bar{z}-\SS}{\sqrt{2}} ), \qquad \wh{K}^0_{\textup{edge}}(z,w)=G(z,w)\tfrac12 \erfc( -\tfrac{z+\bar{w}-\SS}{\sqrt{2}} ). 
 \end{equation} 

\subsection{Recursive formula} Let us define the \emph{Berezin kernel}
\begin{equation}
\wh{B}^c_N(z,w):= \frac{|\wh{K}_N^c(z,w)|^2}{\wh{R}_N^c(z)}.
\end{equation}
We now derive the following recursive formula for $\wh{R}^c_N$, see \cite[Lemma 7.6.2]{ameur2011fluctuations} for a similar statement. 

\begin{lem} \label{Lem_recursion}
For any $a \ge 0$ and $c>-1,$ we have
\begin{equation}
\wh{R}_{N}^{c+1} (z)=\wh{R}_{N+1}^c(z)-\wh{B}_{N+1}^c(0,z).
\end{equation}
\end{lem}
As an immediate consequence, by letting $N\to\infty$, we obtain \eqref{recursion_1 pt}.
Before the proof, let us present some examples. 

\begin{eg*} (\textit{Bulk case}) By \eqref{K R wh bulk} and \cite[Eq.(8.7.1)]{olver2010nist}, we have 
\begin{equation}
\frac{|\wh{K}^c_{\textup{bulk}}(0,z)|^2}{\wh{R}^c_{\textup{bulk}}(z)}=e^{-|z|^2} \frac{|z|^{2c}}{\Gamma(c+1)}.
\end{equation}
Then it follows from the recurrence relation of the regularised Gamma function (see \cite[Eq.(8.8.5)]{olver2010nist}) that 
\begin{equation}
\wh{R}^{c+1}_{\textup{bulk}}(z)=\wh{R}^{c}_{\textup{bulk}}(z)-\frac{|\wh{K}^c_{\textup{bulk}}(0,z)|^2}{\wh{R}^c_{\textup{bulk}}(z)}=P(c+1,|z|^2).
\end{equation}

\end{eg*}

\begin{eg*} (\textit{Edge case})
For $a=1$, by \eqref{K bGinibre}, we have
\begin{equation} \label{R wh edge a1}
\wh{R}^1_{\textup{edge}}(z)=\tfrac12 \erfc( -\tfrac{z+\bar{z}}{\sqrt{2}} )-\tfrac12 e^{-|z|^2} | \erfc( -\tfrac{z}{\sqrt{2}} ) |^2. 
\end{equation}
Similarly, we have
\begin{align}
\begin{split}
\wh{R}^2_{\textup{edge}}(z)&=\tfrac12 \erfc( -\tfrac{z+\bar{z}}{\sqrt{2}} )-\tfrac12 e^{-|z|^2} | \erfc( -\tfrac{z}{\sqrt{2}} ) |^2
-\tfrac{1}{\pi-2}e^{-|z|^2} \Big| (\sqrt{\tfrac{\pi}{2}} z- 1)\erfc( -\tfrac{z}{\sqrt{2}} )+e^{-z^2/2}  \Big|^2.
\end{split}
\end{align}
See Figure~\ref{Fig_Rc edge} below for the graphs of $\wh{R}^c_{\textup{edge}}$. 

	\begin{figure}[h!]
	\begin{center}
		\begin{subfigure}[h]{0.32\textwidth}
			\includegraphics[width=\textwidth]{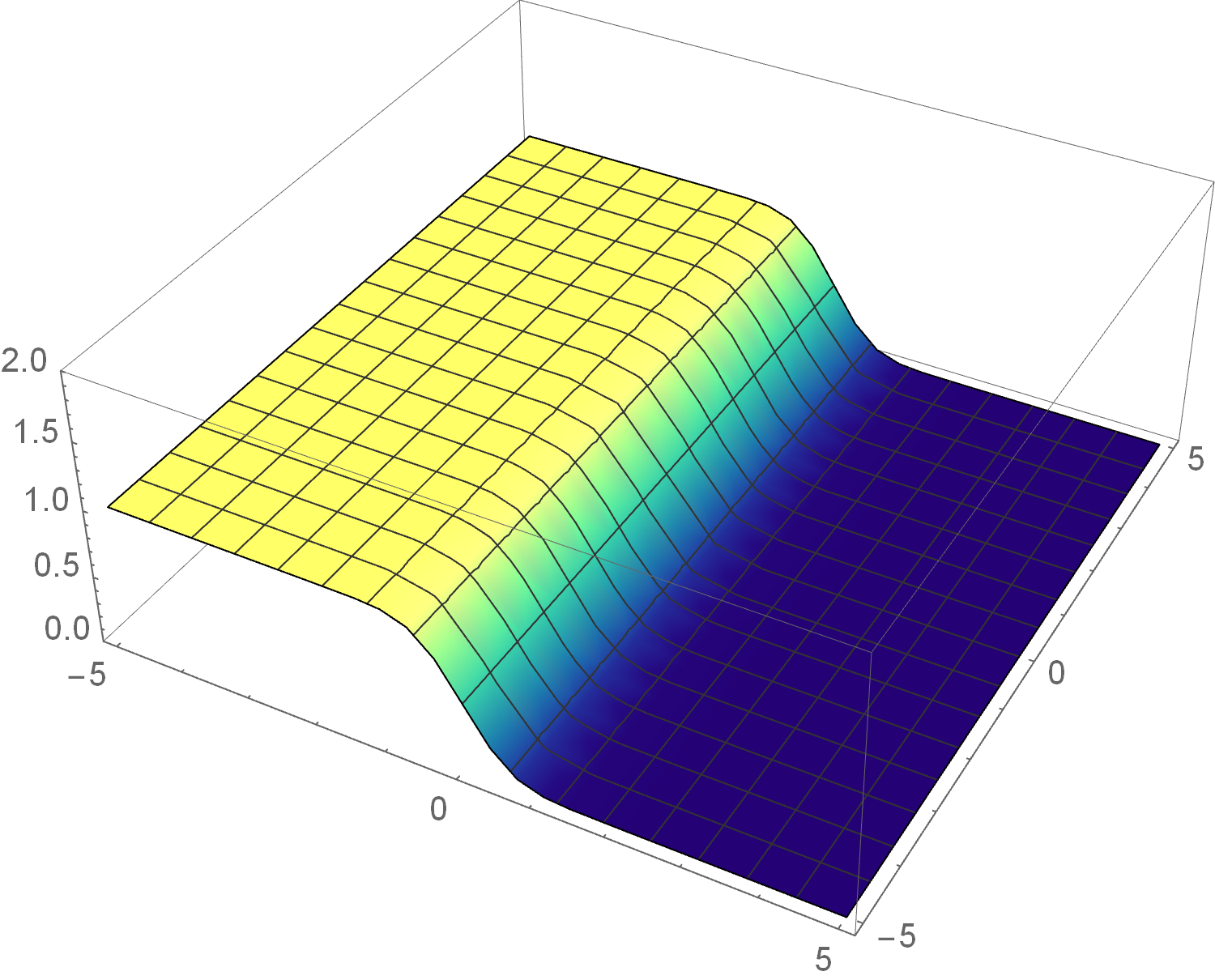}
			\caption{$c=0$}
		\end{subfigure} 
		\begin{subfigure}[h]{0.32\textwidth}
		\includegraphics[width=\textwidth]{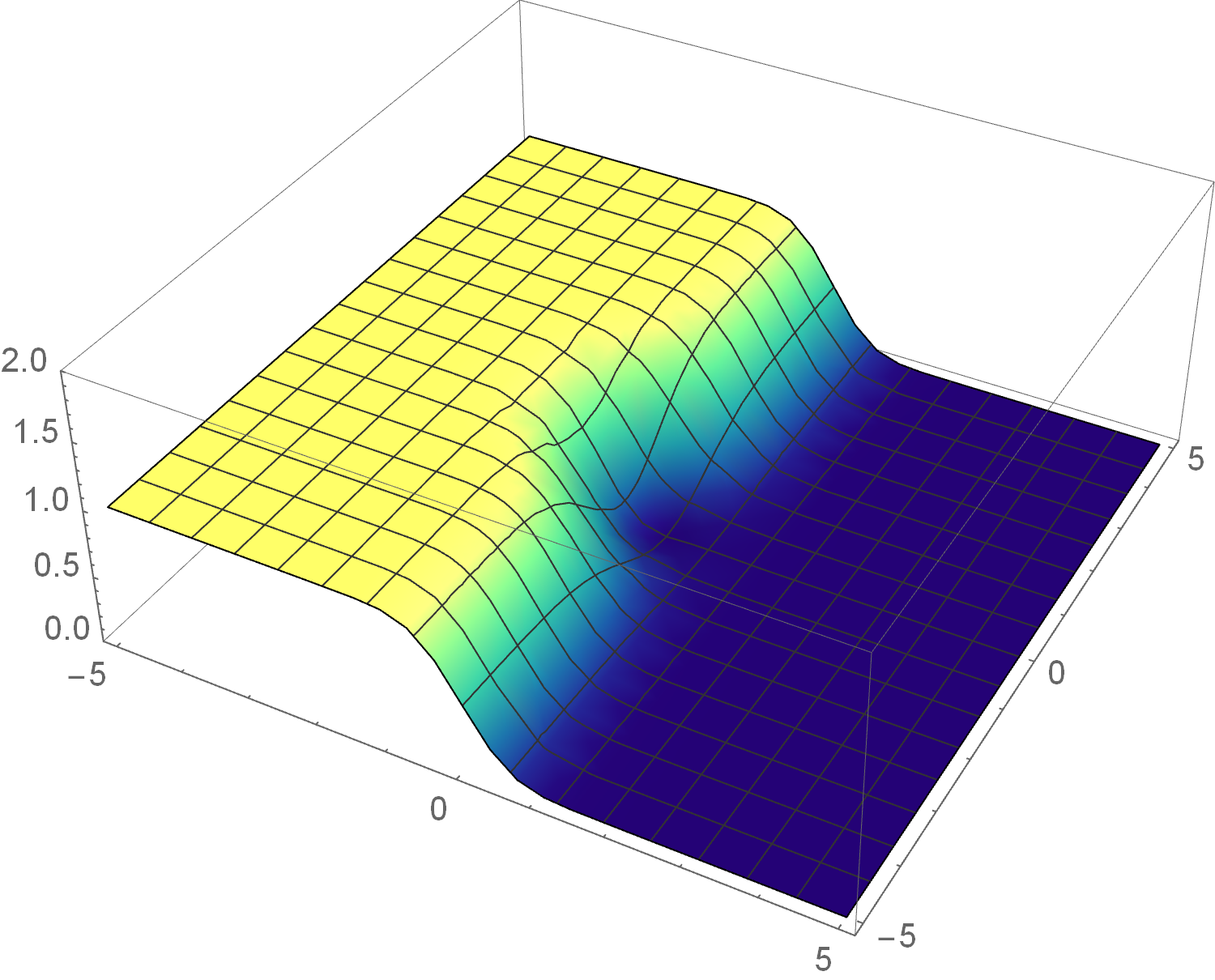}
		\caption{$c=1$}
	\end{subfigure} 
	\begin{subfigure}[h]{0.32\textwidth}
	\includegraphics[width=\textwidth]{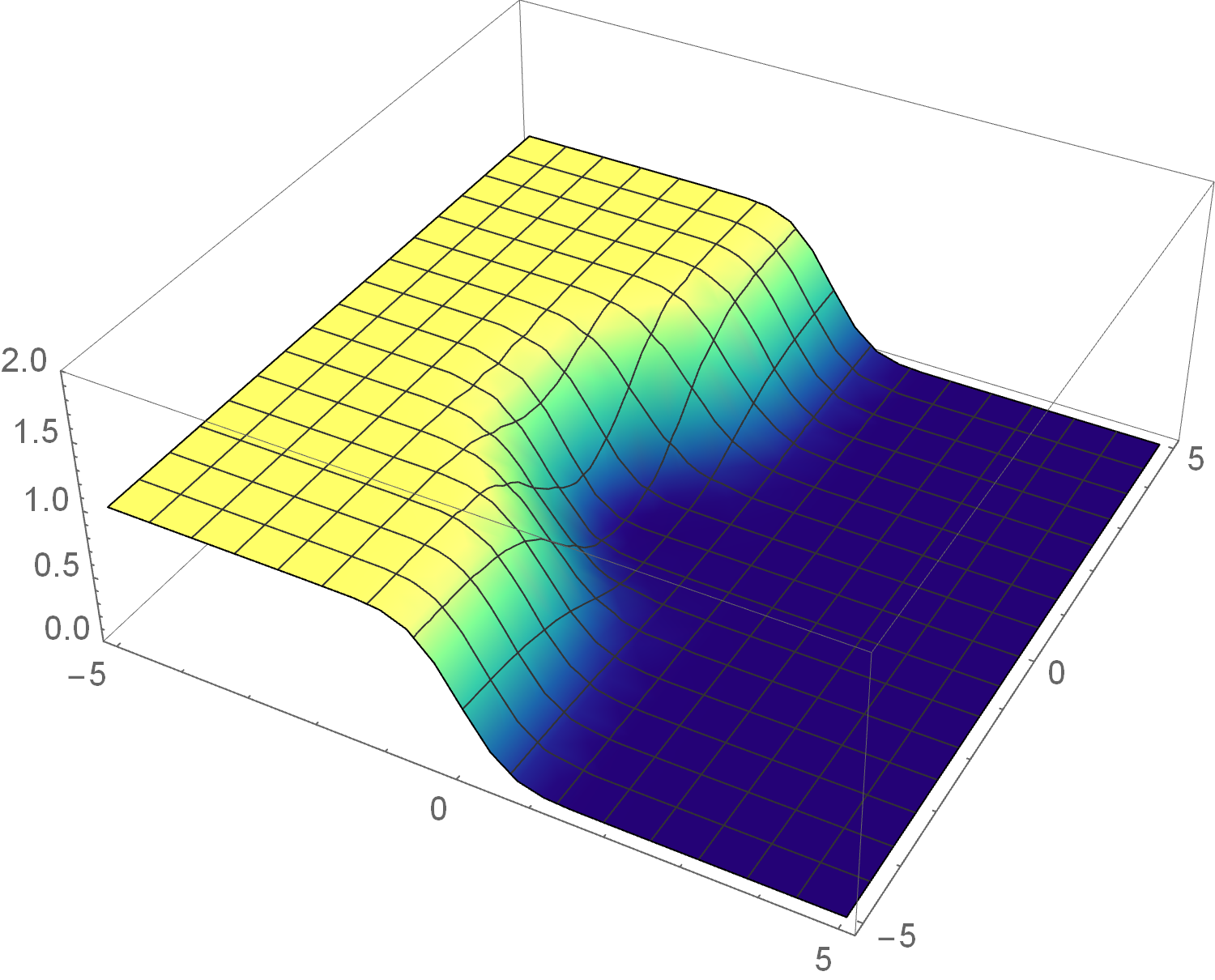}
	\caption{$c=2$}
\end{subfigure} 
	\end{center}
	\caption{The plots display the graphs of $\wh{R}^c_{\textup{edge}}$ for a few values of $c$. }\label{Fig_Rc edge}
\end{figure}

\end{eg*}

We now prove Lemma~\ref{Lem_recursion}. 

\begin{proof}[Proof of Lemma~\ref{Lem_recursion}]

Let us write 
\begin{align}\label{RNk}
\begin{split}
\wh{\bfR}_{N,k}^c(\wh{\zeta}_1,\ldots,\wh{\zeta}_k)&:=\tfrac{1}{\wh{Z}_N} \tfrac{N!}{(N-k)!}  \int_{ \C^{N-k} }  \prod_{j<k} |\wh{\zeta}_j-\wh{\zeta}_k|^2  e^{-N \sum_j Q_c(\wh{\zeta}_j)} \prod_{j=k+1}^N dA(\wh{\zeta}_{j}) 
\end{split}
\end{align}
for the $k$-point correlation (joint intensity) function. Recall that we have
\begin{equation} \label{RNk det}
\wh{\bfR}_{N,k}^c(\wh{\zeta}_1,\ldots,\wh{\zeta}_k)=\det\Big[ \wh{\bfK}_N(\wh{\zeta}_j,\wh{\zeta}_l) \Big]_{j,l=1}^k,
\end{equation}
see e.g. \cite{forrester2010log}.

The main idea of the proof is the following simple observation
\begin{align*}
 | \Delta(\wh{\zeta}_1,\cdots, \wh{\zeta}_N) |^2 \prod_{j=1}^N |\wh{\zeta}_j|^2
=  | \Delta(\wh{\zeta}_1,\cdots, \wh{\zeta}_N,0) |^2, 
\end{align*}
where $\Delta$ is the Vandermonde determinant.
Using this, we have 
\begin{align*}
\wh{\bfR}_{N,1}^{c+1}(\wh{\zeta}_1)&=\frac{1}{\wh{Z}_N} \frac{N!}{(N-1)!}\int_{ \C^{N-1} }  | \Delta(\wh{\zeta}_1,\cdots, \wh{\zeta}_N) |^2  e^{-N \sum_j Q_{c+1}(\wh{\zeta}_j)} \prod_{j=2}^N dA(\wh{\zeta}_{j})
\\
&=\frac{1}{\wh{Z}_N} \frac{1}{N+1}\frac{(N+1)!}{(N-1)!}\int_{ \C^{N-1} }  | \Delta(\wh{\zeta}_1,\cdots, \wh{\zeta}_N,0) |^2  e^{-N \sum_j Q_{c}(\wh{\zeta}_j)} \prod_{j=2}^N dA(\wh{\zeta}_{j})
\\
&=\frac{1}{N+1} \wh{\bfR}_{N+1,2}^{c}(\wh{\zeta}_1,0),
\end{align*}
which leads to 
$$
\wh{\bfR}_{N,1}^{c+1}(\zeta)=\frac{N}{N+1} \frac{\wh{\bfR}_{N+1,2}^{c}(\zeta,0)}{ \wh{\bfR}_{N+1,1}^{c}(0)}.
$$
Therefore we obtain 
\begin{align*}
\wh{R}_{N,1}^{c+1} (z)=\frac{1}{N} \wh{\bfR}_{N,1}^{c+1} ( \tfrac{z}{\sqrt{N}} )=\frac{1}{N+1} \frac{\wh{\bfR}_{N+1,2}^{c}( \frac{z}{\sqrt{N}},0)}{ \wh{\bfR}_{N+1,1}^{c}(0)}.
\end{align*}
Since 
$$
\wh{R}_{N+1,1}^c(0)= \frac{1}{N+1} \wh{\bfR}_{N+1,1}^{c}(0), \qquad   \wh{R}_{N+1,2}^c(0,z)= \frac{1}{(N+1)^2}\wh{\bfR}_{N+1,2}^{c}( \tfrac{z}{\sqrt{N}},0),
$$
we conclude 
$$
\wh{R}_{N,1}^{c+1} (z)= \frac{  \wh{R}_{N+1,2}^c(0,z) }{ \wh{R}_{N+1,1}^c(0) }= \wh{R}_{N+1,1}^c(z)-\wh{B}_{N+1}^c(0,z).
$$
This completes the proof. 
\end{proof}

\subsection{Multi-fold transformations} 
In this subsection, we show Proposition~\ref{Prop_multitrans}. 

We first note that by taking $N \to \infty$ of \eqref{K whK N rel}, we have   
\begin{equation} \label{K_multitrans}
	K^c(z,w)= d(z\bar{w})^{d-1} \sum_{l=0}^{d-1} \widehat{K}^{ \frac{c+l+1}{d}-1  } (z^d,w^d).
	\end{equation}

\begin{rmk}
	In the opposite direction, one can also express $\widehat{K}^c$ in terms of $K^c$. For instance when $d=2$, we have the relations  
	\begin{align}
		K^c(z,w)+K^c(z,-w)&=4z\bar{w} \, \wh{K}^{ \frac{c-1}{2} }(z^2,w^2),
		\\
		K^c(z,w)-K^c(z,-w)&=4z\bar{w} \, \wh{K}^{ \frac{c}{2} }(z^2,w^2).
	\end{align}
	Summing these two equations, we obtain \eqref{K_multitrans} with $d=2$:
	\begin{equation}
		K^c(z,w)=2z\bar{w} \Big( \wh{K}^{ \frac{c-1}{2} }(z^2,w^2)+\wh{K}^{ \frac{c}{2} }(z^2,w^2) \Big) .
	\end{equation}
	We remark that when $c\in \{0,1\}$, the term 
	$2z\bar{w}\, \wh{K}^{ 0 }(z^2,w^2)$ in the right-hand side of the above equation corresponds to the kernel appearing in the context of chiral Ginibre ensembles, see \cite[Theorem 3]{MR4229527}.
\end{rmk}

Before the proof, we interpret $K^c_{\textup{bulk}}$ for general $d>1$ from the viewpoint of Proposition~\ref{Prop_multitrans}. 

\begin{eg*} (\textit{Bulk case}) 
	It follows from the definition \eqref{E_{a,b}} that
	\begin{align*}
		\sum_{l=0}^{d-1} x^l\,E_{1,\frac{c+l+1}{d}}(x^d) &=	\sum_{l=0}^{d-1} \sum_{k=0}^\infty \frac{x^{dk+l}}{\Gamma(k+\frac{c+l+1}{d})}
		= \sum_{j=0}^{\infty} \frac{x^j}{ \Gamma( \frac{j}{d}+\frac{1+c}{d} ) }=	E_{\frac1d,\frac{1+c}{d}}(x).
	\end{align*}
	Then by \eqref{K_multitrans}, one can obtain $K^c_{\textup{bulk}}$ from $\wh{K}^c_{\textup{bulk}}$ as 
	\begin{align}
		\begin{split}
			K^c_{\textup{bulk}}(z,w)&= d(z\bar{w})^{d-1} \sum_{l=0}^{d-1} \widehat{K}_{\textup{bulk}}^{ \frac{c+l+1}{d}-1  } (z^d,w^d)
			\\
			&=d(z\bar{w})^{c} e^{-|z|^{2d}/2-|w|^{2d}/2}  \sum_{l=0}^{d-1} (z\bar{w})^{l} E_{ 1,\frac{c+l+1}{d} }( (z\bar{w})^d )
			\\
			&=d (z\bar{w})^ce^{ -|z|^{2d}/2-|w|^{2d}/2 } E_{\frac1d,\frac{1+c}{d}}(z\bar{w}).
		\end{split} 
	\end{align}
\end{eg*}

\begin{proof}[Proof of Proposition~\ref{Prop_multitrans}]

By the change of variable $\zeta \mapsto \zeta^d$, it is easy to observe that $p_{j,N}^c$ and $q_{j,N}^c$ enjoy the intimate relation
	\begin{equation} \label{ONP transform}
		q_{dj+l,dN}^{c}(\zeta)=\sqrt{d}\,\zeta^l p_{j,N}^{ \frac{c+l+1}{d}-1 }(\zeta^d).
	\end{equation}
    This property is also discussed in \cite[Section 3]{MR3849128} but it is easy enough to recall a proof.
	By definition, we have
	\begin{align*}
		\delta_{jk}&=\int_{ \C }  q^{c}_{j,dN}(\zeta) \overline{ q^{c}_{k,dN} (\zeta)  } |\zeta|^{2c} e^{-dN V(\zeta)}\,dA(\zeta) 
		=\int_{ \C }  p_{j,N}^c(\zeta) \overline{p_{k,N}^c (\zeta) } |\zeta|^{2c} e^{-N Q(\zeta)}\,dA(\zeta).
	\end{align*}
	Since $V_c(\zeta)$ is invariant under the discrete rotation $\zeta \mapsto e^{ 2\pi i/d  } \cdot \zeta$, there exists a polynomial $p_j$ such that $q_{dj+l,dN}^{c}(\zeta)=\zeta^l p_j(\zeta^d).$
	By the change of variable $\eta=\zeta^d$, we have 
	\begin{align*}
		\delta_{jk}
		&=d \int_{ 0<\arg \zeta < \frac{2\pi}{d} } p_j(\zeta^d) \overline{p_k(\zeta^d)} |\zeta|^{2c+2l} e^{-NQ(\zeta^d)}\,dA(\zeta)
	  =\frac{1}{d}   \int_{\C} p_j(\eta) \overline{p_k(\eta)} |\eta|^{ \frac{2c+2l+2}{d}-2 } e^{-NQ(\eta)} \,dA(\eta).
	\end{align*}
	Thus we obtain 
	$p_j(\zeta)=\sqrt{d}\, p_{j,N}^{ \frac{c+l+1}{d}-1 }(\zeta),$
	which leads to \eqref{ONP transform}.
	
	By \eqref{bfKN} and \eqref{bfwhKN} we have
	\begin{align}
		\bfK_{dN}^c(\zeta,\eta)&=(\zeta \bar{\eta})^ce^{ -\frac{dN}{2} V(\zeta)-\frac{dN}{2} V(\eta) } \sum_{j=0}^{dN-1} q_{j,dN}^{c}(\zeta) \overline{q_{j,dN}^{c}(\eta)},
		\\
		\widehat{\bfK}_{N}^c(\zeta,\eta)&=(\zeta \bar{\eta})^ce^{ -\frac{N}{2} Q(\zeta)-\frac{N}{2} Q(\eta) } \sum_{j=0}^{N-1} p_{j,N}^c(\zeta) \overline{p_{j,N}^c(\eta)}.
	\end{align}
	Observe here that by \eqref{ONP transform},
	\begin{align*}
		\sum_{j=0}^{dN-1} q_{j,dN}^{c}(\zeta) \overline{q_{j,dN}^{c}(\eta)} &=  \sum_{l=0}^{d-1} \sum_{j=0}^{N-1} q_{dj+l,dN}^{c}(\zeta) \overline{q_{dj+l,dN}^{c}(\eta)}
		=d \sum_{l=0}^{d-1} (\zeta \bar{\eta})^l \sum_{j=0}^{N-1} p_{j,N}^{ \frac{c+l+1}{d}-1 } (\zeta^d) \overline{p_{j,N}^{ \frac{c+l+1}{d}-1 } (\eta^d)}.
	\end{align*} 
	Thus we obtain 
	\begin{align*}
		\bfK_{dN}^c(\zeta,\eta)&=d \sum_{j=0}^{d-1} (\zeta \bar{\eta})^{c+l}  e^{ -\frac{N}{2} Q(\zeta^d)-\frac{N}{2} Q(\eta^d) } \sum_{j=0}^{N-1} p_{j,N}^{\frac{c+l+1}{d}-1 } (\zeta^d) \overline{p_{j,N}^{ \frac{c+l+1}{d}-1 } (\eta^d)}
		\\
		&=d(\zeta \bar{\eta})^{d-1} \sum_{l=0}^{d-1} \widehat{\bfK}_N^{ \frac{c+l+1}{d}-1 }(\zeta^d,\eta^d).
	\end{align*}
	Therefore we conclude 
	\begin{align*}
K_{dN}^c(z,w)&=\frac{1}{N^{1/d}} \bfK_N^c \Big( \frac{z}{N^{ \frac{1}{2d} }}, \frac{w}{N^{ \frac{1}{2d} }} \Big)
		\\
		&=d(z\bar{w})^{d-1} \frac{1}{N} \sum_{l=0}^{d-1} \widehat{\bfK}_N^{ \frac{c+l+1}{d}-1 } \Big( \frac{z^d}{\sqrt{N}}, \frac{w^d}{\sqrt{N}} \Big)
		= d(z\bar{w})^{d-1} \sum_{l=0}^{d-1} \widehat{K}_N^{ \frac{c+l+1}{d}-1  } (z^d,w^d),
	\end{align*}
	which completes the proof.	
\end{proof}

\section{Christoffel-Darboux identity for planar orthogonal polynomials}\label{Section_CDI}

This section is devoted to proving the Christoffel-Darboux identity, Theorem~\ref{Thm_CDI}, see \cite{MR1917675} for a similar method of deriving such identity in the context of bi-orthogonal polynomials. 
We also refer to \cite[Subsection 4.1]{akemann2021scaling} for a version of the Christoffel-Darboux which involves differential operators. 

\subsection{Elliptic potential revisited}\label{Subsec_CDI EG}

To better introduce the general strategy of deriving the Christoffel-Darboux identity for planar orthogonal polynomials, let us first consider the elliptic potential
\begin{equation}
Q(\zeta):=\tfrac{1}{1-\tau^2}(|\zeta|^2-\tau \,\re\zeta^2), \qquad \tau \in [0,1).
\end{equation}
The random normal matrix ensemble associated with such a potential is equivalent to the \emph{elliptic Ginibre ensemble}. 
It is well known that the orthogonal polynomial with respect to the measure $e^{-NQ}\,dA$ can be expressed in terms of the Hermite polynomial $H_j(x)=(-1)^j e^{x^2} \frac{d^j}{dx^j}e^{-x^2}$. 
More precisely, the monic orthogonal polynomial
\begin{equation}
	P_j(\zeta)=( \tfrac{\tau}{2N} )^{\frac{j}{2}} H_j( \sqrt{  \tfrac{N}{2\tau} }\zeta )
\end{equation}
satisfies the orthogonality relation
\begin{align}
	\int_{ \C }  P_j(\zeta) \overline{ P_k (\zeta) } e^{-N Q(\zeta)}\,dA(\zeta)= h_j \, \delta_{jk} , \qquad 	h_j=\sqrt{1-\tau^2} \tfrac{j!}{N^{j+1}},
\end{align}
see e.g. \cite{van1990new} or \cite[Lemma 7]{MR3845296}.

Let us write
\begin{equation}
	W(\zeta):=e^{ -\frac{\tau N}{2(1-\tau^2)} \zeta^2 }, \qquad \psi_j(\zeta):= W(\zeta) P_j(\zeta), \qquad \phi_j (\zeta):=W(\zeta)\, \frac{P_j(\zeta)}{h_j}.
\end{equation}
Then the associated correlation kernel $\wt{\bfK}_{N}$ is written as 
\begin{align}
\begin{split}
\wt{\bfK}_{N}(\zeta,\eta)&=e^{ -\frac{N}{1-\tau^2} \zeta \bar{\eta} } \sum_{j=0}^{N-1}  \overline{ \phi_j(\eta) } \psi_j(\zeta).
\end{split}
\end{align}
We have the following form of the Christoffel-Darboux identity.
\begin{prop}\label{Prop_CDI Hermite} 
We have 
\begin{equation} \label{CDI Hermite}
\bp_\eta \wt{\bfK}_{N}(\zeta,\eta)=\frac{N}{1-\tau^2} e^{ -\frac{N}{1-\tau^2} \zeta \bar{\eta} } \Big( \tau\, \overline{ \phi_{N}(\eta) } \psi_{N-1}(\zeta) - \overline{ \phi_{N-1}(\eta) } \psi_N(\zeta) \Big).
\end{equation}
\end{prop}

\begin{rmk}
Note in particular that the $1$-point function $\wt{\bfR}_{N}(\zeta)=\wt{\bfK}_{N}(\zeta,\zeta)$ satisfies 
\begin{align}\label{CDI Hermite R}
\begin{split}
\partial_x \wt{\bfR}_{N}(x+iy)=-2\frac{N^{N+1}}{(N-1)!}\frac{ e^{ -N Q(x+iy) } }{(1-\tau^2)(1+\tau)}  \Re\Big[   P_{N}(x-iy)  P_{N-1}(x+iy)  \Big] .
\end{split}
\end{align}
We refer to \cite[Proposition 2.3]{lee2016fine} for a direct proof of \eqref{CDI Hermite} using the three-term recurrence relation and differentiation rule of Hermite polynomials. (See also \cite[Lemma 4.1]{AB} for a related statement.) 
Together with Plancherel-Rotach type strong asymptotics of Hermite polynomials, the identities \eqref{CDI Hermite}, \eqref{CDI Hermite R} were used in \cite{lee2016fine,AB} to derive the associated limiting local kernels in various situations. Beyond the study of determinantal point processes, the identity \eqref{CDI Hermite R} was also utilized to analyse real eigenvalue distributions of real elliptic random matrices, see \cite{byun2021real}.  
\end{rmk}

\begin{proof}[Proof of Proposition~\ref{Prop_CDI Hermite}] 
Let us define semi-infinite dimensional vectors 
$$
\Psi:=[ \psi_0,\psi_1,\cdots ]^t, \qquad \Phi:=[ \phi_0,\phi_1,\cdots ]^t,
$$
where $t$ is the transpose of a matrix and write
\begin{equation} \label{projection op}
	\Pi_N:=\textup{diag}(\underbrace{1,\cdots,1}_N,0,\cdots )
\end{equation}
for the projection (truncation) operator. 
Then the kernel $\wt{\bfK}_{N}$ can be rewritten as 
\begin{align}
\begin{split}
\wt{\bfK}_{N}(\zeta,\eta)
=e^{ -\frac{N}{1-\tau^2} \zeta \bar{\eta} }  \Phi^*(\eta) \Pi_N \Psi(\zeta),
\end{split}
\end{align}
where $*$ denotes the Hermitian transpose of the matrix.
We also define
$$
\langle T | S \rangle := \int \overline{T}\,S\,e^{-\frac{N}{1-\tau^2}|\zeta|^2}\,dA(\zeta)=\overline{ \langle S| T \rangle  }.
$$

By definition, there exist semi-infinite dimensional matrices $A,B$ such that 
\begin{align}
	\pa \phi_j&=\sum_k A_{jk} \phi_k, \qquad \qquad \pa \Phi=A \Phi,
	\\
		\zeta \psi_j(\zeta)&=\sum_k B_{jk} \psi_k(\zeta), \qquad  \zeta \Psi(\zeta)=B \Psi. 
\end{align}
Notice here that integration by parts gives rise to
\begin{align*}
	B_{jk} &=\int_{\C} \overline{ \phi_k(\zeta) } e^{-\frac{N}{1-\tau^2}\zeta \bar{\zeta}} \zeta \psi_j(\zeta)\,dA(\zeta)
	=-\tfrac{1-\tau^2}{N} \int_{\C} \overline{ \phi_k(\zeta) } \bp \Big( e^{-\frac{N}{1-\tau^2}\zeta \bar{\zeta}} \Big) \psi_j(\zeta)\,dA(\zeta)
	\\
	&=\tfrac{1-\tau^2}{N}\int_{\C}  \bp \overline{ \phi_k(\zeta) } \psi_j(\zeta)\,e^{-\frac{N}{1-\tau^2}\zeta \bar{\zeta}}\,dA(\zeta)
=\tfrac{1-\tau^2}N \overline{ A_{kj} }.
\end{align*}
In other words, we have 
\begin{equation}
B=\frac{1-\tau^2}{N} A^*.
\end{equation}
Using this, we obtain
\begin{align}
	\begin{split}
		(\zeta-\tfrac{1-\tau^2}{N}\bp_\eta) 	 \Phi^*(\eta) \Pi_N \Psi(\zeta) &= \Phi^*(\eta)  \Pi_N B  \Psi(\zeta) - \tfrac{1-\tau^2}{N}\Phi^*(\eta) A^* \Pi_N  \Psi(\zeta)
		\\
		&=\Phi^*(\eta) (  \Pi_N B -  B \Pi_N ) \Psi(\zeta). 
	\end{split}
\end{align}
Thus we have 
\begin{align}
\begin{split}
\bp_\eta \bfK_{N}(\zeta,\eta)&=\bp_\eta \Big[e^{ -\frac{N}{1-\tau^2} \zeta \bar{\eta} }  \Phi^*(\eta) \Pi_N \Psi(\zeta) \Big]
\\
&= -\tfrac{N}{1-\tau^2} e^{ -\frac{N}{1-\tau^2} \zeta \bar{\eta} }  (\zeta-\tfrac{1-\tau^2}{N}\bp_\eta) 	 \Phi^*(\eta) \Pi_N \Psi(\zeta)
\\
&= -\tfrac{N}{1-\tau^2} e^{ -\frac{N}{1-\tau^2} \zeta \bar{\eta} } \Phi^*(\eta) (  \Pi_N B -  B \Pi_N ) \Psi(\zeta). 
\end{split}
\end{align}

Now let us determine the matrix $B$. It follows from the three-term recurrence relation of Hermite polynomials that 
\begin{equation}
	\zeta P_j(\zeta)= P_{j+1}(\zeta)+\tfrac{j\tau}{N} P_{j-1}(\zeta). 
\end{equation}
Thus we obtain 
\begin{equation}
B_{j,k}=
\begin{cases}
   1 &\text{if }k=j+1,
   \\
   \frac{j\tau}{N} &\text{if }k=j-1,
   \\
   0 &\text{otherwise}.
\end{cases}
\end{equation}
Using this, we conclude 
\begin{align*}
	 \Phi^*(\eta) (  \Pi_N B -  B \Pi_N ) \Psi(\zeta)&=\overline{ \phi_{N-1}(\eta) } B_{N-1,N} \psi_N(\zeta)-\overline{ \phi_{N}(\eta) } B_{N,N-1} \psi_{N-1}(\zeta)
	 \\
	 &=\overline{ \phi_{N-1}(\eta) } \psi_N(\zeta)-\tau\, \overline{ \phi_{N}(\eta) } \psi_{N-1}(\zeta).
\end{align*}
This completes the proof. 
\end{proof}

\subsection{Gaussian potential with an insertion of a point charge}\label{Subsec_CDI IG}

In this subsection, we derive the Chritoffel-Darboux identity (Theorem~\ref{Thm_CDI}) for the orthogonal polynomials associated with the potential of the form \eqref{Qc Vc}. 
The overall strategy is similar to the one presented in the previous subsection. However, it requires some modifications due to the lack of standard three-term recurrence relation. 
We now prove Theorem~\ref{Thm_CDI}. 

\begin{proof}[Proof of Theorem~\ref{Thm_CDI}] 

As in Subsection~\ref{Subsec_CDI EG}, we write $\Psi:=[ \psi_0,\psi_1,\cdots ]^t$, $\Phi:=[ \phi_0,\phi_1,\cdots ]^t$.

Using the projection operator $\Pi_N$ in \eqref{projection op}, we write
\begin{equation}
	\wt{\bfK}_{N}^c(\zeta,\eta)=e^{-N \zeta\bar{\eta}} \Phi^*(\eta) \Pi_n \Psi(\zeta).
\end{equation}
Then we have 
\begin{equation} \label{bfK bp eta}
\bp_\eta \wt{\bfK}_{N}^c(\zeta,\eta) = -N e^{ -N \zeta \bar{\eta} }  (\zeta-\tfrac{1}{N}\bp_\eta) 	 \Phi^*(\eta) \Pi_n \Psi(\zeta).
\end{equation}

For each $j$, let 
\begin{equation}
L_{j,j-1}=-
\frac{ \int \overline{\zeta\psi_j(\zeta)}\,\phi_0(\zeta)\,e^{-N|\zeta|^2}\,dA(\zeta)}{\int \overline{\zeta\psi_{j-1}(\zeta)}\,\phi_0(\zeta)\,e^{-N|\zeta|^2}\,dA(\zeta)}.
\end{equation}
Notice that the denominator does not vanish. 
Note also that 
$$	
\zeta ( \psi_j(\zeta)+L_{j,j-1} \psi_{j-1}(\zeta)  ) \,  \bot \, \phi_0
$$
with respect to the inner product 
\begin{equation}\label{inner product}
	\langle T | S \rangle := \int \overline{T}\,S\,e^{-N|\zeta|^2}\,dA(\zeta)=\overline{ \langle S| T \rangle  }
\end{equation}
The numbers $L_{j,j-1}$ are building blocks to define the lower diagonal matrix 
$$
L:=\begin{pmatrix}0&0&0&0&\dots\\
L_{2,1}&0&0&0&\dots\\
0&L_{3,2}&0&0&\dots\\
0&0&L_{4,3}&0&\dots\\
\vdots&\vdots&\vdots&\ddots&\vdots
\end{pmatrix}.
$$

Write 
\begin{equation}\label{def psitilde}
\wt{ \psi}_j :=\psi_j+L_{j,j-1} \psi_{j-1}. 
\end{equation}
Then if 
$$
\phi(\zeta)= (\text{polynomials of deg} \le j-2) \cdot (\zeta-a) \cdot (\zeta-a)^c, 
$$
we have
\begin{equation}
\langle \phi \, |  \, \zeta \wt{\psi}_j \rangle =\langle \pa \phi \, | \, \wt{\psi}_j \rangle=0.
\end{equation}
In other words, we have 
$$ 
\mathrm{span}\{ \phi_0, \phi_1,\cdots, \phi_{j-1}  \} \, \bot  \, \zeta \wt{\psi}_j,
$$
which leads to 
\begin{equation}\label{z psi1}
\zeta \wt{\psi}_j(\zeta) 
= \psi_{j+1}(\zeta)+B_{j,j} \, \psi_j(\zeta)
\end{equation}
for some $B_{j,j}$.
Thus we obtain 
\begin{equation} \label{B act}
 \zeta( I+L ) \Psi=B \, \Psi,  \qquad 
B:=\begin{pmatrix}
B_{1,1}&1&0&0&\dots\\
0&B_{2,2}&1&0&\dots\\
0&0&B_{3,3}&1&\dots\\
0 & 0 &0& B_{4,4}&\ddots \\
\vdots & \vdots & \vdots &\ddots&\ddots
\end{pmatrix}.
\end{equation}

Let us also write
\begin{equation} \label{Ujj+1}
 U_{j,j+1}=-\frac{h_{j+1}}{h_j}\frac{P_j(a)}{P_{j+1}(a)}, 
\end{equation} 
and define the upper diagonal matrix 
$$
U:=\begin{pmatrix}
0& U_{1,2}&0&0&\dots\\
0&0& U_{2,3} &0&\dots\\
0&0&0& U_{3,4}&\dots\\
0 & 0 & 0 &0&\ddots\\
\vdots & \vdots & \vdots &\vdots&\ddots
\end{pmatrix}.
$$
Then the function 
\begin{align}
\begin{split}
\label{whphi}
\wh{\phi}_j(\zeta)&:=\phi_j(\zeta)+U_{j,j+1} \phi_{j+1}(\zeta)
=(\text{polynomials of deg}\le j) \cdot (\zeta-a)\cdot (\zeta-a)^c
\end{split}
\end{align}
satisfies
\begin{equation} 
\langle \pa \wh{\phi}_j\, | \, \psi_k \rangle = \langle \wh{\phi}_j\,| \, \zeta \, \psi_k \rangle=0 \quad \text{if }k\le j-2.  
\end{equation}
Thus we have
\begin{equation}
	\pa \wh{\phi}_j=A_{j,j} \phi_j+A_{j,j-1} \phi_{j-1}
\end{equation}
for some $A_{j,k}$, equivalently, 
\begin{equation} \label{A act}
	\pa(I+U) \Phi=A\, \Phi, \qquad 
A:=\begin{pmatrix}
A_{1,1}&0&0&0&\dots\\
A_{2,1}&A_{2,2}&0&0&\dots\\
0&A_{3,2}&A_{3,3}&0&\dots\\
0&0&A_{4,3}&A_{4,4}&\dots\\
\vdots&\vdots&\vdots&\ddots&\ddots
\end{pmatrix}.
\end{equation}

We now determine $A_{j,j-1}$ and $A_{j,j}$. Note that integration by parts gives
\begin{align} \label{AB rel intbypart}
\begin{split}
	\overline{B} (I+U)^t&=\overline{B} \langle \Psi \, | \, \Phi^t \rangle (I+U)^t = \langle B \Psi \, | \, \Phi^t (I+U)^t \rangle 
	\\
	&=\langle  \zeta( I+L ) \Psi\,| \, \Phi^t(I+U)^t \rangle 
=\tfrac{1}{N}\langle ( I+L ) \Psi\,| \, \pa \Phi^t(I+U)^t \rangle  	\\
	&=\tfrac{1}{N}\langle ( I+L ) \Psi\,| \, \Phi^tA^t \rangle=\tfrac{1}{N}(I+\overline{L}) A^t.
\end{split}
\end{align}
Thus we obtain the relation
\begin{equation}\label{AB relation}
	\tfrac{1}{N}A(I+L^*)=(I+U)B^*, \qquad 	B=\tfrac{1}{N}(I+L) A^* (I+U^*)^{-1}.
\end{equation}
Comparing the terms involving $A_{j,j-1}$, one can observe that
\begin{equation} \label{Ajj-1}
A_{j,j-1}=N.
\end{equation}
To determine $A_{j,j}$, note that
\begin{align*}
\pa \wh{\phi}_j(\zeta)&= \pa \Big(  \phi_j+U_{j,j+1} \phi_{j+1} \Big)
=\frac{1}{h_j}\pa \Big( (\zeta-a)^cP_j \Big)+\frac{U_{j,j+1} }{h_{j+1}}\pa  \Big( (\zeta-a)^cP_{j+1} \Big)
\\
&=(\zeta-a)^{c-1} \frac{1}{h_j} \Big[ \Big( cP_j+(\zeta-a)P_j' \Big)-\frac{ P_j(a) }{ P_{j+1}(a) } \Big( cP_{j+1}+(\zeta-a)P_{j+1}' \Big) \Big]
\\
&=(\zeta-a)^{c} \frac{1}{h_j} \Big[ \frac{c}{\zeta-a} \Big(P_j-\frac{ P_j(a) }{ P_{j+1}(a) } P_{j+1} \Big) + P_j' -\frac{ P_j(a) }{ P_{j+1}(a) } P_{j+1}'  \Big].
\end{align*}
This gives 
$$
A_{j,j} P_j+N\frac{ h_j }{ h_{j-1} } P_{j-1}= \frac{c}{\zeta-a} \Big(P_j-\frac{ P_j(a) }{ P_{j+1}(a) } P_{j+1} \Big) + P_j' -\frac{ P_j(a) }{ P_{j+1}(a) } P_{j+1}' . 
$$
Comparing the coefficient of $\zeta^{j}$ term of this identity, we obtain 
\begin{equation}\label{Ajj}
 A_{j,j}=-  \frac{P_j(a)}{P_{j+1}(a)}  (c+j+1) .
\end{equation}
Notice in particular that $A_{j,k}$'s are real. 

Now let us consider the decomposition
\begin{equation}\label{A decomposition}
	A=N\, T_-+A_0,
\end{equation} 
where 
$$T_-:=\begin{pmatrix}0&0&0&0&\dots\\
1&0&0&0&\dots\\
0&1&0&0&\dots\\
0&0&1&0&\dots\\
\vdots&\vdots&\vdots&\ddots&\vdots
\end{pmatrix}, \qquad A_0:=\textup{diag}(A_{1,1},A_{2,2},\dots)
$$
are the translation and the diagonal part respectively. Write
\begin{equation}
A^*=NT_++A_0^*, \qquad T_+:=T_-^*.
\end{equation}

Note also that we have
\begin{align*}
&\quad (T_+-\zeta) \Psi=(T_+-(I+L)^{-1}B ) \Psi
=(T_+-(I+L)^{-1}\tfrac{1}{N}(I+L) A^* (I+U^*)^{-1} ) \Psi
\\
&=(T_+-\tfrac{1}{N}A^* (I+U^*)^{-1} ) \Psi=(T_+(I+U^*)-\tfrac{1}{N}A^* ) (I+U^*)^{-1} \Psi=(T_+U^*-\tfrac{1}{N}A_0^* ) (I+U^*)^{-1} \Psi,
\end{align*}
where the second and the fourth identity follow from \eqref{AB relation} and \eqref{A decomposition} respectively.

We pause here to observe that $(T_+U^*-\tfrac{1}{N}A_0^* )$ is invertible. 
Suppose that this is not the case. Then there exists some $k$ such that
$U^*_{k-1,k}-\tfrac{1}{N}A_{k-1,k-1}=0$. Consequently, we have $\psi_{k}(\zeta)=\zeta\psi_{k-1}(\zeta)$. 
This is a  contradiction due to the assumption $a \not=0$, see Remark~\ref{Rmk_a not0}. 
Therefore we have shown that $(T_+U^*-\tfrac{1}{N}A_0^* )$ is invertible. 

By letting 
\begin{equation}
	\wh{\Psi}:
= (T_+U^*-\tfrac{1}{N}A_0^*)^{-1}  (T_+-\zeta)\Psi,
\end{equation}
we have 
\begin{equation}
	(I+U^*) 	\wh{\Psi}=\Psi. 
\end{equation}

Note that 
\begin{align*}
	\begin{split}
	\wh{\Psi}&:=[ \wh{\psi}_0,\wh{\psi}_1,\cdots ]^t
	=(T_+U^*-\tfrac{1}{N}A_0^*)^{-1}  (T_+-\zeta)  [ \psi_0,\psi_1,\cdots ]^t
	\\
	&=\textup{diag}(  U^*_{1,2}-\tfrac{1}{N}A_{1,1},U^*_{2,3}-\tfrac{1}{N}A_{2,2},\cdots  )^{-1} [ \psi_1-\zeta \psi_0,\psi_2-\zeta \psi_1,\cdots ]^t.
\end{split}
\end{align*}
Thus we have 
\begin{equation} \label{psi hat}
\wh{\psi}_j=\frac{  \psi_{j+1}-\zeta \psi_j }{ U_{j,j+1}^*-\tfrac{1}{N} A_{j,j} }.
\end{equation}
Here the denominator again does not vanish due to \eqref{Ujj+1} and \eqref{Ajj}. 
Then by \eqref{B act} and \eqref{AB relation}, we have
\begin{align*}
(I+L) \zeta 	(I+U^*) 	\wh{\Psi}&=(I+L) \zeta  \Psi = B\Psi 
=B(I+U^*) \wh{\Psi} =\tfrac{1}{N}(I+L) A^* \wh{\Psi},
\end{align*}
which leads to 
\begin{equation} \label{UA Psihat}
  \zeta 	(I+U^*) 	\wh{\Psi}=\tfrac{1}{N}A^* \wh{\Psi}. 
\end{equation}

Combining \eqref{UA Psihat}, \eqref{A act} and \eqref{AB relation}, we obtain  
\begin{align*}
(\zeta-\tfrac{1}{N}\bp_\eta) 	 \Phi^*(\eta) \Pi_n \Psi(\zeta)
&=(\zeta-\tfrac{1}{N}\bp_\eta) 	 \Phi^*(\eta) \Pi_n (I+U^*)\wh{\Psi}(\zeta)
\\
&=\tfrac{1}{N} \Phi^*(\eta) [\Pi_n,A^*] \wh{\Psi}(\zeta)-\tfrac{1}{N}\bp_\eta \Phi^*(\eta)[\Pi_n,I+U^*] \wh{\Psi}(\zeta).
\end{align*}
Moreover by \eqref{Ajj-1} and \eqref{psi hat}, we have 
\begin{align*}
	\begin{split}
\Phi^*(\eta) [\Pi_n,A^*] \wh{\Psi}(\zeta)&=\overline{ \phi_{n-1}(\eta) } A_{n,n-1} \wh{\psi}_n(\zeta)-\overline{ \phi_{n}(\eta) } A_{n-1,n} \wh{\psi}_{n-1}(\zeta)
\\
&=\frac{N}{  U_{n,n+1}^*-\tfrac{1}{N} A_{n,n} } \overline{ \phi_{n-1}(\eta) }  (  \psi_{n+1}(\zeta)-\zeta \psi_n(\zeta)  ).
\end{split}
\end{align*}
Similarly, we obtain 
\begin{align*} 
	\begin{split}
	\bp_\eta \Phi^*(\eta)[\Pi_n,I+U^*] \wh{\Psi}(\zeta)
	&=	\bp_\eta \overline{ \phi_{n-1}(\eta) } U_{n,n-1}^* \wh{\psi}_n(\zeta)-	\bp_\eta \overline{ \phi_{n}(\eta) } U_{n-1,n}^* \wh{\psi}_{n-1}(\zeta)
		\\
		&=    -\frac{U_{n-1,n}^*}{ U_{n-1,n}^*-\tfrac{1}{N} A_{n-1,n-1} }	\bp_\eta \overline{ \phi_{n}(\eta) }( \psi_{n}(\zeta)-\zeta \psi_{n-1}(\zeta) ).
	\end{split}
\end{align*}
Combining all of the above identities with \eqref{bfK bp eta}, the proof is complete.
\end{proof}

To our purpose, let us define 
\begin{equation}
\wt{K}_N^c(z,w):=\frac{1}{N} \wt{\bfK}_{N}^c(\zeta,\eta), \qquad 
\begin{cases}
\zeta=a+\frac{z}{\sqrt{N}},
\\
\eta=a+\frac{w}{\sqrt{N}},
\end{cases}
\end{equation}
and write $\wt{R}_N^c(z):=\wt{K}_N^c(z,z).$ Notice that by \eqref{bfwhKN bfwtKN}, the function $\wt{R}_N^c$ is related to $\wh{R}_N^c$ as 
\begin{equation}\label{wtRN whRN}
\wh{R}_N^c(z)=\wt{R}_N^c(-z).
\end{equation}
As an immediate consequence of Theorem~\ref{Thm_CDI}, we have the following corollary.

\begin{cor}\label{Cor_CDI local}
We have 
\begin{equation} \label{CDI_local R}
\bp_z \wt{R}_{N}^c(z)=e^{-|z|^2} (\RN{1}_N^c(z)- \RN{2}_N^c(z) ),
\end{equation}
where 
\begin{align}
\begin{split} \label{IN c(z)}
\RN{1}_N^c(z)&=\frac{e^{ -a^2N -a\sqrt{N}(z+\bar{z}) } }{N\sqrt{N}} \frac{1}{ \tfrac{N+c}{N}h_{N-1}-h_{N}  }
 \overline{ \psi'_{N}(\zeta) }    \Big( \psi_{N}(\zeta)-\zeta \psi_{N-1}(\zeta) \Big),
\end{split}
\end{align}
and
\begin{align} 
\begin{split} \label{IIN c(z)}
\RN{2}_N^c(z)&=\frac{e^{ -a^2N -a\sqrt{N}(z+\bar{z}) }}{\sqrt{N}} \frac{P_{N+1}(a)}{P_N(a)} \frac{h_N/h_{N-1} }{ \tfrac{N+c+1}{N} h_N-h_{N+1}    } 
 \overline{ \psi_{N-1}(\zeta) }  \Big(  \psi_{N+1}(\zeta)-\zeta \psi_N(\zeta)  \Big).
\end{split}
\end{align}
Here, $\zeta=a+\frac{z}{\sqrt{N}}$.
\end{cor}

\section{Large-$N$ limit of the one-point function}\label{Section_largeN}

In this section, we shall prove Theorem~\ref{Thm_bInsertion}. Let us write $\wt{R}^c_{\textup{bulk}}$ for the large-$N$ limit of $\wt{R}_N^c$ when $a \in [0,1)$ is fixed. Similarly, we write $\wt{R}^c_{\textup{edge}}$ for the large-$N$ limit of $\wt{R}_N^c$ when $a$ is given by \eqref{a critical}. 

We consider the case that $a$ is given by \eqref{a critical} and prove Theorem~\ref{Thm_bInsertion}.
We need to compute the asymptotic behaviours of the right-hand side of \eqref{CDI_local R}. 
This consists of the terms involving orthogonal polynomials and norms which are presented in Lemmas~\ref{Lem_IN 12 1} and ~\ref{Lem_orthogonal norm diff} respectively. 
Combining these, we obtain Theorem~\ref{Thm_bInsertion}.

To derive Lemmas~\ref{Lem_IN 12 1} and ~\ref{Lem_orthogonal norm diff}, we shall use the Riemann-Hilbert analysis for $P_j$ by Bertola, Elias Rebelo and Grava \cite{MR3668632}. 
We postpone to Appendices~\ref{Section_OP fine asymp} and ~\ref{Section_OP norms fine asymp} the proofs of most of the lemmas used during the proof of
the main theorem.

Following \cite[Proposition 4.5]{MR3849128}, let
\begin{equation} \label{def of phi z0}
\phi(z)\equiv \phi(z;z_0)= \frac{z-1}{z_0}-\log z,  \qquad z_0=\frac{1}{a^2} \frac{k}{N}
\end{equation}
and 
\begin{equation} \label{def of A(z0)}
A(z_0)^2=2 \log z_0-2 \frac{z_0-1}{z_0}, \qquad 
\zeta(z)\equiv \zeta(z;z_0)=-A(z_0) + \sqrt{ 2 \phi(z;z_0)+A(z_0)^2 }.
\end{equation}
We also write 
\begin{equation} \label{def of SSr}
\SS_r:= 2\sqrt{N}\Big(  \sqrt{ \frac{N}{N+r} } a -1 \Big)= \SS-\frac{r}{\sqrt{N}}-\frac{r\SS}{2N}+O(N^{-\frac32}) .
\end{equation}
Recall that $\Psi$ is a solution to the Riemann-Hilbert problem in Section~\ref{Section_Intro} and that $H$ and $W=Z/U$ are given by \eqref{def of U Z H}. 
We shall use the fine asymptotic behaviour of the orthogonal polynomial $P_k.$

\begin{prop} \label{Thm_BEG op}
Let $z\in D$ and $k=N+r$ with $r$ fixed. Then as $N \to \infty$, we have
\begin{align}
 P_k(z)  & = z^{k} \Big( \frac{z}{z-a}\Big)^c  e^{\frac{k}{2} \phi(z/a) }   \Big(\sqrt{k} \, \zeta(z/a)\Big)^{ \frac{c}{2} } \Psi_{11}( \sqrt{k}\, \zeta(z/a); \sqrt{k}\,A(z_0) ) \Big( 1+ \frac{\mathfrak{p}_k(z)}{ k^{\frac12} } +O( \frac{1}{k^{1+\frac{c}{2}} } ) \Big) ,
\end{align} 
where 
\begin{equation}
\begin{split}
\mathfrak{p}_k(z) & = H(\SS_r)\Big(\frac{a}{z-a} -\frac{1}{\zeta(z/a)} \Big)
\\
&\quad + W(\SS_r)\Big(  \frac{ a }{ z-a  } \Big( \frac{z-a}{z} \Big)^{c}  \frac{1}{\zeta(z/a)^{ c  }}- \frac{1}{\zeta(z/a)}  \Big) \frac{  \Psi_{21}( \sqrt{k}\, \zeta(z/a); \sqrt{k}\,A(z_0) )  }{ \Psi_{11}( \sqrt{k}\, \zeta(z/a); \sqrt{k}\,A(z_0) )   }.
\end{split}
\end{equation}
Here, $\phi$, $A$ and $\SS_r$ are given by \eqref{def of phi z0}, \eqref{def of A(z0)} and \eqref{def of SSr}. 
\end{prop}  

The leading order asymptotic of Proposition~\ref{Thm_BEG op} is given in  \cite[Theorem 1.3]{MR3849128}.
Furthermore the authors presented a constructive way to derive the subleading correction terms albeit it requires long (but straightforward) computations.
We defer the detailed computations to Appendix~\ref{Section_OP fine asymp}. 

As a direct consequence of Proposition~\ref{Thm_BEG op}, we obtain the following. 
To lighten notations, we sometimes omit the argument and write for instance
\begin{equation}
\Psi_{11}\equiv \Psi_{11}(z;\SS), \qquad H=H(\SS), \qquad  \pa_s \Psi_{11}= \pa_s \Psi_{11}(z;s)|_{s=\SS}. 
\end{equation}

\begin{lem} \label{Lem_OP diff deri}
Let $z\in  D$. Then as $N \to \infty$, the following holds.
\begin{itemize}
    \item \textbf{(Asymptotics of the difference)} We have 
\begin{align*}
&\quad \psi_{N+1} \Big(a+\frac{z}{\sqrt{N}}\Big)-\Big(a+\frac{z}{\sqrt{N}}\Big)  \psi_{N} \Big(a+\frac{z}{\sqrt{N}}\Big)  = \Big(a+\frac{z}{\sqrt{N}}\Big)^{N+1+c}  e^{ \frac{z(z+2\SS)}{4} } z^{ \frac{c}{2} } \Psi_{11}(z;\SS) 
\\
&\times \Big[ -\Big( \frac{z}{2}+ \frac{ \pa_s \Psi_{11} }{ \Psi_{11}   } \Big) \frac{1}{\sqrt{N}} + \Big[  \mathfrak{F}(z)+\mathfrak{G}(z)+\mathfrak{H}(z)+ \Big( \frac{z}{2} \frac{ \pa_s \Psi_{11} }{ \Psi_{11}   } + \frac{z^2}{8} +
\frac12 \frac{ \pa_s^2 \Psi_{11} }{ \Psi_{11} } \Big)   \Big] \frac{1}{N}+O(N^{-\frac32})\Big]
\end{align*}
and
\begin{align*}
&\quad \psi_{N} \Big(a+\frac{z}{\sqrt{N}}\Big)-\Big(a+\frac{z}{\sqrt{N}}\Big)  \psi_{N-1} \Big(a+\frac{z}{\sqrt{N}}\Big)  = \Big(a+\frac{z}{\sqrt{N}}\Big)^{N+c}  e^{ \frac{z(z+2\SS)}{4} } z^{ \frac{c}{2} } \Psi_{11}(z;\SS) 
\\
&\times \Big[ -\Big( \frac{z}{2}+ \frac{ \pa_s \Psi_{11} }{ \Psi_{11}   } \Big) \frac{1}{\sqrt{N}} + \Big[  \mathfrak{F}(z)+\mathfrak{G}(z)+\mathfrak{H}(z)- \Big( \frac{z}{2} \frac{ \pa_s \Psi_{11} }{ \Psi_{11}   } + \frac{z^2}{8} +
\frac12 \frac{ \pa_s^2 \Psi_{11} }{ \Psi_{11} } \Big)   \Big] \frac{1}{N} +O(N^{-\frac32})\Big],
\end{align*}
where
\begin{align}
\begin{split} \label{def of mathfrak F}
\mathfrak{F}(z)&= c\frac{ z(2z+\SS) }{24} +\frac{z}{2} \Big[ \frac{z(2z+\SS)}{6} \frac{ \pa_z \Psi_{11}}{ \Psi_{11} }+ \frac{\SS^2}{12} \frac{ \pa_s \Psi_{11} }{ \Psi_{11}   } \Big]
\\
&\quad + \Big[ \frac{ z(4z^2+6\SS z+3\SS^2) }{24}+\frac{c}{2} \frac{2z+\SS}{6}\Big]  \frac{ \pa_s \Psi_{11} }{ \Psi_{11}   } +  \frac{ z^2(4z^2+6\SS z+3\SS^2) }{48}  \\
&\quad + \frac{ z(z+\SS)}{4} +\frac{c}{12} +  \frac{z}{12}  \frac{ \pa_z^2 \Psi_{11} }{ \Psi_{11} }+ \frac{1}{12}  \frac{ \pa_s^2 \Psi_{11} }{ \Psi_{11}  }+\frac{z(2z+\SS)}{6}  \frac{ \pa_z \pa_s \Psi_{11} }{   \Psi_{11}    },
\end{split}
\end{align}
\begin{equation}
 \label{def of mathfrak G}
\mathfrak{G}(z) = -\frac{\SS-z}{3z}   \Big( H'+W'\frac{ \Psi_{21} }{ \Psi_{11} } \Big) -  \frac{H }{3z} +W  \frac{ \Psi_{21} }{ \Psi_{11} } \frac{c-1}{3z} + W\frac{ \Psi_{21} }{ \Psi_{11} } \frac{\SS-z}{3z}   \Big(  \frac{ \pa_s \Psi_{11} }{ \Psi_{11} }-\frac{ \pa_s \Psi_{21} }{ \Psi_{21} } \Big), 
\end{equation}
and 
\begin{equation} \label{def of mathfrak H}
\mathfrak{H}(z) = -\Big( \frac{z}{2}+ \frac{ \pa_s \Psi_{11} }{ \Psi_{11}   } \Big) \frac{\SS-z}{3z} \Big( H+W\frac{ \Psi_{21} }{ \Psi_{11} } \Big).
\end{equation}
 \item \textbf{(Asymptotics of the derivative)} We have 
\begin{align*}
 \frac{1}{N} \psi_N'\Big( a+\frac{z}{\sqrt{N}} \Big) & =  \Big(a+\frac{z}{\sqrt{N}}\Big)^{N+c-1}  e^{ \frac{z(z+2\SS)}{4} } z^{ \frac{c}{2} } \Psi_{11}(z;\SS)
\\
&\quad \times \Big[ 1+ \Big( \frac{z+\SS}{2}+\frac{c}{2}\frac{1}{z}+\frac{ \pa_z \Psi_{11} }{ \Psi_{11} }+F_N^{(1)}(z)+G_N^{(1)} \Big) \frac{1}{\sqrt{N}}+O( \frac{1}{N^{ 1+\frac{c}{2} }} )  \Big],
\end{align*}
where 
\begin{equation}
\label{Fk (1)}
F_k^{(1)}(z) = - \Big( \frac{ z(4z^2+6\SS z+3\SS^2) }{24}+\frac{r}{2}z\Big) -\frac{c}{2} \frac{2z+\SS}{6}+ \Big[ -\frac{z(2z+\SS)}{6} \frac{ \pa_z \Psi_{11}}{ \Psi_{11} }-\Big( \frac{\SS^2}{12}+r\Big) \frac{ \pa_s \Psi_{11} }{ \Psi_{11}   } \Big] 
\end{equation}
and 
\begin{equation} \label{Gk (1)}
G_k^{(1)}(z)= \frac{\SS-z}{3z} \Big( H+W\frac{ \Psi_{21} }{ \Psi_{11} } \Big). 
\end{equation}
\end{itemize}
\end{lem}

\begin{lem}[\textbf{Asymptotics of the terms involving orthogonal polynomials}] \label{Lem_IN 12 1}
Let $z\in  D$. Then as $N \to \infty$, we have
\begin{align}
\begin{split}
&\quad	\frac{ 1 }{N}  \overline{ \psi'_{N}\Big(a+\frac{z}{\sqrt{N}}\Big) }  \bigg[ \psi_{N}\Big(a+\frac{z}{\sqrt{N}}\Big)-\Big(a+\frac{z}{\sqrt{N}}\Big) \psi_{N-1}\Big(a+\frac{z}{\sqrt{N}}\Big) \bigg]  
\\
&=  e^{ (\SS+z+\bar{z})\sqrt{N} -\frac{z^2+\bar{z}^2}{4}-\frac{\SS^2}{4} }  |z|^c |\Psi_{11}(z;\SS)|^2 
\\
&\times \Big[  -\Big( \frac{z}{2}+ \frac{ \pa_s \Psi_{11} }{ \Psi_{11}   } \Big) \frac{1}{\sqrt{N}}+ \frac{ \mathfrak{I}(z) }{N} +O(N^{-\frac32-\frac{c}{2}}) \Big]
\Big[ 1+ \mathfrak{C}(z) \frac{1}{\sqrt{N}}+O(\frac{1}{N})\Big]
\end{split}
\end{align}
and 
\begin{align}
\begin{split}
&\quad 	\overline{ \psi_{N-1}\Big(a+\frac{z}{\sqrt{N}}\Big) }  \bigg[   \psi_{N+1}\Big(a+\frac{z}{\sqrt{N}}\Big)-\Big(a+\frac{z}{\sqrt{N}}\Big) \psi_N \Big(a+\frac{z}{\sqrt{N}}\Big) \bigg] 
\\
&=  e^{ (\SS+z+\bar{z})\sqrt{N} -\frac{z^2+\bar{z}^2}{4}-\frac{\SS^2}{4} }  |z|^c |\Psi_{11}(z;\SS)|^2 
\\
&\quad \times \Big[  -\Big( \frac{z}{2}+ \frac{ \pa_s \Psi_{11} }{ \Psi_{11}   } \Big) \frac{1}{\sqrt{N}} + \frac{ \mathfrak{II}(z) }{ N }+O(N^{-\frac32-\frac{c}{2}}) \Big]\Big[ 1+ \mathfrak{C}(z) \frac{1}{\sqrt{N}}+O(\frac{1}{N})\Big],
\end{split}
\end{align}
where 
\begin{align}
\begin{split} \label{mathfrak I}
 \mathfrak{I}(z) &= -\Big( \frac{z}{2}+ \frac{ \pa_s \Psi_{11} }{ \Psi_{11}   } \Big) \Big( \overline{  \frac{z+\SS}{2}+\frac{c}{2}\frac{1}{z}+\frac{ \pa_z \Psi_{11} }{ \Psi_{11} }+F_N^{(1)}(z)+G_N^{(1)} } \Big)
 \\
 &\quad + \mathfrak{F}(z)+\mathfrak{G}(z)+\mathfrak{H}(z)-  \Big( \frac{z}{2} \frac{ \pa_s \Psi_{11} }{ \Psi_{11}   } + \frac{z^2}{8} +
\frac12 \frac{ \pa_s^2 \Psi_{11} }{ \Psi_{11} } \Big), 
\end{split}
\end{align}
\begin{align}
\begin{split} \label{mathfrak II}
\mathfrak{II}(z) &= -\Big( \frac{z}{2}+ \frac{ \pa_s \Psi_{11} }{ \Psi_{11}   } \Big) \Big( \frac{2z+\SS}{2}+\overline{ F_{N-1}^{(1)}(z)+G_{N-1}^{(1)}(z)}  \Big)
\\
&\quad +\mathfrak{F}(z)+\mathfrak{G}(z)+\mathfrak{H}(z)+ \Big( \frac{z}{2} \frac{ \pa_s \Psi_{11} }{ \Psi_{11}   } + \frac{z^2}{8} +
\frac12 \frac{ \pa_s^2 \Psi_{11} }{ \Psi_{11} } \Big) , 
\end{split}
\end{align}
and 
\begin{equation}
\label{def of mathfrack C}
\mathfrak{C}(z) =\frac{1}{24} \Big( (\SS+2z)(12c+(\SS+2z)^2) + (\SS+2\bar{z})(12(c-1)+(\SS+2\bar{z})^2)\Big).
\end{equation}
Here $\mathfrak{F}, \mathfrak{G}$, $\mathfrak{H}$ are given by \eqref{def of mathfrak F}, \eqref{def of mathfrak G}, \eqref{def of mathfrak H} and $F_k^{(1)}, G_k^{(1)}$ are given by \eqref{Fk (1)}, \eqref{Gk (1)}. 
\end{lem}
\begin{proof}
This directly follows from Lemma~\ref{Lem_OP diff deri}. 
Note that the term \eqref{def of mathfrack C} originates from  
\begin{equation}
\begin{split}
&\quad \overline{ \Big(a+\frac{z}{\sqrt{N}}\Big)^{N+c-1} }  \Big(a+\frac{z}{\sqrt{N}}\Big)^{N+c} = e^{ (\SS+z+\bar{z})\sqrt{N}-\frac18 (\SS+2z)^2-\frac18 (\SS+2\bar{z})^2  }
\\
&\times \Big[ 1+ \frac{1}{24} \Big( (\SS+2z)(12c+(\SS+2z)^2) + (\SS+2\bar{z})(12(c-1)+(\SS+2\bar{z})^2)\Big) \frac{1}{\sqrt{N}}+O(\frac{1}{N})\Big].
\end{split}
\end{equation}
\end{proof}

\begin{lem}[\textbf{Asymptotics of the terms involving orthogonal norms}] \label{Lem_orthogonal norm diff}
As $N \to \infty$, we have
\begin{equation}
 \frac{e^{ -a^2N +\SS \sqrt{N} } }{ \sqrt{N}} \frac{1}{ \tfrac{N+c}{N}h_{N-1}-h_{N}  }= \frac{ e^{-\SS^2/4} }{ \sqrt{2\pi} }  \AAA \Big(  N-\AAA \BBB_0\sqrt{N}+o(\sqrt{N}) \Big) 
\end{equation}
and 
\begin{equation}
\frac{e^{ -a^2N +\SS \sqrt{N} } }{ \sqrt{N}}  \frac{h_N/h_{N-1} }{ \tfrac{N+c+1}{N} h_N-h_{N+1}    } =\frac{ e^{-\SS^2/4} }{ \sqrt{2\pi} }  \AAA\Big( N-\AAA(\BBB_0+\BBB)\sqrt{N}+o(\sqrt{N})  \Big) 
\end{equation}
where 
\begin{equation}
\BBB_0= \frac32 \frac{W'(\SS)}{W(\SS)} -\SS \Big( \frac{W'(\SS)}{W(\SS)} \Big)^2 - \Big( \frac{W'(\SS)}{W(\SS)} \Big)^3+\SS \frac{W''(\SS)}{W(\SS)}. 
\end{equation}
Here $\AAA$ and $\BBB$ are given by \eqref{def of AAA} and \eqref{def of BBB}. 
\end{lem}

\begin{prop} \label{Prop_asymptotic derivative} 
As $N \to \infty$, we have
\begin{equation}
\begin{split}
	\bp_z \wt{R}_{N}^c(z) & =  \frac{ \AAA }{ \sqrt{2\pi} }\,e^{ -|z|^2-\frac{ (z+\SS)^2+(\bar{z}+\SS)^2 }{4} } \,|z|^c  \,\Psi_{21}(z;\SS) \,W
\\
&\times \Big[    \overline{\Psi_{11}(z;\SS)} \Big( \frac{Z}{\overline z}-z -\AAA \BBB + \frac{Z+c}{Y} -  Y \frac{Z}{Z+c}\Big)-\frac{W}{\overline z}\frac{Z+c}{Y} \overline{ \Psi_{21}(z;\SS)  }   +Z +o(1)\Big].
\end{split}
\end{equation}
\end{prop}

\begin{proof}
Recall that $\RN{1}_N^c$ and $\RN{2}_N^c$ are given by \eqref{IN c(z)} and \eqref{IIN c(z)}. 
Note that by \eqref{psi k asymp local}, we have
\begin{equation}
\begin{split}
\frac{P_{N+1}(a)}{P_N(a)} &= a \,\Big( 1 -\Big(  \frac{ \pa_s \Psi_{11} }{ \Psi_{11}   } \Big)\frac{ 1}{ \sqrt{N} } +O( \frac{1}{N^{1+\frac{c}{2}}} ) \Big) = 1 +\Big( \frac{\SS}{2}-\frac{ \pa_s \Psi_{11} }{ \Psi_{11}   } \Big)\frac{ 1}{ \sqrt{N} } +O( \frac{1}{N^{1+\frac{c}{2}}} )  . 
\end{split}
\end{equation}
Combining this with Lemmas~\ref{Lem_IN 12 1} and \ref{Lem_orthogonal norm diff}, after simplifications, we obtain  
\begin{align*}
	\RN{1}_N^c(z)& =-\sqrt{N} \, \frac{ e^{ -\frac{ (z+\SS)^2+(\bar{z}+\SS)^2 }{4} } }{ \sqrt{2\pi} } \,|z|^c |\Psi_{11}(z;\SS)|^2  \Big( \frac{z}{2}+ \frac{ \pa_s \Psi_{11} }{ \Psi_{11}   } \Big)   \AAA  
\\
&\quad \times \Big[ 1+ \Big( -\AAA \BBB-\Big( \frac{z}{2}+ \frac{ \pa_s \Psi_{11} }{ \Psi_{11}   } \Big) ^{-1} \mathfrak{I}(z)+\mathfrak{C}(z)  \Big)\frac{1}{\sqrt{N}}+O( \frac{1}{N} ) \Big]
\end{align*}
and 
\begin{align*}
&\quad	\RN{2}_N^c(z)=-\sqrt{N} \, \frac{ e^{ -\frac{ (z+\SS)^2+(\bar{z}+\SS)^2 }{4} } }{ \sqrt{2\pi} } \,|z|^c |\Psi_{11}(z;\SS)|^2  \Big( \frac{z}{2}+ \frac{ \pa_s \Psi_{11} }{ \Psi_{11}   } \Big)   \AAA 
\\
&\times \Big[ 1+ \Big( -\AAA (\BBB_0+\BBB)-\Big( \frac{z}{2}+ \frac{ \pa_s \Psi_{11} }{ \Psi_{11}   } \Big) ^{-1} \mathfrak{II}(z)+\mathfrak{C}(z) + \frac{\SS}{2}- \frac{ \pa_s \Psi_{11}(0,\SS) }{ \Psi_{11}(0,\SS)   }  \Big)\frac{1}{\sqrt{N}}+O( \frac{1}{N}  ) \Big].
\end{align*}
This gives 
\begin{equation}
\begin{split}
&\quad \RN{1}_N^c(z)-\RN{2}_N^c(z) =  \frac{ e^{ -\frac{ (z+\SS)^2+(\bar{z}+\SS)^2 }{4} } }{ \sqrt{2\pi} } \,|z|^c |\Psi_{11}(z;\SS)|^2    \AAA
\\
&\times \Big[  \mathfrak{I}(z)-\mathfrak{II}(z)+\Big( \frac{z}{2}+ \frac{ \pa_s \Psi_{11} }{ \Psi_{11}   } \Big) \Big( -\AAA \BBB + \frac{\SS}{2}- \frac{ \pa_s \Psi_{11}(0,s)|_{s=\SS} }{ \Psi_{11}(0,\SS)   } \Big)  +o(1)\Big]. 
\end{split}
\end{equation}

We now compute $ \mathfrak{I}(z)-\mathfrak{II}(z)$. 
By \eqref{mathfrak I} and \eqref{mathfrak II}, we have 
\begin{align}
\mathfrak{I}(z)-\mathfrak{II}(z) &= -\Big( \frac{z}{2}+ \frac{ \pa_s \Psi_{11} }{ \Psi_{11}   } \Big) \Big( -z+\overline{  \frac{c}{2}\frac{1}{z}+\frac{ \pa_z \Psi_{11} }{ \Psi_{11} }-\frac{ \pa_s \Psi_{11} }{ \Psi_{11}   }   }  \Big)  - \Big( z \frac{ \pa_s \Psi_{11} }{ \Psi_{11}   } + \frac{z^2}{4}+\frac{ \pa_s^2 \Psi_{11} }{ \Psi_{11} } \Big) . 
\end{align}
By the Lax pair given in \cite[Eq.(2.1)]{MR3849128}, we have
\begin{equation} \label{Psi 11 pa z}
\frac{\pa_z\Psi_{11}(z;s)}{\Psi_{11}(z;s)}=-\frac{z+s}{2}-\frac{\frac{c}{2}+Z}{z}+W\Big(1+\frac{1}{z}\frac{c+Z}{Y}\Big)\frac{\Psi_{21}(z;s)}{\Psi_{11}(z;s)}.
\end{equation}
Also we have
\begin{align}
 \label{Psi 11 pa s}
\frac{\pa_s\Psi_{11}(z;s)}{\Psi_{11}(z;s)}& =-\frac{z}{2}+W\frac{\Psi_{21}(z;s)}{\Psi_{11}(z;s)},
\\
\label{Psi 11 pa s s}
\frac{ \pa_s^2 \Psi_{11}(z;s) }{ \Psi_{11}(z;s) } & = \frac{z^2}{4}-Z(s)+ W'(s) \frac{ \Psi_{21}(z;s) }{ \Psi_{11}(z;s) }.
\end{align}
Combining \eqref{Psi 11 pa z}, \eqref{Psi 11 pa s} and \eqref{Psi 11 pa s s}, after some computations, we obtain 
\begin{equation}
\mathfrak{I}(z)-\mathfrak{II}(z) = W\frac{\Psi_{21}}{\Psi_{11}} \Big( \frac{\SS}{2}+\frac{Z}{\overline z}-z-\frac{W'}{W} \Big)-\frac{W^2}{\overline z}\frac{c+Z}{Y} \Big| \frac{\Psi_{21}}{\Psi_{11}} \Big|^2   +Z, 
\end{equation}
which leads to 
\begin{equation}
\begin{split}
&\quad \RN{1}_N^c(z)-\RN{2}_N^c(z) =  \frac{ e^{ -\frac{ (z+\SS)^2+(\bar{z}+\SS)^2 }{4} } }{ \sqrt{2\pi} } \,|z|^c |\Psi_{11}(z;\SS)|^2    \AAA
\\
&\times \Big[  W\frac{\Psi_{21}}{\Psi_{11}} \Big( \SS+\frac{Z}{\overline z}-z-\frac{W'}{W}-\AAA \BBB - \frac{ \pa_s \Psi_{11}(0,s)|_{s=\SS} }{ \Psi_{11}(0,\SS)   } \Big)-\frac{W^2}{\overline z}\frac{c+Z}{Y} \Big| \frac{\Psi_{21}}{\Psi_{11}} \Big|^2   +Z +o(1)\Big].
\end{split}
\end{equation}
Recall that $W$ is given by \eqref{def of W}. 
Using \eqref{U' Z' Y'} and \eqref{def of U Z H}, it is straightforward to check 
\begin{gather}
\frac{W'}{W} =s-\frac{Z+c}{Y}, \qquad  \frac{W''}{W} = 1-c+s^2-2Z-s \frac{Z+c}{Y}, \label{W' W''}
\\
\frac{W'''}{W} = s (3-2c-4Z+s^2)-2YZ + \frac{ c^2-c(2+s^2-5Z)+Z(4Z-s^2-2) }{Y}. \label{W'''}
\end{gather}
Now it follows from \eqref{CDI_local R}, \eqref{W' W''} and 
\begin{equation} \label{Psi pa s 0}
\frac{ \pa_s \Psi_{11}(0,s)|_{s=\SS} }{ \Psi_{11}(0,\SS)   } = W(\SS) \frac{  \Psi_{21}(0,\SS) }{ \Psi_{11}(0,\SS)   } = Y \frac{Z}{Z+c}
\end{equation}
that 
\begin{equation} 
\begin{split}
&\quad 	\bp_z \wt{R}_{N}^c(z) =  \frac{ \AAA }{ \sqrt{2\pi} }\,e^{ -|z|^2-\frac{ (z+\SS)^2+(\bar{z}+\SS)^2 }{4} } \,|z|^c |\Psi_{11}(z;\SS)|^2  
\\
&\times \Big[   W\frac{\Psi_{21}(z;\SS)}{\Psi_{11}(z;\SS)} \Big( \frac{Z}{\overline z}-z -\AAA \BBB + \frac{Z+c}{Y} -  Y \frac{Z}{Z+c}\Big)-\frac{W^2}{\overline z}\frac{Z+c}{Y} \Big| \frac{\Psi_{21}(z;\SS)}{\Psi_{11}(z;\SS)} \Big|^2   +Z +o(1)\Big].
\end{split}
\end{equation}
\end{proof}

\begin{proof}[Proof of Theorem~\ref{Thm_bInsertion}]
By \cite[Theorem 5]{MR4030288}, for any $y \in \R$, we have
\begin{equation}
\wt{R}_{\textup{edge}}^c(x+iy) \to 
\begin{cases}
0 & \text{as }	x \to +\infty,,
\\
1 & \text{as }x \to -\infty,
\end{cases} \qquad (x\in \R). 
\end{equation}
Using Proposition~\ref{Prop_asymptotic derivative}, the first limit gives 
\begin{align}
\begin{split}
&\quad \wt{R}_{\textup{edge}}^c(z) =  \frac{ \AAA }{ \sqrt{2\pi} } \,e^{ -\frac{ (z+\SS)^2 }{4} } \,z^{ \frac{c}{2} }  \,\Psi_{21}(z;\SS) \,W 
\\
&\times \int_{+\infty}^{\bar{z}}   e^{ -z w-\frac{ (w+\SS)^2 }{4} } w^{ \frac{c}{2} }   \Big[   \Psi_{11}(w;\SS) \Big( \frac{Z}{w}-z -\AAA \BBB + \frac{Z+c}{Y} -  Y \frac{Z}{Z+c}\Big)-\frac{W}{w}\frac{Z+c}{Y}  \Psi_{21}(w;\SS)     +Z \Big] \,dw. 
\end{split}
\end{align}
Now the relation \eqref{wtRN whRN} completes the proof. 
\end{proof}

\appendix

\section{Fine asymptotic behaviours of the orthogonal polynomials} \label{Section_OP fine asymp}

In this section, we show Proposition~\ref{Thm_BEG op} and Lemma~\ref{Lem_OP diff deri}.

\begin{proof}[Proof of Proposition~\ref{Thm_BEG op}]  
Let $z\in \Omega_\infty\cap D$. We also write $\Psi \equiv \Psi^c$ and
\begin{equation} \label{Psi wh Psi}
\wh{\Psi}(z) \equiv \wh{\Psi}(z;s):= \Psi(z;s) \begin{pmatrix} 
e^{\theta} & 0
\\
0 & e^{-\theta}
\end{pmatrix}\begin{pmatrix} 
0 &1
\\
-1 & 0
\end{pmatrix}^{\chi_L}, \qquad \theta= \frac{z^2}{4}+\frac{s}{2}z,
\end{equation} where $\chi_L$ is the characteristic function whose support is the left of the contour $\hat{\Gamma}_{r=1}$, see \cite[Eqs.(2.6),(3.2), p.27]{MR3849128}.
We also denote 
\begin{equation*}
\begin{split}
\wt{P}(z)& 
:= \begin{pmatrix}
(\sqrt{k}\zeta)^{\frac{c}{2}}   \Big( \frac{z-1}{z} \Big)^{ -\frac{c}{2}  }     & 0 
\\
(\sqrt{k}\zeta)^{\frac{c}{2}} \frac{ U k^{ -\frac{c+1}{2} } }{z-1}   \Big( \frac{z-1}{z} \Big)^{ -\frac{c}{2}  }    -(\sqrt{k}\zeta)^{-\frac{c}{2}} \frac{U}{\sqrt{k}\,\zeta}   \Big( \frac{z-1}{z} \Big)^{ \frac{c}{2}  }    & (\sqrt{k}\zeta)^{-\frac{c}{2}}   \Big( \frac{z-1}{z} \Big)^{ \frac{c}{2}  }   
\end{pmatrix}  \wh{\Psi}( \sqrt{k}\, \zeta; \sqrt{k}\,A ),
\end{split}
\end{equation*}
which gives 
\begin{equation*}
\begin{split}
\wt{P}(z)_{11}&=(\sqrt{k}\zeta)^{ \frac{c}{2} }   \Big( \frac{z-1}{z} \Big)^{ -\frac{c}{2}  }    \wh{\Psi}_{11},
\\
\wt{P}(z)_{21}&=\Big[ (\sqrt{k}\zeta)^{\frac{c}{2}} \frac{ U k^{ -\frac{c+1}{2} } }{z-1}   \Big( \frac{z-1}{z} \Big)^{ -\frac{c}{2}  }    -(\sqrt{k}\zeta)^{-\frac{c}{2}} \frac{U}{\sqrt{k}\,\zeta}   \Big( \frac{z-1}{z} \Big)^{ \frac{c}{2}  }    \Big] \wh{\Psi}_{11}+(\sqrt{k}\zeta)^{-\frac{c}{2}}   \Big( \frac{z-1}{z} \Big)^{ \frac{c}{2}     } \wh{\Psi}_{21}.
\end{split}
\end{equation*}
Let $E$ be the error matrix given in \cite[Subsection 4.3.6]{MR3849128}. 
It satisfies the asymptotic expansion 
\begin{equation}
E(z)=I+\frac{ E^{(1)} }{ k^{\frac12} }+  \frac{ E^{(2)} }{ k^{ \frac{1-c}{2} } } + \frac{ E^{(3)} }{ k^{1+\frac{c}{2}} }+O(\frac{1}{k}),
\end{equation}
where 
\begin{gather*}
E^{(1)} = \Big(-\frac{H}{\zeta} + \frac{H}{z-1} \Big) \sigma_3,
\qquad 
E^{(2)} = \Big[ -\Big( \frac{z-1}{z \zeta} \Big)^{-c} \Big( \frac{Z}{U \zeta} \Big) + \frac{ Z/U }{ z-1  } \Big] \sigma_+,
\\
E^{(3)} = \Big[ \Big( \frac{z-1}{ z \zeta } \Big)^{c} \frac{ (s-Y) U }{ \zeta^2 }-2 \frac{HU}{ (z-1)\zeta } + \frac{2}{3} \frac{ c^2(2H-c-s)+H-c-s }{ z-1 }+ \frac{ U(2H+Y-s) }{z-1}\Big) \Big] \sigma_-.
\end{gather*}
Here
\begin{equation}
\sigma_3=\begin{pmatrix}
1 & 0 
\\
0 & -1 
\end{pmatrix}, \qquad \sigma_+=\begin{pmatrix}
0 & 1
\\
0 & 0 
\end{pmatrix},  \qquad \sigma_-=\begin{pmatrix}
0 & 0
\\
1 & 0 
\end{pmatrix},
\end{equation}
see \cite[Eqs.(4.18),(4.22)]{MR3849128}.
Therefore we have 
\begin{equation} \label{E asymp}
E(z)=\begin{bmatrix}
1+\Big(-\frac{H}{\zeta} + \frac{H}{z-1} \Big)  \frac{ 1 }{ k^{\frac12}  } &   E^{(2)}_{12}  \frac{ 1 }{ k^{ \frac{1-c}{2} } } 
\smallskip 
\\
 E^{(3)}_{21} \frac{1}{ k^{1+\frac{c}{2}} } &  1-\Big(-\frac{H}{\zeta} + \frac{H}{z-1} \Big)  \frac{1 }{ k^{\frac12}  }
\end{bmatrix}+O(\frac{1}{k}).
\end{equation}

The function $\pi_k$ in \cite{MR3849128} is related to $P_k$ as 
\begin{equation} \label{Pk pi k}
 P_k(z)
 = a^k \, \pi_k(z/a). 
\end{equation}
By the asymptotic result in \cite[Subsection 4.4]{MR3849128}, we have
\begin{equation}
\pi_k(z)= z^k \Big(1-\frac{1}{z}\Big)^{ -\frac{c}{2} } [E(z) \wt{P}(z) ]_{11}. 
\end{equation}
Combining the above asymptotic behaviours of $E$ and $\wt{P}$, after long but straightforward computations, 
we obtain
\begin{equation}
\begin{split}
&\quad \pi_k(z) = z^k \Big(\frac{z-1}{z}\Big)^{  -c }   k^{ \frac{c}{4} } \zeta(z)^{ \frac{c}{2} }
\\
&\times \bigg[ \wh{\Psi}_{11} +\Big[ H\Big(\frac{1}{z-1} -\frac{1}{\zeta(z)} \Big)   \wh{\Psi}_{11}+ W  \Big(  \frac{ 1 }{ z-1  }  \Big( \frac{z-1}{z} \Big)^{ c  }    \zeta(z)^{ -c  }- \frac{1}{\zeta(z)}  \Big)   \wh{\Psi}_{21} \Big] \frac{ 1 }{ k^{ \frac12 } } +O( \frac{1}{k^{1+\frac{c}{2}} } ) \bigg].
\end{split}
\end{equation}

Note that by \eqref{def of A(z0)}, we have
\begin{equation}
\sqrt{k}\,\zeta(z)=\Big( z-1-\frac{(z-1)^2}{3} \Big) \sqrt{k}+O(1).
\end{equation}
Recall that 
\begin{equation}
\phi(z;z_0)= \frac{z-1}{z_0}-\log z, \quad \Big(z_0=\frac{t_c^2}{t^2}\Big) \qquad \hat{\varphi}(\lambda)=\log(t_c-\lambda^d)+\frac{\lambda^d}{t_c}-\log t_c 
\end{equation}
and 
\begin{equation}
\phi(z;z_0)= \frac12 \zeta^2(z;z_0)+A(z_0) \zeta(z;z_0), \qquad \zeta(1;z_0) \equiv 0, \qquad A(1)=0, 
\end{equation}
see \cite[Proposition 4.5]{MR3849128}.

By \eqref{def of A(z0)} and \eqref{Psi wh Psi}, we have
\begin{equation}
e^\theta= \exp\Big( \frac{k}{4} \zeta^2+\frac{ \sqrt{k}\,A }{2} \sqrt{k} \zeta \Big) = e^{ \frac{k}{2} \phi(z;z_0) },  \qquad  \wh{\Psi}_{11}( \sqrt{k}\, \zeta; \sqrt{k}\,A ) = e^{ \frac{k}{2} \phi(z;z_0) } \Psi_{11}( \sqrt{k}\, \zeta; \sqrt{k}\,A ). 
\end{equation}
Therefore by \eqref{Pk pi k} we obtain 
\begin{equation}
\begin{split}
&\quad \psi_k(z) =z^{k+c}  e^{\frac{k}{2} \phi(z/a;z_0) }   \Big(\sqrt{k} \, \zeta(z/a)\Big)^{ \frac{c}{2} } \Psi_{11}( \sqrt{k}\, \zeta(z/a;z_0); \sqrt{k}\,A(z_0) )
\\
&\times \bigg[ 1 +\Big[ H\Big(\frac{a}{z-a} -\frac{1}{\zeta(z/a)} \Big)   + W\Big(  \frac{ a }{ z-a  } \Big( \frac{z-a}{z} \Big)^{c}  \frac{1}{\zeta(z/a)^{ c  }}- \frac{1}{\zeta(z/a)}  \Big) \frac{  \Psi_{21}  }{ \Psi_{11}   } \Big] \frac{ 1 }{ k^{ \frac12 } } +O( \frac{1}{k^{1+\frac{c}{2}} } ) \bigg].
\end{split}
\end{equation}
A similar computation can be done for $z$ in other regions of $D$.
Note that 
\begin{equation}
\label{eq op relation}
P_{n,N}(z;a)=\Big(\frac{n}{N}\Big)^{\frac{n}{2}}P_{n,n}\Big(\sqrt{\tfrac{N}{n}}z,\sqrt{\tfrac{N}{n}}a\Big),
\end{equation}
see e.g. \cite[p.304]{MR3670735}.
Now Proposition~\ref{Thm_BEG op} follows from \eqref{def of SSr}. 
\end{proof}


\begin{lem} \label{Lem_Fk asymp}
Let
\begin{equation}
\begin{split}
F_k(z):=e^{\frac{k}{2} \phi(  1+\frac{z/a}{\sqrt{N}} ;z_0) }   \Big(\sqrt{k} \, \zeta\Big( 1+\frac{z/a}{\sqrt{N}} \Big)\Big)^{ \frac{c}{2} } \Psi_{11}\Big( \sqrt{k}\, \zeta\Big( 1+\frac{z/a}{\sqrt{N}};z_0\Big); \sqrt{k}\,A(z_0) \Big).
\end{split}
\end{equation}
Then we have
\begin{equation}
F_k(z)= e^{ \frac{z(z+2\SS)}{4} } z^{ \frac{c}{2} } \Psi_{11}(z;\SS) \Big( 1+\frac{ F_k^{(1)}(z) }{ \sqrt{N} } + \frac{ F_k^{(2)}(z) }{ N } +O(N^{ -\frac32 }) \Big),
\end{equation}
where $F_k^{(1)}$ is given by \eqref{Fk (1)} and 
\begin{align} \label{Fk (2)}
\begin{split}
 F_k^{(2)}(z) & = \Big( \frac{ z(4z^2+6\SS z+3\SS^2) }{24}+\frac{r}{2}z\Big)  \frac{c}{2} \frac{2z+\SS}{6}
\\
&+ \Big[ \frac{ z(4z^2+6\SS z+3\SS^2) }{24}+\frac{r}{2}z+ \frac{c}{2} \frac{2z+\SS}{6}\Big] \Big[ \frac{z(2z+\SS)}{6} \frac{ \pa_z \Psi_{11}}{ \Psi_{11} }+\Big( \frac{\SS^2}{12}+r\Big) \frac{ \pa_s \Psi_{11} }{ \Psi_{11}   } \Big] 
\\
&+  \frac12 \Big( \frac{ z(4z^2+6\SS z+3\SS^2) }{24}+\frac{r}{2}z\Big) ^2+ \Big(  \frac{z (z+\SS)(2z^2+2\SS z +\SS^2) }{ 16 }+r\,\frac{ z(z+\SS)}{4} \Big)  
\\
&+ \frac{c(c-2)}{8} \Big( \frac{2z+\SS}{6} \Big)^2+\frac{c}{2} \Big(\frac{ 7z^2+9\SS z+3\SS^2 }{36}+\frac{r}{6}\Big) 
\\
& +   \Big(\frac{ z(7z^2+9\SS z+3\SS^2) }{36}+\frac{r}{6}z\Big) \frac{ \pa_z \Psi_{11} }{ \Psi_{11} }+\Big(\frac{\SS^3}{36}+\frac{r}{6}\SS\Big) \frac{ \pa_s \Psi_{11}}{ \Psi_{11}  } 
\\
&+  \frac12\Big(\frac{z(2z+\SS)}{6}\Big)^2 \frac{ \pa_z^2 \Psi_{11} }{ \Psi_{11} }+\frac12\Big( \frac{\SS^2}{12}+r\Big)^2 \frac{ \pa_s^2 \Psi_{11} }{ \Psi_{11} }+\frac{z(2z+\SS)}{6} \Big( \frac{\SS^2}{12}+r\Big)\frac{ \pa_z \pa_s \Psi_{11} }{ \Psi_{11} } .
\end{split}
\end{align}
\end{lem}

\begin{proof}
Let $k=N+r$. By \eqref{def of A(z0)}, 
we have
\begin{equation}
\begin{split}
\sqrt{N+r} \, \zeta\Big(1+\frac{z/a}{\sqrt{N}};z_0\Big) 
&= z\Big[1 -\frac{2z+\SS}{6}\frac{1}{\sqrt{N}}+\Big(\frac{ 7z^2+9\SS z+3\SS^2 }{36}+\frac{r}{6}\Big)\frac{1}{N}+O(N^{-\frac32}) \Big] 
\end{split}
\end{equation}
and
\begin{equation}
\sqrt{N+r} \,A(z_0)=\SS-\Big( \frac{\SS^2}{12}+r \Big) \frac{1}{\sqrt{N}}+\Big( \frac{\SS^3}{36}+\frac{r}{6}\SS\Big) \frac{1}{N}+O(N^{-\frac32}).
\end{equation}
Note also that
\begin{equation}
\begin{split}
\frac{N+r}{2}\,\phi\Big( 1+\frac{z/a}{\sqrt{N}}; z_0 \Big) &=  \frac{z(z+2\SS)}{4} -\Big( \frac{ z(4z^2+6\SS z+3\SS^2) }{24}+\frac{r}{2}z\Big) \frac{1}{\sqrt{N}}
\\
&\quad + \Big(  \frac{z (z+\SS)(2z^2+2\SS z +\SS^2) }{ 16 }+r\,\frac{ z(z+\SS)}{4} \Big) \frac{1}{N}+O(N^{-\frac32}). 
\end{split}
\end{equation}
This gives 
\begin{align*}
&\quad e^{\frac{k}{2} \phi(  1+\frac{z/a}{\sqrt{N}} ;z_0) }  e^{ -\frac{z(z+2\SS)}{4} } =  1- \Big( \frac{ z(4z^2+6\SS z+3\SS^2) }{24}+\frac{r}{2}z\Big) \frac{1}{\sqrt{N}}  
\\
&+ \Big[ \frac12 \Big( \frac{ z(4z^2+6\SS z+3\SS^2) }{24}+\frac{r}{2}z\Big) ^2+ \Big(  \frac{z (z+\SS)(2z^2+2\SS z +\SS^2) }{ 16 }+r\,\frac{ z(z+\SS)}{4} \Big)  \Big] \frac{1}{N}+O(N^{-\frac32}).
\end{align*}
We have 
\begin{align*}
  \Big(\sqrt{k} \, \zeta\Big( 1+\frac{z/a}{\sqrt{N}} \Big)\Big)^{ \frac{c}{2} }  z^{ -\frac{c}{2} } & =  1-\frac{c}{2} \frac{2z+\SS}{6}\frac{1}{\sqrt{N}}
\\
&+\Big[ \frac{c(c-2)}{8} \Big( \frac{2z+\SS}{6} \Big)^2+\frac{c}{2} \Big(\frac{ 7z^2+9\SS z+3\SS^2 }{36}+\frac{r}{6}\Big) \Big] \frac{1}{N}+O(N^{-\frac32}).
\end{align*}
Finally, we have 
\begin{align*}
&\quad \Psi_{11}\Big( \sqrt{k}\, \zeta\Big( 1+\frac{z/a}{\sqrt{N}};z_0\Big); \sqrt{k}\,A(z_0) \Big) = \Psi_{11}(z;\SS) 
- \Big[ \frac{z(2z+\SS)}{6} \pa_z \Psi_{11}+\Big( \frac{\SS^2}{12}+r\Big) \pa_s \Psi_{11} \Big] \frac{1}{\sqrt{N}}
\\
&+ \Big[ \Big(\frac{ z(7z^2+9\SS z+3\SS^2) }{36}+\frac{r}{6}z\Big) \pa_z \Psi_{11}+\Big(\frac{\SS^3}{36}+\frac{r}{6}\SS\Big) \pa_s \Psi_{11} \Big] \frac{1}{N}
\\
&+ \Big[ \frac12\Big(\frac{z(2z+\SS)}{6}\Big)^2 \pa_z^2 \Psi_{11}+\frac12\Big( \frac{\SS^2}{12}+r\Big)^2 \pa_s^2 \Psi_{11}+\frac{z(2z+\SS)}{6} \Big( \frac{\SS^2}{12}+r\Big)\pa_z \pa_s \Psi_{11} \Big] \frac{1}{N}+O(N^{-\frac32}).
\end{align*}
Combining all of the above, lemma follows.
\end{proof}

\begin{lem} \label{Lem_Gk asymp}
Let 
\begin{equation*}
\begin{split}
&\quad G_k(z):=1 +\Big[ H(\SS_r)\Big(\frac{a}{z}\sqrt{N} -\frac{1}{\zeta(1+\frac{z/a}{\sqrt{N}})} \Big) 
\\
&+ W(\SS_r)\Big(  \frac{ a }{ z  }\sqrt{N} \Big( \frac{z}{a\sqrt{N}+z} \Big)^{c}  \frac{1}{\zeta(1+\frac{z/a}{\sqrt{N}})^{ c  }}- \frac{1}{\zeta(1+\frac{z/a}{\sqrt{N}})}  \Big) \frac{  \Psi_{21}( \sqrt{k}\, \zeta\Big( 1+\frac{z/a}{\sqrt{N}};z_0\Big); \sqrt{k}\,A(z_0) )  }{ \Psi_{11}( \sqrt{k}\, \zeta\Big( 1+\frac{z/a}{\sqrt{N}};z_0\Big); \sqrt{k}\,A(z_0) )   } \Big] \frac{ 1 }{ k^{ \frac12 } }.
\end{split}
\end{equation*}
Then we have
\begin{align*}
G_k(z)= 1+ \frac{ G_k^{(1)}(z) }{ \sqrt{N} } + \frac{ G_k^{(2)}(z) }{N}+O(N^{-\frac32}) ,
\end{align*}
where $G_k^{(1)}$ is given by \eqref{Gk (1)} and
\begin{align} \label{Gk (2)}
\begin{split}
\qquad &\quad G_k^{(2)}(z)=-\frac{\SS-z}{3z} \,r \,  \Big( H'(\SS)+W'(\SS)\frac{ \Psi_{21} }{ \Psi_{11} } \Big) + H(\SS) \Big( \frac{3z^2+5\SS z+2\SS^2}{ 36z }-\frac{r}{3z}\Big)
\\
&+W(\SS)  \frac{ \Psi_{21} }{ \Psi_{11} }  \Big[ \frac{c}{z} \Big( \frac{21z^2+7\SS z-2\SS^2}{36}+\frac{r}{3} \Big)+\frac{c(c-1)}{2z} \Big( \frac{2z+\SS}{3} \Big)^2 + \frac{3z^2+5\SS z+2\SS^2 }{ 36 z } -\frac{r}{3z}  \Big]
\\
&+ W(\SS) \frac{ \Psi_{21} }{ \Psi_{11} } \frac{\SS-z}{3z}   \Big[ \frac{z(2z+\SS)}{6} \Big( \frac{\pa_z \Psi_{11}}{ \Psi_{11} }-\frac{\pa_z \Psi_{21}}{ \Psi_{21} } \Big) +\Big( \frac{\SS^2}{12}+r\Big) \Big(  \frac{ \pa_s \Psi_{11} }{ \Psi_{11} }-\frac{ \pa_s \Psi_{21} }{ \Psi_{21} } \Big) \Big].
\end{split}
\end{align}
\end{lem}

\begin{proof}
Since 
\begin{align*}
\frac{1}{\zeta(1+\frac{z/a}{\sqrt{N}})}= \frac{\sqrt{N}}{z}+\frac{2z+\SS}{6z}+ \Big( -\frac{3z^2+5\SS z+2\SS^2 }{ 36 z } +\frac{r}{3z} \Big) \frac{1}{\sqrt{N}}+O(N^{-1}),
\end{align*}
we have
\begin{align*}
\Big( \frac{a}{z}\sqrt{N} -\frac{1}{\zeta(1+\frac{z/a}{\sqrt{N}})} \Big) \frac{1}{k^{\frac12}} &= \frac{\SS-z}{3z} \frac{1}{\sqrt{N}}+\Big( \frac{3z^2+5\SS z+2\SS^2}{ 36z }-\frac{r}{3z}\Big) \frac{1}{N}+O(N^{-\frac32})
\end{align*}
and
\begin{align*}
&\quad \frac{ a }{ z  }\sqrt{N} \Big( \frac{z}{a\sqrt{N}+z} \Big)^{c}  \frac{1}{\zeta(1+\frac{z/a}{\sqrt{N}})^{ c  }}- \frac{1}{\zeta(1+\frac{z/a}{\sqrt{N}})} = \frac{\SS-z}{3z}
\\
&+\Big[ \frac{c}{z} \Big( \frac{21z^2+7\SS z-2\SS^2}{36}+\frac{r}{3} \Big)+\frac{c(c-1)}{2z} \Big( \frac{2z+\SS}{3} \Big)^2 + \frac{3z^2+5\SS z+2\SS^2 }{ 36 z } -\frac{r}{3z}  \Big] \frac{1}{\sqrt{N}}+O(N^{-1}).
\end{align*}

We also have 
\begin{align*}
&\quad \frac{\Psi_{21}}{\Psi_{11}}\Big( \sqrt{k}\, \zeta\Big( 1+\frac{z/a}{\sqrt{N}};z_0\Big); \sqrt{k}\,A(z_0) \Big) = \frac{ \Psi_{21}(z,;\SS) }{ \Psi_{11}(z;\SS) } 
\\
&+ \frac{ \Psi_{21}(z;\SS) }{ \Psi_{11}(z;\SS) }  \Big[ \frac{z(2z+\SS)}{6} \Big( \frac{\pa_z \Psi_{11}}{ \Psi_{11} }-\frac{\pa_z \Psi_{21}}{ \Psi_{21} } \Big) +\Big( \frac{\SS^2}{12}+r\Big) \Big(  \frac{ \pa_s \Psi_{11} }{ \Psi_{11} }-\frac{ \pa_s \Psi_{21} }{ \Psi_{21} } \Big) \Big] \frac{1}{\sqrt{N}}+O(N^{-1}).
\end{align*}
Combining above asymptotic behaviours, we conclude the lemma. 
\end{proof}

\begin{proof}[Proof of Lemma~\ref{Lem_OP diff deri}]
Combining Proposition~\ref{Thm_BEG op} with Lemmas~\ref{Lem_Fk asymp} and \ref{Lem_Gk asymp}, we have 
\begin{align} \label{psi k asymp local}
\begin{split}
&\quad \psi_k \Big(a+\frac{z}{\sqrt{N}}\Big) = \Big(a+\frac{z}{\sqrt{N}}\Big)^{k+c}  e^{ \frac{z(z+2\SS)}{4} } z^{ \frac{c}{2} } \Psi_{11}(z;\SS) 
\\
& \times \Big( 1+ \frac{ F_k^{(1)}(z)+G_k^{(1)}(z) }{ \sqrt{N} } + \frac{F_k^{(1)}(z)G_k^{(1)}(z)+F_k^{(2)}(z)+ G_k^{(2)}(z) }{N} + O( \frac{1}{N^{1+\frac{c}{2}}}+\frac{1}{N} ) \Big),
\end{split}
\end{align}
where the $O( \frac{1}{N^{1+\frac{c}{2}}}+\frac{1}{N} )$ term does not depend on $r.$
Here, $F_k^{(2)}$ and $G_k^{(2)}$ $(j=1,2)$ are defined by \eqref{Fk (2)} and \eqref{Gk (2)}. 
Then lemma follows from straightforward computations. 
\end{proof}

\section{Fine asymptotic behaviours of the orthogonal norms} \label{Section_OP norms fine asymp}

By \cite[Proposition 7.1]{MR3280250}, we have 
\begin{equation} \label{hk wt hk}
    h_k=-\frac{1}{2\pi i} \frac{\Gamma(c+k+1)}{N^{c+k+1}}\frac{\widetilde{h}_k}{P_{k+1}(0)}, \qquad \widetilde{h}_k:=\int_\Gamma P_{k}(z)^2 \wt{w}_{k}(z)\, dz, \quad \wt{w}_{k}(z):=\Big(\frac{z-a}{z}\Big)^c\frac{e^{-Naz}}{z^k}.
\end{equation}

\begin{lem} \label{Lem_PN(0)}
As $N \to \infty$, we have 
\begin{equation}
\begin{split}
 P_{N+r,N}(0) =a^{N+r} e^{-N a^2} \Big(       \frac{ 1 }{ N^{ \frac{1-c}{2} } } W(\SS_r) +O(\frac{1}{N})  \Big).
\end{split}
\end{equation}
\end{lem}

\begin{proof}[Proof of Lemma~\ref{Lem_PN(0)}]
For $z \in \Omega_0 \setminus \D$, 
\begin{equation}
\pi_k(z)=e^{k g(z)} \Big( \frac{z-1}{z}\Big)^{ -\frac{c}{2} } [E(z) \wt{N}(z)]_{11}, \qquad \wt{N}(z) 
= 
\begin{pmatrix}
0 & \Big( \frac{z-1}{z}\Big)^{ -\frac{c}{2} } 
\\
-\Big( \frac{z-1}{z}\Big)^{ \frac{c}{2} }  &  \frac{ U k^{ -\frac{1+c}{2} } }{z-1} \Big( \frac{z-1}{z}\Big)^{ -\frac{c}{2} } 
\end{pmatrix},
\end{equation}
see \cite[p.30]{MR3849128}.
By \eqref{E asymp}, we have 
\begin{equation}
\begin{split}
 [E(z) \wt{N}(z)]_{11}  = - \Big(   W  \frac{1}{z-1}   \frac{ 1 }{ k^{ \frac{1-c}{2} } }  +O(\frac{1}{k}) \Big)  \Big( \frac{z-1}{z}\Big)^{ \frac{c}{2} } = -\Big( \frac{z-1}{z } \Big)^{ \frac{c}{2} } W  \frac{1}{z-1}   \frac{ 1 }{ k^{ \frac{1-c}{2} } }  +O(\frac{1}{k}).  
\end{split}
\end{equation}

Therefore we have
\begin{equation}
\pi_N(z)=-e^{k g(z)} \Big( \frac{ W(\SS) }{z-1}   \frac{ 1 }{ N^{ \frac{1-c}{2} } }  +O(\frac{1}{k})  \Big).
\end{equation}
In particular,
\begin{equation}
\pi_k(0)=e^{-k/z_0 } \Big(       \frac{ W(\SS) }{ k^{ \frac{1-c}{2} } }  +O(\frac{1}{k})  \Big)= e^{-N a^2} \Big(    \frac{ W(\SS)   }{ k^{ \frac{1-c}{2} } }  +O(\frac{1}{k})  \Big), \qquad  P_N(0)=a^N e^{-N a^2} \Big(   \frac{ W(\SS) }{ N^{ \frac{1+\gamma}{2} } }  +O(\frac{1}{N})  \Big).
\end{equation}
Using \eqref{eq op relation}, this gives 
\begin{equation}
\begin{split}
 P_{N+r,N}(0) &= \Big( \frac{N+r}{N}\Big)^{ \frac{N+r}{2} } P_{N+r,N+r}\Big(0, \sqrt{\tfrac{N}{N+r}}a\Big)
=a^{N+r} e^{-N a^2} \Big(       \frac{ W(\SS_r) }{ N^{ \frac{1+\gamma}{2} } }  +O(\frac{1}{N})  \Big).
\end{split}
\end{equation}
\end{proof}

\begin{lem} \label{Lem_hn tilde}
As $N \to \infty$, we have 
\begin{equation}
\wt{h}_{N+r,N} = -2\pi i \, e^{-Na^2} a^{N+r+1} \,  \Big( \frac{1}{N^{\frac{1-c}{2}}}W(\SS_r)+O( \frac{1}{N} )\Big) .
\end{equation}
\end{lem}

\begin{proof}[Proof of Lemma~\ref{Lem_hn tilde}]
Let $\wt{Y}$ be the matrix $Y$ in \cite{MR3280250}. 
Then we have
\begin{equation}
\widetilde{h}_k=-2\pi i  \lim_{z\to\infty}z^{k+1}[\wt{Y}(z)]_{12},
\end{equation}
see \cite[Eq.(7.2)]{MR3280250}.
Let
\begin{equation}
\omega_k(z) := \Big(\frac{z-1}{z}\Big)^c\frac{e^{-Na^2z}}{z^k} = a^k\,  \wt{\omega}_{k}(az),
\end{equation}
see \cite[p.18]{MR3849128}.
Therefore by the change of variables, we have
\begin{equation}
\begin{split}
[Y(z)]_{12}& := \frac{1}{2\pi i } \int \frac{ \pi_k(z') }{ z'-z } \omega_k(z')\,dz' 
=\frac{1}{2\pi i } \int \frac{ P_k(z') }{ z'-az }\,  \wt{\omega}_{k}(az)\,dz'=  [\wt{Y}(az)]_{12},
\end{split}
\end{equation}
where $Y$ is a matrix given in \cite[p.19]{MR3849128}.

Note that by the transforms in \cite{MR3849128}, we have
\begin{equation}
\begin{split}
 Y(z)  
 &= e^{\frac{kl}{2}\sigma_3}E(z) \Big( \frac{z-1}{z} \Big)^{ -\frac{c}{2} \sigma_3 } \begin{pmatrix} 0 & 1 
\\
-1 &0 
\end{pmatrix}  e^{-\frac{kl}{2}\sigma_3}e^{kg(z)\sigma_3} \Big(\frac{z-1}{z}\Big)^{-\frac{c}{2}\sigma_3}. 
\end{split}
\end{equation}
Thus by \eqref{E asymp} we have
\begin{equation}
[Y(z)]_{12}=\Big( \frac{1}{k^{\frac{1-c}{2}}}\frac{W}{z-1}+O(k^{-1})\Big)\Big(\frac{z-1}{z}\Big)^{c} e^{kl-kg(z)},
\end{equation}
which gives 
\begin{equation}
[\wt{Y}(z)]_{12}=[Y(z/a)]_{12}=\Big( \frac{1}{k^{\frac{1-c}{2}}}\frac{W}{z/a-1}+O(k^{-1})\Big)\Big(\frac{z-a}{z}\Big)^{c} e^{kl-kg(z/a)}.
\end{equation}
Therefore we obtain
\begin{align*}
\wt{h}_k & = -2\pi i  \lim_{z\to\infty}z^{k+1}[\wt{Y}(z)]_{12}
= -2\pi i  \lim_{z\to\infty}z^{k+1}\Big( \frac{1}{k^{\frac{1-c}{2}}}\frac{W}{z/a-1}+O(k^{-1})\Big)\Big(\frac{z-a}{z}\Big)^{c} e^{kl} \Big( \frac{a}{z} \Big)^{k}
\\
&= -2\pi i  \lim_{z\to\infty} \Big( \frac{1}{k^{\frac{1-c}{2}}}\frac{W}{1/a-1/z}+O(k^{-1})\Big)\Big(\frac{z-a}{z}\Big)^{c} e^{kl} a^{k}
= -2\pi i   \Big( \frac{W}{k^{\frac{1-c}{2}}}+O(k^{-1})\Big) e^{kl} a^{k+1}.
\end{align*}
We have shown that
\begin{equation}
\wt{h}_{N,N}= -2\pi i \,e^{-Na^2} a^{N+1} \,  \Big( \frac{1}{N^{\frac{1-c}{2}}}W(\SS)+O(k^{-1})\Big) .
\end{equation}
Note also that by \eqref{eq op relation}, 
\begin{equation}
\label{eq hn relation}
\widetilde{h}_{n,N}=\Big(\frac{n}{N}\Big)^{\frac{n+1}{2}}\widetilde{h}_{n,n}\Big(\sqrt{\tfrac{N}{n}}a\Big).
\end{equation}
Thus 
\begin{align}
\begin{split}
\wt{h}_{N+r,N} & = \Big(\frac{N+r}{N}\Big)^{\frac{N+r+1}{2}}\widetilde{h}_{N+r,N+r}\Big(\sqrt{\tfrac{N}{N+r}}a\Big) \\
&= -2\pi i \, e^{-Na^2} a^{N+r+1} \,  \Big( \frac{1}{N^{\frac{1-c}{2}}}W(\SS_r)+O(N^{-1})\Big) .
\end{split}
\end{align}
\end{proof}

\begin{proof}[Proof of Lemma~\ref{Lem_orthogonal norm diff}]
By combining \eqref{hk wt hk} with Lemmas~\ref{Lem_PN(0)} and ~\ref{Lem_hn tilde}, we have
\begin{equation}
\begin{split}
  h_{N+r} 
&= \frac{\Gamma(c+N+r+1)}{ N^{c+N+r+1}}  \frac{   W(\SS_r) +O(N^{-\frac{1+c}{2}}) }  {    W(\SS_{r+1}) +O(N^{-\frac{1+c}{2}}) }, \qquad h_N/h_{N-1}=1+O(1/N).
\end{split}
\end{equation}
We also have
\begin{align}
\frac{N+c}{N} h_{N-1}-h_N &= \frac{ \Gamma(N+c) }{  N^{N+c} } \Big(    \frac{   W(\SS_{-1}) +O(N^{-\frac{1+c}{2}}) }  {    W(\SS_{0}) +O(N^{-\frac{1+c}{2}}) }  -  \frac{   W(\SS_0) +O(N^{-\frac{1+c}{2}}) }  {    W(\SS_{1}) +O(N^{-\frac{1+c}{2}}) } \Big),
\\
\frac{N+c+1}{N} h_{N}-h_{N+1} &= \frac{ \Gamma(N+c+1) }{  N^{N+c+1} } \Big(    \frac{   W(\SS_0) +O(N^{-\frac{1+c}{2}}) }  {    W(\SS_{1}) +O(N^{-\frac{1+c}{2}}) }  -  \frac{   W(\SS_1) +O(N^{-\frac{1+c}{2}}) }  {    W(\SS_{2}) +O(N^{-\frac{1+c}{2}}) }  \Big). 
\end{align}
Note here that 
\begin{align}
\begin{split}
&\quad \frac{ W(\SS_{r-1}) }{ W(\SS_{r}) } -\frac{ W(\SS_{r}) }{ W(\SS_{r+1}) } = \Big[ \frac{W''(\SS)}{W(\SS)} -\Big( \frac{W'(\SS)}{W(\SS)} \Big)^2 \Big]\frac{1}{N}
\\
&+\Big[ \frac32 \frac{W'}{W} -\SS \Big( \frac{W'}{W} \Big)^2 - \Big( \frac{W'}{W} \Big)^3+\SS \frac{W''}{W} +r \Big\{ 3 \frac{W'' W'}{W^2} -2 \Big( \frac{W'}{W} \Big)^3 -\frac{W'''}{W} \Big\}   \Big] \frac{1}{N\sqrt{N}}+O(N^{-2}).
\end{split}
\end{align}
Then lemma follows from the Stirling's formula.
\end{proof}

\section{Asymptotic behaviours of the $1$-point function} \label{appendix_asymptotic 1pt}

Let us recall that 
\begin{align}
\bp_z \wt{R}^0(z) & = -\frac{1}{\sqrt{2\pi}} e^{-\frac12 (z+\bar{z})^2 }, 
\\
\bp_z \wt{R}^1(z) & = -\frac{1}{\sqrt{2\pi}} e^{-\frac12 (z+\bar{z})^2 } +  \frac{1}{\sqrt{2\pi}} \erfc( \tfrac{z}{\sqrt{2}} ) e^{-|z|^2 } \Big(  e^{-\frac{\bar{z}^2}{2}} + \sqrt{ \frac{\pi}{2} } z  \erfc( \tfrac{ \bar{z} }{\sqrt{2}} )   \Big).
\end{align}
Note that as $z \to \infty$, 
\begin{equation}
\Psi_{11}(z;s)=z^{-\frac{c}{2}} e^{ -\frac{z^2}{4}-\frac{s z}{2} } \Big(1+O(\frac{1}{z})\Big), \qquad \Psi_{21}(z;s)=z^{-\frac{c}{2}} e^{ -\frac{z^2}{4}-\frac{s z}{2} } O(\frac{1}{z}). 
\end{equation}
see \cite[p.9]{MR3849128}.
Using these, we obtain that as $z \to \infty$, 
\begin{align*}
e^{ -|z|^2-\frac{ (z+\SS)^2+(\bar{z}+\SS)^2 }{4} } \,|z|^c |\Psi_{11}(z;\SS)|^2 &= e^{-\frac12(z+\bar{z}+\SS)^2} \Big(1+O(\frac{1}{z})\Big)
\end{align*}
and 
\begin{align*}
 W\frac{\Psi_{21}(z;s)}{\Psi_{11}(z;s)} \Big( \SS+\frac{Z}{\overline z}-z-\frac{W'}{W}-\AAA \BBB_1 - \frac{ \pa_s \Psi_{11}(0,\SS) }{ \Psi_{11}(0,\SS)   } \Big)-\frac{W^2}{\overline z}\frac{c+Z}{Y} \Big| \frac{\Psi_{21}(z;s)}{\Psi_{11}(z;s)} \Big|^2   +Z= O(1).
\end{align*}
Therefore by Proposition~\ref{Prop_asymptotic derivative} we obtain 
\begin{equation}
 \bp \wt{R}^c(z) =O(1) e^{-\frac12(z+\bar{z}+\SS)^2}, \qquad z \to \infty. 
\end{equation}
We mention that this asymptotic behaviour can be directly checked for $c=0,1.$

We now consider the asymptotic behaviour near the origin. 
Using the Lax pair given in \cite[Eq.(2.1)]{MR3849128}, we have 
$$
\frac{d}{dz} \Psi=\Big( A_0 \frac{1}{z} +\dots \Big) \Psi, \qquad A_0 = \begin{pmatrix}
Z+c & Z
\\
UY & U Y
\end{pmatrix}
\begin{pmatrix}
-\frac{c}{2} & 0
\\
0 & \frac{c}{2}
\end{pmatrix} \begin{pmatrix}
Z+c & Z
\\
UY & U Y
\end{pmatrix}^{-1}.
$$
Then it follows from  
\begin{equation}
\Psi=  \begin{pmatrix}
Z+c & Z
\\
UY & U Y
\end{pmatrix}\big(I+O(z)\big) \begin{pmatrix}
z^{-\frac{c}{2} } & 0 
\\
0 & z^{ \frac{c}{2} }
\end{pmatrix} \cdot (\textup{Regular}), \qquad (z \to 0)
\end{equation}
that 
\begin{equation}
\Psi_{11}(z;\SS)=\frac{Z+c}{z^{c/2}}(1+O(z)),\qquad \Psi_{21}(z;\SS)=\frac{UY}{z^{c/2}}(1+O(z)), \qquad (z \to 0).
\end{equation}
Therefore 
\begin{equation}
\frac{ \Psi_{21}(0;\SS) }{ \Psi_{11}(0;\SS) } = \frac{ U Y }{ Z+c } =\frac{Y}{W} \frac{ Z }{Z+c}
\end{equation}
By \cite[Eq.(2.1)]{MR3849128}, we have
\begin{align}
\Psi'_{11}& =-\Big(\frac{z+s}{2}+\frac{Z+c/2}{z}\Big)\Psi_{11}+\Big(\frac{Z}{U}+\frac{1}{z}\frac{Z(Z+c)}{YU }\Big)\Psi_{21},
\\
\Psi'_{21}&=-(U+\frac{UY}{z})\Psi_{11}+(\frac{z+s}{2}+\frac{Z+c/2}{z})\Psi_{21}.
\end{align}
This gives 
\begin{equation}
\frac{\Psi_{21}\Psi'_{11}-\Psi_{11}\Psi'_{21}}{\Psi_{11}^2}= \Big(\frac{Z}{U}+\frac{1}{z}\frac{Z(Z+c)}{YU }\Big) \frac{\Psi_{21}^2}{\Psi_{11}^2}-2\Big(\frac{z+s}{2}+\frac{Z+c/2}{z}\Big)\frac{\Psi_{21}}{\Psi_{11}}+\Big(U+\frac{UY}{z}\Big).
\end{equation}
Combining the above, after simplifications, we obtain that as $z\to 0,$
\begin{align}
 \frac{ \Psi_{21}(z;\SS) }{ \Psi_{11}(z;\SS) } & = \frac{ \Psi_{21}(0;\SS) }{ \Psi_{11}(0;\SS) }+\frac{\Psi_{21}(0;\SS)\Psi'_{11}(0;\SS)-\Psi_{11}(0;\SS)\Psi'_{21}(0;\SS)}{\Psi_{11}(0;\SS)^2}z+O(z^2)
\\
&=\frac{ U Y }{ Z+c }+z\Big[\frac{Z}{U} \Big(\frac{ U Y }{ Z+c }\Big)^2-s\frac{ U Y }{ Z+c }+U\Big]+O(z^2).
\end{align}

On the other hand, by \eqref{Psi pa s 0}, we have 
\begin{align*}
&\quad   W\frac{\Psi_{21}(z;\SS)}{\Psi_{11}(z;\SS)} \Big(\frac{Z}{\overline z}- \frac{W}{\overline z}\frac{c+Z}{Y}  \overline{\frac{\Psi_{21}(z;\SS)}{\Psi_{11}(z;\SS)}} -z+\frac{Z+c}{Y}-\AAA \BBB_1 - \frac{ \pa_s \Psi_{11}(0,\SS) }{ \Psi_{11}(0,\SS)   } \Big)   +Z
\\
&= \frac{1}{\bar{z}} W \Big(  Z\frac{\Psi_{21}(z;\SS)}{\Psi_{11}(z;\SS)} -W \frac{Z+c}{Y} \Big| \frac{\Psi_{21}(z;\SS)}{\Psi_{11}(z;\SS)} \Big|^2 \Big)+  W\frac{\Psi_{21}(z;\SS)}{\Psi_{11}(z;\SS)}   \Big( -z+\frac{Z+c}{Y}-\AAA \BBB_1 -  Y \frac{Z}{Z+c} \Big)   +Z
\end{align*}
Note that
\begin{align*}
W \Big( Z\frac{\Psi_{21}(0;\SS)}{\Psi_{11}(0;\SS)} -W \frac{Z+c}{Y} \Big| \frac{\Psi_{21}(0;\SS)}{\Psi_{11}(0;\SS)} \Big|^2 \Big) &=  W Z\frac{Y}{W} \frac{ Z }{Z+c} -W^2 \frac{Z+c}{Y} \Big( \frac{Y}{W} \frac{ Z }{Z+c} \Big)^2 =0. 
\end{align*}
Thus we have 
\begin{align}
 Z\frac{\Psi_{21}(z;\SS)}{\Psi_{11}(z;\SS)} -W \frac{Z+c}{Y} \Big| \frac{\Psi_{21}(z;\SS)}{\Psi_{11}(z;\SS)} \Big|^2 
&= -Z\Big[\frac{Z}{U} \Big(\frac{ U Y }{ Z+c }\Big)^2-s\frac{ U Y }{ Z+c }+U\Big] \bar{z}+O(z^2).
\end{align}
Therefore we obtain
\begin{equation}
\frac{1}{\bar{z}} W \Big(  Z\frac{\Psi_{21}(z;\SS)}{\Psi_{11}(z;\SS)} -W \frac{Z+c}{Y} \Big| \frac{\Psi_{21}(z;\SS)}{\Psi_{11}(z;\SS)} \Big|^2 \Big)=-ZW\Big(W(\frac{ U Y }{ Z+c })^2-s\frac{ U Y }{ Z+c }+U\Big)+O(z).
\end{equation}
On the other hand, 
\begin{align*}
W\frac{\Psi_{21}(0;\SS)}{\Psi_{11}(0;\SS)}   \Big( \frac{Z+c}{Y}-\AAA \BBB -  Y \frac{Z}{Z+c} \Big)   +Z 
&= 2Z- Y \frac{ Z }{Z+c}\Big( \AAA \BBB +  Y \frac{Z}{Z+c} \Big).
\end{align*}
By Proposition~\ref{Prop_asymptotic derivative}, this gives that as $z \to 0$, 
\begin{equation}
\begin{split}
\bp_z \wt{R}^c(z) & =  \frac{ \AAA \mathfrak{D} }{ \sqrt{2\pi} }\,e^{-\frac12 (z+\bar{z}+\SS)^2 }+o(1).
\end{split}
\end{equation}
where 
\begin{equation}
\mathfrak{D}= -ZW\Big(W(\frac{ U Y }{ Z+c })^2-\SS \frac{ U Y }{ Z+c }+U\Big)+  2Z- Y \frac{ Z }{Z+c}\Big( \AAA \BBB +  Y \frac{Z}{Z+c} \Big).
\end{equation}

\section{Scaling limits for the bulk case} \label{appendix_bulk}

Here, we consider the case that $a \in (0,1).$
Let us first briefly recall the strong asymptotics of $P_k$ from \cite{MR3670735}. Let 
\begin{equation}
\phi_A(\zeta):=a(\zeta-a)-\log \frac{\zeta}{a}=\frac{az}{\sqrt{N}}-\log \Big( 1+\frac{z}{a\sqrt{N}} \Big).
\end{equation}
Note that as $N \to \infty,$
\begin{equation}\label{phi_A asy}
\phi_A(\zeta)=\frac{a^2-1}{a}\frac{z}{\sqrt{N}}+\frac{z^2}{2a^2} \frac1N+O(N^{-\frac32}).
\end{equation}
Let us also write 
\begin{equation}\label{phi A hat}
\wh{\phi}_A(\zeta):=\frac{N}{N-1}a(\zeta-a)-\log \frac{\zeta}{a}=\phi_A(\zeta)+\frac{1}{N-1}a(\zeta-a).
\end{equation}

By \cite[Theorem 3]{MR3670735},  for $\zeta$ in a neighbourhood of $a$, we have
\begin{align}\label{PN LY asy}
\begin{split}
P_N(\zeta)&=\zeta^N \Big( \frac{\zeta}{\zeta-a} \Big)^c
 \Big[ 1-(-N\phi_A(\zeta))^c e^{N \phi_A(\zeta) } \Big(  \hat{f}(-N\phi_A(\zeta)) +O(\tfrac1N)\Big)    +O(\tfrac{1}{N^\infty}) \Big]
\end{split}
\end{align}
and 
\begin{align}\label{PN1 LY asy}
\begin{split}
P_{N-1}(\zeta)&=\zeta^{N-1} \Big( \frac{\zeta}{\zeta-a} \Big)^c 
\\
&\times \Big[ 1-(-(N-1)\wh{\phi}_A(\zeta))^c e^{ (N-1) \wh{\phi}_A(\zeta) } \Big(  \hat{f}(-(N-1)\wh{\phi}_A(\zeta)) +O(\tfrac1N)\Big)    +O(\tfrac{1}{N^\infty}) \Big],
\end{split}
\end{align}
where the error bound $O(\tfrac{1}{N^\infty})$ means that $O(\tfrac{1}{N^k})$ for all $k>0$. 
Here 
\begin{equation}
\hat{f}(z):=-\frac{1}{2\pi i} \int_{ \mathcal{L} } \frac{e^s}{ s^c(s-z) }\,ds, 
\end{equation}
where the integration contour $\LL$ begins at $-\infty$, encircles the origin once in the counter-clockwise direction and returns to $-\infty$. 
Note that as $z \to \infty$,
\begin{equation}\label{f hat z inf}
\hat{f}(z)=\frac{1}{2\pi i} \int_{ \mathcal{L} } s^{-c}e^s\,ds \cdot \frac{1}{z}+O(|z|^{-2})=\frac{1}{\Gamma(c)} \frac{1}{z}+O(|z|^{-2}).
\end{equation}

For general $c>-1$, we present an alternative derivation of \eqref{K ML da} and \eqref{K R wh bulk} by virtue of the Christoffel-Darboux identity. Recall that $P$ denotes the regularised incomplete gamma function.

\begin{thm}\label{Thm_largeN pc} \emph{(Large-$N$ limit for the bulk case)}
For each $c>-1$, we have 
\begin{equation}
\wt{R}_{\textup{bulk}}^c(z)=P(c,|z|^2).
\end{equation}
\end{thm}

\begin{proof}
By \eqref{PN LY asy}, we have 
\begin{equation}
\label{psiN LY asy} 
	\psi_N(\zeta)=\zeta^{N+c} \Big[ 1-(-N\phi_A(\zeta))^c e^{N \phi_A(\zeta) } \Big(  \hat{f}(-N\phi_A(\zeta)) +O(\tfrac1N)\Big)    +O(\tfrac{1}{N^\infty}) \Big].
\end{equation}
Differentiating \eqref{psiN LY asy}, we have
\begin{align*}
\psi_N'(\zeta)&=(N+c)\zeta^{N+c-1}(1+O(\tfrac{1}{N^\infty}) )-(N+c)\zeta^{N+c-1}(-N\phi_A(\zeta))^c e^{N \phi_A(\zeta) } \Big(  \hat{f}(-N\phi_A(\zeta)) +O(\tfrac1N)\Big)
\\
&+N\zeta^{N+c-1}(a\zeta-1)(-N\phi_A(\zeta))^c e^{N \phi_A(\zeta) } \Big(  \big(\tfrac{-c}{N\phi_A(\zeta)}-1\big)\big(\hat{f}(-N\phi_A(\zeta))+O(\tfrac1N)\big)+\hat{f}'(-N\phi_A(\zeta)) \Big),
\end{align*}
where we have used $\phi_A'(\zeta)=(a\zeta-1)/\zeta.$ 
Rearranging the terms using \eqref{phi_A asy} and \eqref{f hat z inf}, we have
\begin{align*}
\psi_N'(\zeta)&=(N+c)\zeta^{N+c-1}(1+O(\tfrac{1}{N^\infty}) )
-Na\zeta^{N+c}(-N\phi_A(\zeta))^c e^{N \phi_A(\zeta) }\big( \hat{f}(-N\phi_A(\zeta))+O(\tfrac1N)\big).
\end{align*}
Then it follows from 
\begin{equation} \label{zeta power exp}
\zeta^{N+c}= a^{N+c} e^{\frac{1}{a}\sqrt{N}z- \frac{z^2}{2a^2}} \cdot (1+o(1))
\end{equation}
and \eqref{phi_A asy} that 
$$
\psi_N'(\zeta)=a^{N+c-1}N \Big[   e^{\frac{1}{a}\sqrt{N}z- \frac{z^2}{2a^2}} \cdot(1+o(1)) - \Big(  \hat{f}(-N\phi_A(\zeta)) +O(\tfrac1N)\Big) a^2 ( \tfrac{1-a^2}{a}\sqrt{N}z ) ^{c} e^{ a\sqrt{N}z }  \cdot(1+o(1))  \Big].
$$

Now let us compute the asymptotic of $\psi_{N}(\zeta)-\zeta \psi_{N-1}(\zeta).$
By \eqref{psiN LY asy} and \eqref{PN1 LY asy}, \begin{align*}
\psi_{N}(\zeta)-\zeta \psi_{N-1}(\zeta)&= \zeta^{N+c}(-N{\phi}_A(\zeta))^c e^{N {\phi}_A(\zeta) } \Big(  \hat{f}(-N{\phi}_A(\zeta)) +O(\tfrac1N)\Big)\\
	&-\zeta^{N+c}(-(N-1)\wh{\phi}_A(\zeta))^c e^{(N-1) \wh{\phi}_A(\zeta) } \Big(  \hat{f}(-(N-1) \wh{\phi}_A(\zeta)) +O(\tfrac1N)\Big).
\end{align*}
Note that by \eqref{phi A hat}, 
\begin{align*}
(N-1) \wh{\phi}_A(\zeta)- N\phi_A(\zeta) &= a(\zeta-a)-\phi_A(\zeta) = \frac{z}{a} \frac{1}{\sqrt{N}}+O(\frac1N). 
\end{align*}
Using this, we have 
\begin{align*}
\Big( \frac{N-1}{N}\frac{\hat{\phi}_A(\zeta)}{\phi_A(\zeta)} \Big)^c e^{(N-1) \wh{\phi}_A(\zeta)- N\phi_A(\zeta) } \frac{ \hat{f}(-(N-1)\hat{\phi}_A(\zeta)) }{ \hat{f}(-N{\phi}_A(\zeta)) }=1+\frac{z}{a} \frac{1}{\sqrt{N}}+O(\frac1N). 
\end{align*}
This gives that 
\begin{align*}
\psi_{N}(\zeta)-\zeta \psi_{N-1}(\zeta)&=\zeta^{N+c}(-N\phi_A(\zeta))^c e^{N \phi_A(\zeta) }\hat{f}(-N\phi_A(\zeta)) \frac{1}{a} \frac{1}{\sqrt{N}} \cdot (z+O(\tfrac{1}{\sqrt{N}})). 
\end{align*}
Then it again follows from \eqref{zeta power exp} and \eqref{phi_A asy} that 
\begin{equation}
\psi_{N}(\zeta)-\zeta \psi_{N-1}(\zeta)= a^{N+c-1} e^{a\sqrt{N}z} 
( \tfrac{1-a^2}{a}\sqrt{N}z ) ^{c} \hat{f}(-N\phi_A(\zeta)) \frac{1}{\sqrt{N}}\cdot (z+o(1)).
\end{equation}

Note that by \eqref{PN LY asy}, we have
\begin{equation}
	\frac{P_{N+1}(a)}{P_N(a) } = a+o(1).
\end{equation}
Using the above asymptotic behaviours, we obtain 
\begin{align}
\begin{split}
	\RN{1}_N^c(z)
	&= \frac{e^{ -N a^2 } }{N} \frac{ a^{2N+2c-2}\,( \tfrac{1-a^2}{a}\sqrt{N} ) ^{c-1}  }{ \tfrac{N+c}{N}h_{N-1}-h_{N}  }  
	\\
	& \times \Big( e^{(\frac{1}{a}-a)\sqrt{N}\bar{z}- \frac{\bar{z}^2}{2a^2}}-\hat{f}(-N\phi_A(\overline{\zeta}))a^2 ( \tfrac{1-a^2}{a}\sqrt{N}\bar{z} ) ^{c}  \Big) \cdot (z^c+o(1))
	\end{split}
\end{align}
and
\begin{align}
\begin{split}
	\RN{2}^c_N(z)
	&= \frac{e^{ -N a^2 }}{N}  \frac{h_N/h_{N-1}\,a^{2N+2c} ( \tfrac{1-a^2}{a}\sqrt{N} ) ^{c-1}}{ \tfrac{N+c+1}{N} h_N-h_{N+1}    }
	\\
	& \times  \Big(  e^{(\frac{1}{a}-a)\sqrt{N}\bar{z}- \frac{\bar{z}^2}{2a^2}}-\hat{f}(-N\phi_A(\overline{\zeta})) ( \tfrac{1-a^2}{a}\sqrt{N}\bar{z} ) ^{c} \Big)  \cdot (z^c+o(1)).
	\end{split}
\end{align}
Combining above equations with 
\begin{equation} \label{f hat phiA}
\hat{f}(-N\phi_A(\zeta)) = \frac{1}{\Gamma(c)} \Big( \frac{1}{\tfrac{1-a^2}{a} \sqrt{N}z} +o(1)\Big),
\end{equation}
we obtain that for $\re z<0,$
\begin{equation}
	\bp_z \wt{R}_N(z) = C_N(a)\, \Big( \frac{z\,|z|^{2c-2} e^{-|z|^2}}{\Gamma(c)} +o(1) \Big),
\end{equation}
where
$$
C_N(a):=\frac{e^{ -N a^2 }}{N}  a^{2N+2c} ( \tfrac{1-a^2}{a}\sqrt{N} ) ^{2c-2}  \Big(  \frac{ h_N/h_{N-1} }{ \tfrac{N+c+1}{N} h_N-h_{N+1}    } -\frac{ 1 }{ \tfrac{N+c}{N}h_{N-1}-h_{N}  }  \Big).
$$
We remark here that the case $\re z>0$ follows from the case $\re z<0$ since in the end, the limiting point process has the rotation invariance, see \cite[Section 5]{ameur2018random}.

Since $\wt{R}_N$ has a non-trivial limit $\wt{R}$, the existence of the limit
\begin{equation}
C(a):=\lim_{N\to\infty} C_N(a),\qquad (C(a)\not=0)
\end{equation}
follows. Therefore we obtain that 
\begin{equation}
\wt{R}_{\textup{bulk}}^c(z)=C(a)\,P(c,|z|^2)+\wh{C}(a),
\end{equation}
where $\wh{C}(a) \in \R$ is some constant. 

We now recall from the general theory on determinantal point process that 
\begin{align} \label{R bulk asymp}
	\wt{R}_{\textup{bulk}}^c(z) &\to 1, \qquad \qquad (z\to \infty),
	\\
		\wt{R}_{\textup{bulk}}^c(z)&=O(|z|^{2c}) , \qquad (z\to 0),
\end{align}
see \cite[Theorem 1.4]{ameur2018random}. This behaviour implies that $C(a)=1$ and $\wh{C}(a)=0,$ which completes the proof.  
\end{proof}

\bibliographystyle{abbrv}
\bibliography{RMTbib}

\begin{thebibliography}{10}

\bibitem{akemann2011oxford}
G.~Akemann, J.~Baik, and P.~Di~Francesco~(Editors).
\newblock {\em The {O}xford {H}andbook of {R}andom {M}atrix {T}heory}.
\newblock Oxford University Press, Oxford, 2011.

\bibitem{MR4229527}
G.~Akemann, S.-S. Byun, and N.-G. Kang.
\newblock A non-{H}ermitian generalisation of the {M}archenko-{P}astur
  distribution: From the circular law to multi-criticality.
\newblock {\em Ann. Henri Poincar\'{e}}, 22(4):1035--1068, 2021.

\bibitem{akemann2021scaling}
G.~Akemann, S.-S. Byun, and N.-G. Kang.
\newblock Scaling limits of planar symplectic ensembles.
\newblock {\em SIGMA Symmetry Integrability Geom. Methods Appl.}, 18:Paper No.
  007, 40, 2022.

\bibitem{MR3845296}
G.~Akemann, M.~Cikovic, and M.~Venker.
\newblock Universality at weak and strong non-{H}ermiticity beyond the elliptic
  {G}inibre ensemble.
\newblock {\em Comm. Math. Phys.}, 362(3):1111--1141, 2018.

\bibitem{akemann2001qcd3}
G.~Akemann, D.~Dalmazi, P.~Damgaard, and J.~Verbaarschot.
\newblock {QCD}3 and the replica method.
\newblock {\em Nucl.Phys. B}, 601(1-2):77--124, 2001.

\bibitem{akemann2021skew}
G.~Akemann, M.~Ebke, and I.~Parra.
\newblock Skew-orthogonal polynomials in the complex plane and their
  {B}ergman-like kernels.
\newblock {\em Comm. Math. Phys.}, 389:621--659, 2022.

\bibitem{MR1982915}
G.~Akemann and G.~Vernizzi.
\newblock Characteristic polynomials of complex random matrix models.
\newblock {\em Nuclear Phys. B}, 660(3):532--556, 2003.

\bibitem{MR4244340}
Y.~Ameur.
\newblock A localization theorem for the planar {C}oulomb gas in an external
  field.
\newblock {\em Electron. J. Probab.}, 26:Paper No. 46--21, 2021.

\bibitem{AB}
Y.~Ameur and S.-S. Byun.
\newblock Almost-{H}ermitian random matrices and bandlimited point processes.
\newblock {\em Anal. Math. Phys. (to appear), arXiv:2101.03832}, 2021.

\bibitem{ameur2011fluctuations}
Y.~Ameur, H.~Hedenmalm, and N.~Makarov.
\newblock Fluctuations of eigenvalues of random normal matrices.
\newblock {\em Duke Math. J.}, 159(1):31--81, 2011.

\bibitem{MR4030288}
Y.~Ameur, N.-G. Kang, N.~Makarov, and A.~Wennman.
\newblock Scaling limits of random normal matrix processes at singular boundary
  points.
\newblock {\em J. Funct. Anal.}, 278(3):108340, 2020.

\bibitem{ameur2018random}
Y.~Ameur, N.-G. Kang, and S.-M. Seo.
\newblock The random normal matrix model: insertion of a point charge.
\newblock {\em Potential Anal.}, 58(2):331--372, 2023.

\bibitem{MR3280250}
F.~Balogh, M.~Bertola, S.-Y. Lee, and K.~D. T.-R. McLaughlin.
\newblock Strong asymptotics of the orthogonal polynomials with respect to a
  measure supported on the plane.
\newblock {\em Comm. Pure Appl. Math.}, 68(1):112--172, 2015.

\bibitem{MR3668632}
F.~Balogh, T.~Grava, and D.~Merzi.
\newblock Orthogonal polynomials for a class of measures with discrete
  rotational symmetries in the complex plane.
\newblock {\em Constr. Approx.}, 46(1):109--169, 2017.

\bibitem{balogh2015equilibrium}
F.~Balogh and D.~Merzi.
\newblock Equilibrium measures for a class of potentials with discrete
  rotational symmetries.
\newblock {\em Constr. Approx.}, 42(3):399--424, 2015.

\bibitem{berezin2022planar}
S.~Berezin, A.~B.~J. Kuijlaars, and I.~Parra.
\newblock Planar orthogonal polynomials as type {I} multiple orthogonal
  polynomials.
\newblock {\em SIGMA Symmetry Integrability Geom. Methods Appl.}, 19:Paper No.
  020, 2023.

\bibitem{MR3849128}
M.~Bertola, J.~G. Elias~Rebelo, and T.~Grava.
\newblock Painlev\'{e} {IV} critical asymptotics for orthogonal polynomials in
  the complex plane.
\newblock {\em SIGMA Symmetry Integrability Geom. Methods Appl.}, 14:Paper No.
  091, 34, 2018.

\bibitem{MR1917675}
M.~Bertola, B.~Eynard, and J.~Harnad.
\newblock Duality, biorthogonal polynomials and multi-matrix models.
\newblock {\em Comm. Math. Phys.}, 229(1):73--120, 2002.

\bibitem{MR1949138}
P.~Bleher and A.~Its.
\newblock Double scaling limit in the random matrix model: the
  {R}iemann-{H}ilbert approach.
\newblock {\em Comm. Pure Appl. Math.}, 56(4):433--516, 2003.

\bibitem{MR2921180}
P.~M. Bleher and A.~B.~J. Kuijlaars.
\newblock Orthogonal polynomials in the normal matrix model with a cubic
  potential.
\newblock {\em Adv. Math.}, 230(3):1272--1321, 2012.

\bibitem{byun2023planar}
S.-S. Byun.
\newblock Planar equilibrium measure problem in the quadratic fields with a
  point charge.
\newblock {\em preprint arXiv:2301.00324}, 2023.

\bibitem{byun2021universal}
S.-S. Byun and M.~Ebke.
\newblock Universal scaling limits of the symplectic elliptic {G}inibre
  ensembles.
\newblock {\em Random Matrices Theory Appl. (Online) arXiv:2108.05541}, 2021.

\bibitem{byun2022progress}
S.-S. Byun and P.~J. Forrester.
\newblock Progress on the study of the {G}inibre ensembles {I}: {G}in{UE}.
\newblock {\em preprint arXiv:2211.16223}, 2022.

\bibitem{byun2021real}
S.-S. Byun, N.-G. Kang, J.~O. Lee, and J.~Lee.
\newblock Real eigenvalues of elliptic random matrices.
\newblock {\em Int. Math. Res. Not.}, (3):2243--2280, 2023.

\bibitem{byun2022determinantal}
S.-S. Byun and M.~Yang.
\newblock Determinantal {C}oulomb gas ensembles with a class of discrete
  rotational symmetric potentials.
\newblock {\em preprint arXiv:2210.04019}, 2022.

\bibitem{MR3820329}
D.~Chafa\"{\i}, A.~Hardy, and M.~Ma\"{\i}da.
\newblock Concentration for {C}oulomb gases and {C}oulomb transport
  inequalities.
\newblock {\em J. Funct. Anal.}, 275(6):1447--1483, 2018.

\bibitem{claeys2008universality}
T.~Claeys and A.~B.~J. Kuijlaars.
\newblock Universality in unitary random matrix ensembles when the soft edge
  meets the hard edge.
\newblock In {\em Integrable systems and random matrices}, volume 458 of {\em
  Contemp. Math.}, pages 265--279. Amer. Math. Soc., Providence, RI, 2008.

\bibitem{MR2434886}
T.~Claeys, A.~B.~J. Kuijlaars, and M.~Vanlessen.
\newblock Multi-critical unitary random matrix ensembles and the general
  {P}ainlev\'{e} {II} equation.
\newblock {\em Ann. of Math. (2)}, 168(2):601--641, 2008.

\bibitem{criado2022vector}
J.~G. Criado~del Rey and A.~B.~J. Kuijlaars.
\newblock A vector equilibrium problem for symmetrically located point charges
  on a sphere.
\newblock {\em Constr. Approx.}, 55(3):775--827, 2022.

\bibitem{deano2019characteristic}
A.~Dea\~{n}o and N.~Simm.
\newblock Characteristic polynomials of complex random matrices and
  {P}ainlev\'{e} transcendents.
\newblock {\em Int. Math. Res. Not. IMRN}, (1):210--264, 2022.

\bibitem{MR2172690}
P.~Elbau and G.~Felder.
\newblock Density of eigenvalues of random normal matrices.
\newblock {\em Comm. Math. Phys.}, 259(2):433--450, 2005.

\bibitem{MR2881072}
J.~Fischmann, W.~Bruzda, B.~A. Khoruzhenko, H.-J. Sommers, and
  K.~\.{Z}yczkowski.
\newblock Induced {G}inibre ensemble of random matrices and quantum operations.
\newblock {\em J. Phys. A}, 45(7):075203, 31, 2012.

\bibitem{forrester2010log}
P.~J. Forrester.
\newblock {\em Log-gases and {R}andom {M}atrices (LMS-34)}.
\newblock Princeton University Press, Princeton, 2010.

\bibitem{ginibre1965statistical}
J.~Ginibre.
\newblock Statistical ensembles of complex, quaternion, and real matrices.
\newblock {\em J. Math. Phys.}, 6(3):440--449, 1965.

\bibitem{hedenmalm2017planar}
H.~Hedenmalm and A.~Wennman.
\newblock Planar orthogogonal polynomials and boundary universality in the
  random normal matrix model.
\newblock {\em Acta Math.}, 227(2):309--406, 2021.

\bibitem{MR3303173}
A.~B.~J. Kuijlaars and A.~L\'{o}pez-Garc\'{\i}a.
\newblock The normal matrix model with a monomial potential, a vector
  equilibrium problem, and multiple orthogonal polynomials on a star.
\newblock {\em Nonlinearity}, 28(2):347--406, 2015.

\bibitem{MR2078083}
A.~B.~J. Kuijlaars and K.~T.-R. McLaughlin.
\newblock Asymptotic zero behavior of {L}aguerre polynomials with negative
  parameter.
\newblock {\em Constr. Approx.}, 20(4):497--523, 2004.

\bibitem{MR3306308}
A.~B.~J. Kuijlaars and G.~L.~F. Silva.
\newblock S-curves in polynomial external fields.
\newblock {\em J. Approx. Theory}, 191:1--37, 2015.

\bibitem{MR3383811}
A.~B.~J. Kuijlaars and A.~Tovbis.
\newblock The supercritical regime in the normal matrix model with cubic
  potential.
\newblock {\em Adv. Math.}, 283:530--587, 2015.

\bibitem{MR3454377}
S.-Y. Lee and N.~G. Makarov.
\newblock Topology of quadrature domains.
\newblock {\em J. Amer. Math. Soc.}, 29(2):333--369, 2016.

\bibitem{lee2016fine}
S.-Y. Lee and R.~Riser.
\newblock Fine asymptotic behavior for eigenvalues of random normal matrices:
  {E}llipse case.
\newblock {\em J. Math. Phys.}, 57(2):023302, 2016.

\bibitem{MR3670735}
S.-Y. Lee and M.~Yang.
\newblock Discontinuity in the asymptotic behavior of planar orthogonal
  polynomials under a perturbation of the {G}aussian weight.
\newblock {\em Comm. Math. Phys.}, 355(1):303--338, 2017.

\bibitem{MR3962350}
S.-Y. Lee and M.~Yang.
\newblock Planar orthogonal polynomials as {T}ype {II} multiple orthogonal
  polynomials.
\newblock {\em J. Phys. A}, 52(27):275202, 14, 2019.

\bibitem{lee2020strong}
S.-Y. Lee and M.~Yang.
\newblock Strong asymptotics of planar orthogonal polynomials: Gaussian weight
  perturbed by finite number of point charges.
\newblock {\em Comm. Pure Appl. Math. (to appear), arXiv:2003.04401}, 2020.

\bibitem{Lewin22}
M.~Lewin.
\newblock Coulomb and {R}iesz gases: the known and the unknown.
\newblock {\em J. Math. Phys.}, 63(6):Paper No. 061101, 77, 2022.

\bibitem{MR3939592}
A.~Mart\'{\i}nez-Finkelshtein and G.~L.~F. Silva.
\newblock Critical measures for vector energy: asymptotics of non-diagonal
  multiple orthogonal polynomials for a cubic weight.
\newblock {\em Adv. Math.}, 349:246--315, 2019.

\bibitem{olver2010nist}
F.~W. Olver, D.~W. Lozier, R.~F. Boisvert, and C.~W. Clark~(Editors).
\newblock {\em NIST Handbook of Mathematical Functions}.
\newblock Cambridge University Press, Cambridge, 2010.

\bibitem{ST97}
E.~B. Saff and V.~Totik.
\newblock {\em Logarithmic potentials with external fields}, volume 316 of {\em
  Grundlehren der Mathematischen Wissenschaften [Fundamental Principles of
  Mathematical Sciences]}.
\newblock Springer-Verlag, Berlin, 1997.
\newblock Appendix B by Thomas Bloom.

\bibitem{MR1097025}
M.~Sakai.
\newblock Regularity of a boundary having a {S}chwarz function.
\newblock {\em Acta Math.}, 166(3-4):263--297, 1991.

\bibitem{Serfaty}
S.~Serfaty.
\newblock Microscopic description of {L}og and {C}oulomb gases.
\newblock In {\em Random matrices}, volume~26 of {\em IAS/Park City Math.
  Ser.}, pages 341--387. Amer. Math. Soc., Providence, RI, 2019.

\bibitem{MR2116267}
R.~Teodorescu, E.~Bettelheim, O.~Agam, A.~Zabrodin, and P.~Wiegmann.
\newblock Normal random matrix ensemble as a growth problem.
\newblock {\em Nuclear Phys. B}, 704(3):407--444, 2005.

\bibitem{van1990new}
S.~Van~Eijndhoven and J.~Meyers.
\newblock New orthogonality relations for the {H}ermite polynomials and related
  {H}ilbert spaces.
\newblock {\em Journal of Mathematical Analysis and Applications},
  146(1):89--98, 1990.

\bibitem{MR3946715}
C.~Webb and M.~D. Wong.
\newblock On the moments of the characteristic polynomial of a {G}inibre random
  matrix.
\newblock {\em Proc. Lond. Math. Soc. (3)}, 118(5):1017--1056, 2019.

\end{thebibliography}

\end{document}